\documentclass[12pt,amsymb,fullpage]{amsart}
\usepackage{amssymb,amscd,pstricks}

\newtheorem{theorem}{Theorem}[section]
\newtheorem{defn}[theorem]{Definition}

\newtheorem{lemma}[theorem]{Lemma}

\newtheorem{eple}[theorem]{Example}
\newtheorem{rmk}[theorem]{Remarks}
\newtheorem{dsc}[theorem]{Discussion}
\newtheorem{nota}[theorem]{Notation}

\newsavebox{\indbin}
\savebox{\indbin}{\begin{picture}(0,0)
\newlength{\gnu}
\settowidth{\gnu}{$\smile$} \setlength{\unitlength}{.5\gnu}
\put(-1,-.65){$\smile$} \put(-.25,.1){$|$}
\end{picture}}

\newcommand{\be}{\begin{enumerate}}
\newcommand{\bd}{\begin{defn}}
\newcommand{\bt}{\begin{theorem}}
\newcommand{\bl}{\begin{lemma}}
\newcommand{\ee}{\end{enumerate}}
\newcommand{\ed}{\end{defn}}
\newcommand{\et}{\end{theorem}}
\newcommand{\el}{\end{lemma}}

\begin{document}
\title{Equilibria in Electrochemistry and Maximal Rates of Reaction}
\author{Tristram de Piro}
\address{Flat 3, Redesdale House, 85 The Park, Cheltenham, GL50 2RP}
\begin{abstract}
We consider Gibbs' definition of chemical equilibrium and connect it with dynamic equilibrium, in terms of no substance formed. We determine the activity coefficient as a function of temperature and pressure, in reactions with or without interaction of a solvent, incorporating the error terms from Raoult's Law and Henry's Law, if necessary. We compute the maximal reaction paths and apply the results to electrochemistry, using the Nernst equation.
\end{abstract}
\maketitle
\begin{section}{Introduction}
\label{introduction}
\indent This paper is divided into $12$ sections. In section \ref{idealised}, we give some basic definitions, and derive the Nernst equation for the standard cell. We prove some results about the activity coefficient $Q$, assuming an idealised law in the behaviour of the activities and chemical potentials, $\mu_{i}=\mu_{i}^{\circ}+RT ln(a_{i})$, for $1\leq i\leq c$, with $c$ substances, which holds throughout the section. In Lemma \ref{vanhoffhelmholtz}, we use the van't Hoff, Gibbs-Helmholtz equations to find an expression for $\Delta G^{\circ}(T)$ along quasi-chemical equilibrium paths. In Lemma \ref{eqlines}, we use an entropy calculation to find $({\partial G\over \partial \xi})_{T,P}$ and combine the result with Lemma \ref{vanhoffhelmholtz} to calculate the activity coefficient $Q$. In Lemma \ref{implies}, we prove every straight line chemical equilibrium path is a dynamic equilibrium path, partially confirming a speculation of Gibbs. The method of constant $Q$ along a path implying dynamic equilibrium is repeatedly used and generalised later in the paper. The question of the existence of feasible paths for a reaction, given a curve in the temperature/pressure plane, is answered in Lemma \ref{exists} and again later generalised.\\
\indent In Section \ref{errorterms}, we consider ideal solutions and introduce a fixed error term from Raoult's law. The results from section $0$ generalise and in Lemma \ref{rates}, we find the paths of maximal reaction, in the sense of maximising extent $\xi$, implicitly, in terms of temperature and pressure $(T,P)$. We apply the results to electrochemistry in Section \ref{idealelectrochemistry}. In Section \ref{errorterms2}, we consider dilute solutions, adding substance $0$, and consider the definition of $Q$, involving the activity $a_{0}$, obtaining the formula for the activity coefficient in Lemma \ref{proofs}. In Section \ref{henryssolutesidealsolvent}, we consider dilute solutions with interaction of the solvent, in which the solvent is ideal and the solutes obey Henry's law, introducing a new fixed error term in Definition \ref{catalyzers2}, and obtaining the maximal reaction paths in Lemma \ref{henrys9}. In Section \ref{fug}, we introduce new fixed error terms from fugacity, in Definition \ref{catalyzers3}. We alter the conventional definition of $Q$ to incorporate fugacity in the error term and obtain the paths of maximal reaction in Lemma \ref{fugacity9}. We apply the results to electrochemistry in Section \ref{electrochemistry}, in particular the reaction in catalyzers, and give a strategy for improving the efficiency of hydrogen and oxygen production from water in Remarks \ref{strategy}. In Section \ref{nointeraction}, we consider the case when there is no interaction of the solvent, and in Sections \ref{nointeraction} and \ref{fug1}, we derive the main results quickly by altering $Q$ to ignore the activity $a_{0}$. However, it an interesting but difficult exercise to try and derive the results using the definition in Section \ref{errorterms2}. We suggest the results here could be used in maximising ethanol production. We gain apply the results to electrochemistry in Section \ref{final}, in particularly the standard cell. In Section \ref{enthalpy}, we reconsider the assumption that $\Delta H^{\circ}(T)$ is constant, made throughout the paper. We show that by increasing the mass of the mixture, in particularly the amount of solvent, we can make the error involved here disappear in the limit. Finally, in Section \ref{path}, we consider independence and existence of paths.\\
\end{section}

\begin{section}{The Idealised Case}
\label{idealised}
\begin{defn}
\label{constants}
For $c$ substances, we define the Gibbs energy $G(T,P,n_{1},\ldots,n_{c})$ by;\\

$G=U+PV-TS$\\

where $U$ is the internal energy, $P$ is pressure, $V$ is volume, $T$ is temperature and $S$ is entropy. We define the enthalpy $H(T,P,n_{1},\ldots,n_{c})$ by;\\

$H=U+PV=G+TS$\\

see \cite{M} and \cite{F}. We define the Gibbs energy at standard pressure $G^{\circ}(T,n_{1},\ldots,n_{c})=G(T,P^{\circ},n_{1},\ldots,n_{c})$, where $P^{\circ}$ is the standard pressure.  We define the chemical potentials, $1\leq i\leq c$, by;\\

$\mu_{i}(T,P)=({\partial G\over \partial n_{i}})_{T,P,n'}$\\

where $n_{i}$ is the amount of substance $i$ measured in moles, and $T,P,n'$ fixes the pressure, temperature and the amount of every substance except substance $i$.

We consider an electrolyte as a solute in a dilute solution and define the activities $a_{i}$, $1\leq i\leq c$, by;\\

$a_{1}=\gamma_{1}x_{1}\simeq 1$\\

$a_{i}={\gamma_{i}m_{i}\over m^{\circ}}$ $(2\leq i\leq c)$\\

where the molality $m_{i}={n_{i}\over w_{1}}$, and $w_{1}$ is the mass of the solvent, component $1$, $m^{\circ}=1$, $x_{i}={n_{i}\over n}$, $n=\sum_{i=1}^{c}n_{i}$, $\gamma_{i},1\leq i\leq c$, are the activity coefficients, $c_{i}={n_{i}\over V}$, $c^{\circ}=1$, and the activity quotient;\\

$Q=\prod_{i=1}^{c}a_{i}^{\nu_{i}}$\\

where $\nu_{i}$, for $1\leq i\leq c$ are the stoichiometric coefficients.

We have, for a solute in a dilute solution, that $\mu_{i}=\mu_{i}^{\circ}+RTln(a_{i})$, see \cite{M}, noting that $\mu_{i}$ is independent of the amount of substance $n_{i}$, and $\mu_{i}^{\circ}$ in the molality description is equal to $\mu_{i}^{(m)}$, where $m_{i}$ is equal to $m^{\circ}=1$ in a hypothetical solution.\\

We define $\Delta G^{\circ}(T)$ and $\Delta H^{\circ}(T)$ to be the changes in Gibbs energy and enthalpy at standard pressure $P^{\circ}$ and temperature $T$, for $1$ mole of reaction, see \cite{BCD}and \cite{CH}.\\

We define the extent $\xi(T,P)$ of a reaction by;\\

$n_{i,0}+\nu_{i}\xi=n_{i}$\\

where $n_{i,0}=n_{i}(initial)$. We assume that if;\\

$h_{T,P,n_{1,0},\ldots,n_{c,0}}(\xi)=G(T,P,n_{1,0}+\nu_{1}\xi,\ldots,n_{c,0}+\nu_{c}\xi)$\\

then ${dh_{T,P,n_{1,0},\ldots,n_{c,0}}\over d\xi}(\xi)$ is independent of $\{n_{1,0},\ldots,n_{c,0},\}\subset\mathcal{R}_{>0}$ and $\xi\in\mathcal{R}_{\geq 0}$, and define this as $({\partial G\over \partial \xi})|_{T,P}$.\\

We define chemical equilibrium by $({\partial G\over \partial \xi})|_{T,P}=0$, and dynamic equilibrium by a path $\gamma:[0,1]\rightarrow (T,P,n_{1},\ldots,n_{c})$, such that $n_{i}'(t)=pr_{2+i}(\gamma)'(t)=0$, for $1\leq i\leq c$, so that no substance is formed. \\

We define $E(T,P,n_{1},\ldots,n_{c})$ to be the potential in the standard cell and $E^{\circ}(T,n_{1},\ldots,n_{c})=E(T,P^{\circ},n_{1},\ldots,n_{c})$ to be the potential at $P^{\circ}$. We let $F$ denote Faraday's constant, $R=N_{A}k$ the gas constant, where $N_{Av}$ is Avogadro's constant and $k$ is Boltzmann's constant. We have that $N=N_{Av}n$, where $N$ is the number of particles, $n$ is the amount in moles. We define the electric potential $\phi$ by $\overline{E}=-\bigtriangledown(\phi)$, $e$ is the charge on an electron, $z_{i}$ is the valence of an ion.

\end{defn}
\begin{lemma}
\label{differential}

$dG=-SdT+VdP+\sum_{i=1}^{c}\mu_{i}dn_{i}$\\

\end{lemma}

\begin{proof}
We have, by the definition of the chemical potential and the laws of differentials, that;\\

$dG={\partial G\over \partial T}_{P,n}dT+{\partial G\over \partial P}_{T,n}dP+\sum_{i=1}^{c}{\partial G\over \partial n_{i}}_{T,P,n'}dn_{i}$\\

$={\partial G\over \partial T}_{P,n}dT+{\partial G\over \partial P}_{T,n}dP+\sum_{i=1}^{c}\mu_{i}dn_{i}$, $(*)$\\

Fixing $n$, using $G=U+PV-TS$, the first law of thermodynamics, $dU=dQ-PdV$, the definition of entropy, $dQ=TdS$, the product rule for differentials, and $(*)$, we have that;\\

$dU=TdS-PdV$\\

$dG=dU+VdP+PdV-SdT-TdS$\\

$=(TdS-PdV)+VdP+PdV-SdT-TdS$\\

$=-SdT+VdP$\\

$dG={\partial G\over \partial T}_{P,n}dT+{\partial G\over \partial P}_{T,n}dP$ $(**)$\\

so that, from $(**)$, and equating coefficients, ${\partial G\over \partial T}_{P,n}=-S$, ${\partial G\over \partial P}_{T,n}=V$. Substituting into $(*)$, we obtain that;\\

$dG=-SdT+VdP+\sum_{i=1}^{c}\mu_{i}dn_{i}$\\

\end{proof}

\begin{defn}
\label{electrical}
We define electrical chemical equilibrium by $({\partial G\over \partial \xi})_{T,P}=0$ where $G$ is the Gibbs energy function for a charged and uncharged species.

\end{defn}

\begin{lemma}{The Nernst Equation for the Standard Cell}\\
\label{nernst}\\

At electrical chemical equilibrium $(T,P)$ and $(T,P^{\circ})$;\\

$E-E^{\circ}=-{RTln(Q)\over 2F}$\\

See \cite{M}.\\

\end{lemma}
\begin{proof}
For $c$ substances, with $c'$ the number of the charged species,using Definition \ref{constants}, we have that the electrostatic potential energy;\\

$U_{el}= \sum_{i=1}^{c'}\phi(\overline{x}_{i})q_{i}$, where $q_{i}=N_{i}ez_{i}=N_{A}n_{i}ez_{i}$\\

where $\{\overline{x}_{i}:1\leq i\leq c'\}$ are the positions of the charged species, $N_{i}$ is the number of particles at $\overline{x}_{i}$.\\

 We have that;\\

$U=U_{chem}+U_{el}$, so that;\\

$G(T,P,n_{1},\ldots,n_{c})=U+PV-TS$\\

$=U_{chem}+U_{el}+PV-TS$\\

$=U_{el}+G_{chem}$\\

$=\sum_{j=1}^{c}\phi(\overline{x}_{j})q_{j}+G_{chem}$\\

$=\sum_{j=1}^{c}\phi(\overline{x}_{j})N_{A}n_{j}ez_{j}+G_{chem}$\\

so that;\\

$\mu_{i}=({\partial G\over \partial n_{i}})_{T,P}$\\

$=({\partial (\sum_{j=1}^{c}\phi(\overline{x}_{j})N_{A}n_{j}ez_{j}+G_{chem})\over \partial n_{i}})_{T,P}$\\

$=\mu_{i,chem}$, $(c'+1\leq i\leq c)$\\

$=\mu_{i,chem}+{\partial (\phi(\overline{x}_{i})N_{A}n_{i}ez_{i})\over \partial n_{i}}$, $(1\leq i\leq c')$\\

$=\mu_{i,chem}+\phi(\overline{x}_{i})N_{A}ez_{i}$\\

$=\mu_{i,chem}+\phi(\overline{x}_{i})F z_{i}$, $(*)$\\

We consider the standard cell reaction $H_{2}(g)+2AgCl(s)+2e^{-}(R)\rightarrow 2HCl+2Ag(s)+2e^{-}(L)$. At chemical equilibrium, similarly to Lemma \ref{equivalences}, generalised to a collection involving charged species, using $(*)$, we have that;\\

$({\partial G\over \partial \xi})_{T,P}={\sum}_{i=1}^{c}\nu_{i}\mu_{i}$\\

$=2\mu(HCl)+2\mu(Ag)-\mu(H_{2})-2\mu(AgCl)+2\mu(e^{-}(L))-2\mu(e^{-}(R))$\\

$=({\partial G_{chem'}\over \partial \xi})_{T,P}+2\mu(e^{-}(L))-2\mu(e^{-}(R))$\\

$=({\partial G_{chem'}\over \partial \xi})_{T,P}+((2\mu_{chem}(e^{-}(L))-2F\phi(L))-(2\mu_{chem}(e^{-}(L))-2F\phi(R)))$\\

$=({\partial G_{chem'}\over \partial \xi})_{T,P}+2F(\phi(R)-\phi(L))$\\

$=({\partial G_{chem'}\over \partial \xi})_{T,P}+2EF=0$ $(\dag)$ \\

where $G_{chem'}$ is the Gibbs energy restricted to the uncharged species. We have that;\\

$({\partial G_{chem'}\over \partial \xi})_{T,P^{\circ}}=\sum_{i=c'+1}^{c}\nu_{i}\mu_{i}^{\circ}$\\

$=(\Delta G_{chem'}^{\circ}+RTln(Q_{chem'}(T,P^{\circ})))$\\

$=\Delta G_{chem'}^{\circ}$, $(\dag\dag)$\\

using the definition of $Q_{chem'}$ in Definition \ref{constants}, the fact that $\mu_{i}=\mu_{i}^{\circ}+RTln(a_{i})$, $\mu_{i}=\mu_{i}^{\circ}$, for $c'+1\leq i\leq c$, so that  $Q_{chem'}(T,P^{\circ})=1$.\\

From $(\dag),(\dag\dag)$, we obtain;\\

$2E^{\circ}F=-({\partial G_{chem'}\over \partial \xi})_{T,P^{\circ}}$\\

$=-\Delta G_{chem'}^{\circ}$ $(\dag\dag\dag)$\\

Similarly, we have that;\\

$({\partial G_{chem'}\over \partial \xi})_{T,P}=\sum_{i=c'+1}^{c}\nu_{i}\mu_{i}$\\

$=(\Delta G_{chem'}^{\circ}+RTln(Q_{chem'}(T,P)))$, $(\dag\dag\dag\dag)$\\

so from $(\dag),(\dag\dag\dag\dag))$, we obtain that;\\

$2EF=-({\partial G_{chem'}\over \partial \xi})_{T,P}$\\

$=-(\Delta G_{chem'}^{\circ}+RTln(Q_{chem'}(T,P)))$, $(\sharp)$\\

Combining $(\sharp), (\dag\dag\dag)$, we obtain that;\\

$2EF-2E^{\circ}F=-(\Delta G_{chem'}^{\circ}+RTln(Q_{chem'}(T,P)))-(-\Delta G_{chem'}^{\circ})$\\

$=-RTln(Q_{chem'}(T,P))$\\

so that;\\

$E-E^{\circ}=-{RTln(Q_{chem'}(T,P))\over 2F}$\\

\end{proof}

\begin{lemma}
\label{gibbs}

For the energy function $G$ involving only an uncharged species;\\

$({\partial G\over \partial \xi})_{T,P}=\Delta G^{\circ}+RTln(Q)$\\

see \cite{M}.\\

\end{lemma}
\begin{proof}
By Lemma \ref{equivalences}, we have that;\\

$({\partial G\over \partial \xi})_{T,P}=\sum_{i=1}^{c}\nu_{i}\mu_{i}$\\

$\Delta G^{\circ}=\sum_{i=1}^{c}\nu_{i}\mu_{i}^{\circ}$, $(*)$\\

Using $(*)$, the fact that $\mu_{i}=\mu_{i}^{\circ}+RTln(a_{i})$, and Definition \ref{constants}, we have that;\\

$({\partial G\over \partial \xi})_{T,P}-\Delta G^{\circ}=\sum_{i=1}^{c}\nu_{i}(\mu_{i}-\mu_{i}^{\circ})$\\

$=\sum_{i=1}^{c}\nu_{i}(\mu_{i}^{\circ}+RTln(a_{i})-\mu_{i}^{\circ})$\\

$=\sum_{i=1}^{c}\nu_{i}RTln(a_{i})$\\

$=RTln(\prod_{i=1}^{c}a_{i}^{\nu_{i}})=RTln(Q)$\\

\end{proof}

\begin{lemma}
\label{delta}

At electrical chemical equilibrium $(T,P)$ and $(T,P^{\circ})$, and chemical equilibrium $(T,P)$;\\

$\Delta G^{\circ}=2F(E-E^{0})$\\

\end{lemma}
\begin{proof}
By Lemma \ref{nernst}, we have that;\\

$E-E^{\circ}=-{RTln(Q)\over 2F}$, $(*)$\\

and, by Lemma \ref{gibbs}, we have that;\\

$0=({\partial G\over \partial \xi})_{T,P}=\Delta G^{\circ}+RTln(Q)$, $(**)$\\

Rearranging $(*),(**)$, we obtain the result.\\

\end{proof}

\begin{lemma}
\label{equivalences}

We have that;\\

$({\partial G\over \partial \xi})_{T,P}=\sum_{i=1}^{c}\nu_{i}\mu_{i}$\\

$\Delta G^{\circ}=\sum_{i=1}^{c}\nu_{i}\mu_{i}^{\circ}$\\

At chemical equilibrium $T,P$, $({\partial G\over \partial \xi})_{T,P}=0$ and at $T,P^{0}$, $\Delta G^{\circ}=0$.\\

If chemical and electrical chemical equilibrium exists at $(T,P^{\circ})$ and $(T,P)$, $Q(T,P)=1$. and $E=E^{\circ}$. Conversely, if $Q(T,P)=1$ and chemical equilibrium exists at $(T,P^{\circ})$ then chemical equilibrium exists at $(T,P)$.\\

Chemical equilibrium exists at $(T,P)$ iff $Q(T,P)=e^{-\Delta G^{\circ}\over RT}$\\

We always have that $Q(T,P^{\circ})=1$.\\

\end{lemma}

\begin{proof}
For the first claim, we have, using the definition of $\xi$, that;\\

$dn_{i}=\nu_{i}d\xi$, $(1\leq i\leq c)$, $(*)$\\

By Lemma \ref{differential}, fixing $T$ and $P$, and using $(*)$, we have that;\\

$dG=\sum_{i=1}^{c}\mu_{i}d n_{i}$\\

$=(\sum_{i=1}^{c}\mu_{i}\nu_{i})d\xi$, $(\dag)$\\

so that;\\

$({\partial G\over \partial \xi})_{T,P}=\sum_{i=1}^{c}\mu_{i}\nu_{i}$, $(\dag\dag)$\\

The second claim from the first, as;\\

$\Delta G^{\circ}(T)=\int_{0}^{1}({\partial G\over \partial \xi})_{T,P^{\circ}}$\\

$=\int_{0}^{1}(\sum_{i=1}^{c}\nu_{i}\mu_{i}^{\circ}(T))d\xi$\\

$=\sum_{i=1}^{c}\nu_{i}\mu_{i}^{\circ}(T)\int_{0}^{1}d\xi$\\

$=\sum_{i=1}^{c}\nu_{i}\mu_{i}^{\circ}(T)$\\

noting that $({\partial G\over \partial \xi})_{T,P^{\circ}}$ doesn't vary with $\xi$. For the third claim, at chemical equilibrium, $T,P$, noting again that $({\partial G\over \partial \xi})_{T,P}$ doesn't vary with $\xi$, and using $(\dag,\dag\dag)$, we have that;\\

$dG=({\partial G\over \partial \xi})_{T,P}=0$, (independently of $\xi$) $(\dag\dag\dag)$\\

At chemical equilibrium $T,P^{\circ}$, using the first and second claims, and $(\dag\dag\dag)$, we have that;\\

$dG=({\partial G\over \partial \xi})_{T,P^{\circ}}$\\

$=\sum_{i=1}^{c}\nu_{i}\mu_{i}^{\circ}$\\

$=\Delta G^{\circ}=0$\\

For the second to last claim, and the first direction, we have, by Lemma \ref{gibbs}, that $RTln(Q)=0$, so that $Q=1$, and, by Lemma \ref{nernst}, that $E-E^{\circ}=-{RTln(Q)\over 2F}=0$. For the converse, we have by Lemma \ref{gibbs}, using the fact that $Q(T,P)=1$;\\

$({\partial G\over \partial \xi})_{T,P}=\Delta G^{\circ}$\\

and, if chemical equilibrium exists at $(T,P^{\circ})$, then, as $Q(T,P^{\circ})=1$ we have that;\\

$({\partial G\over \partial \xi})_{T,P^{\circ}}=\Delta G^{\circ}+RTln(Q)=\Delta G^{\circ}=0$\\

so that $({\partial G\over \partial \xi})_{T,P}=0$\\

For the penultimate claim, in one direction, use Lemma \ref{gibbs}, together with the fact that $({\partial G\over \partial \xi})_{T,P}=0$ and rearrange, the converse is also clear, applying $ln$.\\

For the final claim, we have, by the definition of activities, that;\\

$\mu_{i}=\mu_{i}^{\circ}+RTln(a_{i})$\\

so that $a_{i}(T,P^{\circ})=1$. Now use the definition of $Q$ in Definition \ref{constants}.\\

\end{proof}
\begin{lemma}{van't Hoff,Gibbs-Helmholtz}\\
\label{vanhoffhelmholtz}
\\
Along a chemical equilibrium path, we have that;\\

$ln({Q(T_{2})\over Q(T_{1})})={1\over R}\int_{T_{1}}^{T_{2}}{\Delta H^{\circ}\over T^{2}}dT$\\

${\Delta G^{\circ}(T_{2})\over T_{2}}-{\Delta G^{\circ}(T_{1})\over T_{1}}=-\int_{T_{1}}^{T_{2}}{\Delta H^{\circ}\over T^{2}}dT$\\

In particularly, if $\Delta H^{\circ}$ is temperature independent;\\

$ln({Q(T_{2})\over Q(T_{1})})=-{\Delta H^{\circ}\over R}({1\over T_{2}}-{1\over T_{1}})$\\

${\Delta G^{\circ}(T_{2})\over T_{2}}-{\Delta G^{\circ}(T_{1})\over T_{1}}=\Delta H^{\circ}({1\over T_{2}}-{1\over T_{1}})$\\

$\Delta G^{\circ}(T_{1})={T_{1}\over T_{2}}\Delta G^{\circ}(T_{2})-({T_{1}\over T_{2}}-1)\Delta H^{\circ}$\\

For $c\in\mathcal{R}$, Let $D_{c}$ intersect the line $P=P^{\circ}$ at $(T_{1},P^{\circ})$, then, for $(T_{2},P)\in D_{c}$, we have that;\\

$Q(T_{2},P)=e^{\Delta G^{\circ}(T_{1})-\Delta G^{\circ}(T_{2})\over RT_{2}}$ $(\dag\dag\dag)$\\

$c=\Delta G^{\circ}(T_{1})$\\

$ln({Q(T_{2})\over Q(T_{1})})=ln(Q(T_{2}))={1\over R}\int_{T_{1}}^{T_{2}}{\Delta H^{\circ}-c\over T^{2}}dT$\\

${\Delta G^{\circ}(T_{2})-\Delta G^{\circ}(T_{1})\over T_{2}}=-\int_{T_{1}}^{T_{2}}{\Delta H^{\circ}-c\over T^{2}}dT$\\

and if $\Delta H^{\circ}$ is temperature independent;\\

$ln({Q(T_{2})\over Q(T_{1})})=ln(Q(T_{2}))=-({\Delta H^{\circ}-c\over R})({1\over T_{2}}-{1\over T_{1}})$ $(\dag\dag)$\\

${\Delta G^{\circ}(T_{2})-\Delta G^{\circ}(T_{1})\over T_{2}}=(\Delta H^{\circ}-c)({1\over T_{2}}-{1\over T_{1}})$ $(\dag)$\\

$\Delta G^{\circ}(T_{1})={T_{1}\over T_{2}}\Delta G^{\circ}(T_{2})-\Delta H^{\circ}({T_{1}\over T_{2}}-1)$ $(\dag\dag\dag\dag)$\\

See also \cite{M}.\\

\end{lemma}
\begin{proof}
By Lemma \ref{equivalences}, we have that;\\

$\Delta G^{\circ}=\sum_{i=1}^{c}\nu_{i}\mu_{i}^{\circ}$\\

so that differentiating with respect to $T$;\\

${d(\Delta G^{\circ})\over dT}=\sum_{i=1}^{c}\nu_{i}{d\mu_{i}^{\circ}\over dT}$\\

$=\sum_{i=1}^{c}\nu_{i}({\partial \mu_{i}^{\circ}\over \partial T})_{P,n}$\\

By Euler reciprocity, we have that;\\

$({\partial \mu_{i}^{\circ}\over \partial T})_{P,n}=-({\partial S^{\circ}\over \partial n_{i}})_{T,P,n'}=-\overline{S}^{\circ}_{i}$\\

so that, noting $\overline{S}^{\circ}_{i}$ is independent of $n_{i}$, so we can replace $\overline{S}^{\circ}_{i}$ by $\overline{S}^{\circ}_{m,i}$, the absolute molar entropy of substance $i$, and using thermodynamic arguments;\\

${d(\Delta G^{\circ})\over dT}=-\sum_{i=1}^{c}\nu_{i}\overline{S}^{\circ}_{i}=$\\

$=-\sum_{i=1}^{c}\nu_{i}\overline{S}^{\circ}_{m,i}$\\

$=-\Delta S^{\circ}$ $(*)$\\

Using the product rule, $(*)$ and the definition of enthalpy, we have that;\\

${d\over dT}({\Delta G^{\circ}\over T})={1\over T}{d(\Delta G^{\circ})\over dT}-{1\over T^{2}}\Delta G^{\circ}$\\

$=-{\Delta S^{\circ}\over T}-{\Delta G^{\circ}\over T^{2}}$\\

$=-{\Delta (ST+G)^{\circ}\over T^{2}}$\\

$=-{\Delta H^{\circ}\over T^{2}}$ $(**)$\\

By Lemma \ref{equivalences}, along a chemical equilibrium path, we have that $Q=e^{-\Delta G^{\circ}\over RT}$, so that $ln(Q)={-\Delta G^{\circ}\over RT}$. It follows from $(**)$ that;\\

${d ln(Q)\over dT}={d\over dT}({-\Delta G^{\circ}\over RT})$\\

$={\Delta H^{\circ}\over RT^{2}}$\\

It follows, integrating between $T_{1}$ and $T_{2}$, that;\\

$ln({Q(T_{2})\over Q(T_{1})})=ln(Q)(T_{2})-ln(Q)(T_{1})$\\

$={-\Delta G^{\circ}(T_{2})\over RT_{2}}+{\Delta G^{\circ}(T_{1})\over RT_{1}}$\\

$=\int_{T_{1}}^{T_{2}}{d ln(Q)\over dT}dT$\\

$={1\over R}\int_{T_{1}}^{T_{2}}{\Delta H^{\circ}\over T^{2}}$ $(P)$\\

so that, rearranging, we obtain the first claim. Using the fact, by Lemma \ref{gibbs}, that;\\

$ln(Q(T_{2}))=-{\Delta G^{\circ}(T_{2})\over RT_{2}}$\\

$ln(Q(T_{1}))=-{\Delta G^{\circ}(T_{1})\over RT_{1}}$\\

we obtain, substituting into $(P)$, canceling $R$, and performing the integration, if $\Delta H^{\circ}$ is temperature independent, that;\\

${\Delta G^{\circ}(T_{2})\over T_{2}}-{\Delta G^{\circ}(T_{1})\over T_{1}}=-\int_{T_{1}}^{T_{2}}{\Delta H^{\circ}\over T^{2}}=\Delta H^{\circ}({1\over T_{2}}-{1\over T_{1}})$ $(Q)$\\

For the fifth claim, rearrange $(Q)$. If $D_{c}$ intersects the line $P=P^{\circ}$ at $(T_{1},P^{\circ})$, for the sixth $(\dag\dag\dag)$ and seventh claims, we have, using Lemma \ref{gibbs} and the fact from Lemma \ref{equivalences} that $Q(T_{1},P^{\circ})=1$;\\

$({\partial G\over \partial \xi})_{T_{2},P}=\Delta G^{\circ}(T_{2})+RT_{2}ln(Q(T_{2},P))$\\

$=({\partial G\over \partial \xi})_{T_{1},P^{\circ}}$\\

$=\Delta G^{\circ}(T_{1})+RT_{1}ln(Q(T_{1},P^{\circ}))$\\

$=\Delta G^{\circ}(T_{1})=c$\\

so that, again rearranging, we obtain the result. Along $D_{c}$, we have, using Lemma \ref{gibbs}, that;\\

$ln(Q)={c-\Delta G^{\circ}\over RT}$\\

so that, using the first part;\\

${d ln(Q)\over dT}={d\over dT}({c-\Delta G^{\circ}\over RT})$\\

$={-c\over RT^{2}}+{d\over dT}({-\Delta G^{\circ}\over RT})$\\

$={\Delta H^{\circ}-c\over RT^{2}}$\\

so that, performing the integration, using the fact that $Q(T_{1},P^{\circ})=1$;\\

$ln(Q(T_{2}))-ln(Q(T_{1}))=ln(Q(T_{2}))={1\over R}\int_{T_{1}}^{T_{2}}{\Delta H^{\circ}-c\over T^{2}}dT$\\

We have that, by Lemma \ref{gibbs};\\

$ln (Q(T_{2}))={c-\Delta G^{\circ}(T_{2})\over RT_{2}}$\\

$ln(Q(T_{1}))=0$\\

so that;\\

$ln (Q(T_{2}))=ln (Q(T_{2}))-ln(Q(T_{1}))$\\

$={c-\Delta G^{\circ}(T_{2})\over RT_{2}}$\\

$={\Delta G^{\circ}(T_{1})-\Delta G^{\circ}(T_{2})\over RT_{2}}$\\

$={1\over R}\int_{T_{1}}^{T_{2}}{\Delta H^{\circ}-c\over T^{2}}dT$\\

$={-1\over R}(\Delta H^{\circ}-c)({1\over T_{2}}-{1\over T_{1}})$\\

and rearranging;\\

${\Delta G^{\circ}(T_{2})-\Delta G^{\circ}(T_{1})\over T_{2}}=(\Delta H^{\circ}-c)({1\over T_{2}}-{1\over T_{1}})$\\

$=(\Delta H^{\circ}-\Delta G^{\circ}(T_{1}))({1\over T_{2}}-{1\over T_{1}})$\\

so that, rearranging again;\\

$\Delta G^{\circ}(T_{1})({1\over T_{1}}+{1\over T_{2}}-{1\over T_{2}})={\Delta G^{\circ}(T_{1})\over T_{1}}$\\

$={\Delta G^{\circ}(T_{2})\over T_{2}}-\Delta H^{\circ}({1\over T_{2}}-{1\over T_{1}})$\\

to obtain;\\

$\Delta G^{\circ}(T_{1})={T_{1}\over T_{2}}\Delta G^{\circ}(T_{2})-\Delta H^{\circ}({T_{1}\over T_{2}}-1)$\\

\end{proof}

\begin{lemma}
\label{eqlines}
If there exists a component $D_{c}$, $c\in\mathcal{R}$, which projects onto a closed bounded subinterval $I$ of the line $P=P^{\circ}$, not containing $0$, and intersects $P=P^{\circ}$ at $(T_{1},P^{\circ})$, with $T_{1}>0$, then, for $T_{2}\in I$, $\Delta G^{\circ}$ is linear, with;\\

$\Delta G^{\circ}(T_{2})=T_{2}({(\Delta G^{\circ}(T_{1})-\Delta H^{\circ})\over T_{1}})+\Delta H^{\circ}$\\

for $T_{2}\in I$. If $\epsilon\neq 0$, have that;\\

$({dG\over d\xi})_{T,P}=\lambda+\epsilon ln(P)+\beta T$\\

where $\{\lambda,\epsilon,\beta\}\subset \mathcal{R}$ and $\{\beta,\epsilon\}$ can be effectively determined, and we have that the activity coefficient is given by;\\

$Q(T_{2},P')=e^{{\epsilon ln({P'\over P'^{\circ}})\over RT_{2}}}$\\

and the dynamic equilibrium paths are given by;\\

$({P'\over P'^{\circ}})^{\epsilon\over RT_{2}}=c$\\

for $c\in\mathcal{R}_{\geq 0}$, see Definition \ref{constants}, while the quasi-chemical equilibrium paths are given by;\\

$\lambda+\epsilon ln(P')+\beta T_{2}=c$\\

for $c\in\mathcal{R}$.\\

If $\epsilon=0$;\\

$({dG\over d\xi})_{T,P}=\lambda+\beta T+\sigma ln(T)$\\

where $\{\lambda,\beta,\sigma\}\subset \mathcal{R}$, and $\{\beta,\sigma\}$ can be effectively determined. For every $T_{1}>0$, there exists a straight line feasible chemical path $\gamma$ with $pr_{12}(\gamma)\subset (T=T_{1})$, which is both a dynamic and quasi-chemical equilibrium path, and $Q(T,P)=1$.\\

\end{lemma}
\begin{proof}
For the first claim, by Lemma \ref{vanhoffhelmholtz}, we have that;\\

$\Delta G^{\circ}(T_{2})={T_{2}\over T_{1}}\Delta G^{\circ}(T_{1})-\Delta H^{\circ}({T_{2}\over T_{1}}-1)$\\

$=T_{2}({(\Delta G^{\circ}(T_{1})-\Delta H^{\circ})\over T_{1}})+\Delta H^{\circ}$\\

For the next claim, by Lemma \ref{equivalences} and the proof of Lemma \ref{vanhoffhelmholtz}, we have that;\\

$({\partial ({dG\over d\xi})_{T,P}\over \partial T})_{P}=({\partial (\sum_{i=1}^{c} \nu_{i}\mu_{i})\over \partial T})_{P}$\\

$=\sum_{i=1}^{c} \nu_{i}({\partial \mu_{i}\over \partial T})_{P}$\\

$=\sum_{i=1}^{c} \nu_{i}({\partial \mu_{i}\over \partial T})_{P,n}$\\

$=-\sum_{i=1}^{c} \nu_{i}\overline{S}_{m,i}$ $(*)$\\

To compute $\overline{S}_{m,i}$, we have by the first law of thermodynamics;\\

$dQ=dU+dL=dU+pdV$\\

where $L$ is the work done by the system, see \cite{F}. We can assume that the liquid mixture is in thermal equilibrium with a mixture of ideal gases in the vapour phase, and using the ideal gas law, the definition of temperature for ideal gases, obtain the calculation of internal energy for the mixture;\\

$U(T,P,n_{1},\ldots,n_{c})=\sum_{i=1}^{c}({3\over 2}N_{A}n_{i}kT-N_{A}n_{i}m_{i}\rho_{i})$\\

where $m_{i}$ is the molecular mass of species $i$, $\rho_{i}$ is the specific latent heat of evaporation of species $i$, which we assume is independent of temperature $T$. By a result in \cite{dep1}, using the fact that entropy difference is independent of path, see \cite{P}, we have that $Q$ is independent of $P$. We then have;\\

$dU=\sum_{i=1}^{c}{3\over 2}N_{A}kTdn_{i}+\sum_{i=1}^{c}{3\over 2}N_{A}kn_{i}dT-\sum_{i=1}^{c}N_{A}m_{i}\rho_{i}dn_{i}$\\

$dQ=\sum_{i=1}^{c}{3\over 2}N_{A}kTdn_{i}+\sum_{i=1}^{c}{3\over 2}N_{A}kn_{i}dT-\sum_{i=1}^{c}N_{A}m_{i}\rho_{i}dn_{i}+dL$\\

${dQ\over T}=\sum_{i=1}^{c}{3\over 2}N_{A}kdn_{i}+\sum_{i=1}^{c}{3\over 2}N_{A}kn_{i}{dT\over T}-\sum_{i=1}^{c}N_{A}m_{i}\rho_{i}{dn_{i}\over T}+{g(T,\overline{n})dT\over T}$\\

$+\sum_{i=1}^{c}h_{i}(T,\overline{n}){dn_{i}\over T}$\\

$({dQ\over T})_{n',T,P}={3\over 2}N_{A}kdn_{i}-N_{A}m_{i}\rho_{i}{dn_{i}\over T}+h_{i}(T,\overline{n}){dn_{i}\over T}$\\

It follows that;\\

$\overline{S}_{m,i}=\int_{\Delta n_{i}=1}({dQ\over T})_{n',T,P}={3\over 2}N_{A}k-{N_{A}m_{i}\rho_{i}\over T}+{k_{i}(T)\over T}$ $(**)$\\

So that, from $(*)$;\\

$({\partial ({dG\over d\xi})_{T,P}\over \partial T})_{P}=-\sum_{i=1}^{c} \nu_{i}({3\over 2}N_{A}k-{N_{A}m_{i}\rho_{i}\over T})-\sum_{i=1}^{c}\nu_{i}{k_{i}(T)\over T}$\\

$=-{3\over 2}N_{A}k(\sum_{i=1}^{c} \nu_{i})+{N_{A}\over T}\sum_{i=1}^{c}\nu_{i}\mu_{i}\rho_{i}-\sum_{i=1}^{c}\nu_{i}{k_{i}(T)\over T}$\\

$=-{3\over 2}N_{A}k(\sum_{i=1}^{c} \nu_{i})+{N_{A}\over T}\sum_{i=1}^{c}\nu_{i}\mu_{i}\rho_{i}-{G(T)\over T}$ $(***)$\\

From $(***)$, which is uniform $P$, we see that $({dG\over d\xi})_{T,P}$ is of the form $\alpha(P)+\beta T+\gamma ln(T)-\int {G(T)\over T}$, $(B)$, where $\{\beta,\gamma\}\subset\mathcal{R}$, and, assuming that $({dG\over d\xi})_{T,P}$ is differentiable, $\alpha \in C^{1}(\mathcal{R})$. By a similar calculation, we have that;\\

$({\partial ({dG\over d\xi})_{T,P}\over \partial P})_{T}=({\partial (\sum_{i=1}^{c} \nu_{i}\mu_{i})\over \partial P})_{T}$\\

$=\sum_{i=1}^{c} \nu_{i}({\partial \mu_{i}\over \partial P})_{T}$\\

$=\sum_{i=1}^{c} \nu_{i}({\partial \mu_{i}\over \partial P})_{T,n}$\\

$=\sum_{i=1}^{c} \nu_{i}({\partial V\over \partial n_{i}})_{T,P,n'}$\\

$=\sum_{i=1}^{c} \nu_{i}\overline{V}_{i}$\\

$=\sum_{i=1}^{c}\nu_{i}{N_{A}m_{i}\over \kappa_{i}(T,P)}$ $(A)$\\

where $\kappa_{i}$ is the density of substance $i$. We also have that;\\

$P(\sum_{i=1}^{c}\nu_{i}{N_{A}m_{i}\over \kappa_{i}(T,P)})=P(\sum_{i=1}^{c} \nu_{i}\overline{V}_{i})=G(T)$ $(dL=PdV)$ $(C)$\\

and from $(A),(B),(C)$;\\

$P({\partial ({dG\over d\xi})_{T,P}\over \partial P})_{T}=G(T)$\\

$=P\alpha'(P)$\\

so that $G(T)=\epsilon$, $\alpha(P)=\lambda+\epsilon ln(P)$\\

$({dG\over d\xi})_{T,P}$ is of the form;\\

$\alpha(P)+\beta T+\gamma ln(T)-\int {G(T)\over T}$\\

$=\lambda+\epsilon ln(P)+\beta T+\gamma ln(T)-\epsilon ln(T)$\\

$=\lambda+\epsilon ln(P)+\beta T+\sigma ln(T)$ $(D)$\\

where $\sigma=\gamma-\epsilon$, $\{\beta,\epsilon,\lambda,\sigma\}\subset\mathcal{R}$.\\

If $\epsilon=0$, then $({dG\over d\xi})_{T,P}$ is independent of $P$, and the components $D_{c}$ are all straight line paths. In this case, if $D_{c}$ intersects the line $P=P^{\circ}$ at $(T_{1},P^{\circ})$, then, for all $P>0$;\\

$c=\Delta G^{\circ}(T_{1})+RT_{1}ln(Q(T_{1},P)$\\

$=\Delta G^{\circ}(T_{1})$\\

implies that $RT_{1}ln(Q(T_{1},P)=0$, so that $Q(T_{1},P)=1$ and, by Lemmas \ref{implies} and \ref{exists}, there exists a straight line feasible chemical path $\gamma$ with $pr_{12}(\gamma)\subset (T=T_{1})$, which is both a dynamic and quasi-chemical equilibrium path. From $(D)$, we have that;\\

$({dG\over d\xi})_{T,P}=\lambda+\beta T+\sigma ln(T)$\\

We have that, for any $T_{1}>0$, we can find $c\in\mathcal{R}$ with $\lambda+\beta T_{1}+\sigma ln(T_{1})=c$, so that $D_{c}$ defines a component straight line path passing through $(T_{1},P^{\circ})$. Then we can apply the previous result.\\

If $\epsilon\neq 0$, for any $c\in\mathcal{R}$, we can solve the equation;\\

$\lambda+\epsilon ln(P)+\beta T+\sigma ln(T)=c$\\

for any given $T>0$ and an appropriate choice of $P(T)$. In particularly, there exists a component $D_{c}$ projecting onto the line $P=P^{0}$ and, by the first part, $\Delta G^{\circ}$ is linear, with;\\

$\Delta G^{\circ}(T_{2})=T_{2}({(\Delta G^{\circ}(T_{1})-\Delta H^{\circ})\over T_{1}})+\Delta H^{\circ}$\\

for an intersection at $(T_{1},P^{\circ})$. We also have, using $(D)$, that;\\

$\Delta G^{\circ}(T_{2})=({\partial G\over \partial \xi})_{T,P}(T_{2},P^{\circ})$\\

$=\lambda+\epsilon ln(P^{\circ})+\beta T_{2}+\sigma ln(T_{2})$\\

so that, equating coefficients;\\

$\sigma=0$\\

$\lambda+\epsilon ln(P^{\circ})=\Delta H^{\circ}$\\

$\beta={(\Delta G^{\circ}(T_{1})-\Delta H^{\circ})\over T_{1}}$\\

$\Delta G^{\circ}(T_{1})=\beta T_{1}+\Delta H^{\circ}$\\

$\Delta G^{\circ}(T_{2})=\beta T_{2}+\Delta H^{\circ}$\\

We can then, using Lemma \ref{gibbs}, obtain a formula for the activity coefficient;\\

$Q(T_{2},P')=e^{{(({\partial G\over \partial \xi})_{T,P}|_{T_{2},P'}-\Delta G^{\circ}(T_{2}))\over RT_{2}}}$\\

$e^{{(\Delta H^{\circ}-\epsilon ln(P'^{\circ})+\epsilon ln(P')+\beta T_{2}-(\beta T_{2}+\Delta H^{\circ}))\over RT_{2}}}$\\

$=e^{{\epsilon ln({P'\over P'^{\circ}})\over RT_{2}}}$ $(W)$\\

as required. The claim about the coefficients being determined is clear from the proof. The determination of the dynamical and quasi-chemical equilibrium lines, see Definitions \ref{constants} and Lemma \ref{feasible}, follows from a simple rearrangement of the formulas $Q(T_{1},P')=c$, for some $c\in\mathcal{R}_{\geq 0}$, using $(W)$, and $({dG\over d\xi})_{T,P}=c$, for some $c\in \mathcal{R}$, using $(D)$, with $\sigma=0$.

\end{proof}
\begin{lemma}
\label{electrical}
With notation as in Lemma \ref{eqlines}, if $\epsilon=0$, then if either;\\

$(i)$. $\beta>0,\sigma>0$\\

$(ii)$. $\beta<0,\sigma<0$\\

$(iii)$. $\beta>0,\sigma<0,\lambda-\sigma+\sigma ln({-\sigma\over \beta})\leq 0$\\

$(iv)$. $\beta<0,\sigma>0, \lambda-\sigma+\sigma ln({-\sigma\over \beta})\geq 0$\\

we can always find $T_{0}>0$, defined as the solution to $\lambda+\beta T+\sigma ln(T)=0$, such that $T=T_{0}$ defines a chemical equilibrium line.

\end{lemma}
\begin{proof}
By the proof of Lemma \ref{eqlines}, if $T_{0}$ is a solution to;\\

$\lambda+\beta T+\sigma ln(T)=0$ $(*)$\\

then $({\partial G\over \partial \xi})_{T,P}|_{T=T_{0}}=0$\\

so that $T=T_{0}$ defines a chemical equilibrium line. By considering limits at $\infty$ and noting that the derivative $\beta+{\sigma\over T}$ of $\lambda+\beta T+\sigma ln(T)$, is of a fixed sign in cases $(i),(ii)$, so that $\lambda+\beta T+\sigma ln(T)$ is monotonic, we can see there does exist a unique solution $T_{0}$ in cases $(i),(ii)$. In cases $(iii),(iv)$, computing limits at $\infty$ again, and noting that $\beta+{\sigma\over T}$ is increasing/decreasing, we have that, if the minimum/maximum $T_{1}$ respectively of $\lambda+\beta T+\sigma ln(T)$, given by the solution to;\\

$\beta+{\sigma\over T}=0$\\

so that $T_{1}={-\sigma\over \beta}$, then, in case $(iii)$, if;\\

$\lambda+\beta(T_{1})+\sigma ln(T_{1})<0$\\

iff $\lambda-\sigma+\sigma ln({-\sigma\over \beta})<0$\\

there exist two possible solutions $T_{0}$, with a unique solution if equality holds. Similarly, then, in case $(iv)$, if;\\

$\lambda-\sigma+\sigma ln({-\sigma\over \beta})>0$\\

there exist at least two possible solutions $T_{0}$, with again a unique solution if equality holds.

\end{proof}

\begin{lemma}
\label{allt}
If $\epsilon=0$, we have, for all $T_{1}>0$, that;\\

$({\partial G\over \partial \xi})_{T,P}|_{(T_{1},P_{1})}=({\partial G\over \partial \xi})_{T,P}|_{(T_{1},P_{1}^{\circ})}$\\

iff;\\

$E(T_{1},P_{1})=E(T_{1},P_{1}^{\circ})=E^{\circ}(T_{1})$\\

where $G$ is the Gibbs energy function for the charged and uncharged species.
\end{lemma}
\begin{proof}
By $(\dag)$ of Lemma \ref{nernst}, we have that;\\

$({\partial G\over \partial \xi})_{T,P}=({\partial G_{chem'}\over \partial \xi})_{T,P}+2EF$ $(*)$\\

By Lemma \ref{eqlines}, we have that $({\partial G_{chem'}\over \partial \xi})_{T,P}$ is independent of $P$, in particularly, we have that;\\

$({\partial G_{chem'}\over \partial \xi})_{T_{1},P_{1}}=({\partial G_{chem'}\over \partial \xi})_{T_{1},P_{1}^{\circ}}$ $(**)$\\

so that, combining $(*),(**)$, we obtain the result.
\end{proof}

\begin{lemma}
\label{alltelectrical}
We have, for all $T_{1}>0,P_{1}>0$, that;\\

$2F(E(T_{1},P_{1})-E^{\circ}(T_{1}))=({\partial G\over \partial \xi})_{T,P}|_{(T_{1},P_{1})}-({\partial G\over \partial \xi})_{T,P}|_{(T_{1},P_{1}^{\circ})}-RT_{1}ln(Q(T_{1},P_{1}))$\\

\end{lemma}
\begin{proof}
Following the proof of Lemma \ref{nernst}, we have that;\\

$({\partial G\over \partial \xi})_{T,P}|_{T_{1},P_{1}}=({\partial G_{chem'}\over \partial \xi})_{T,P}|_{T_{1},P_{1}}+2E(T_{1},P_{1})F$ $(*)$\\

$2E^{\circ}(T_{1})F=({\partial G\over \partial \xi})_{T,P}|_{T_{1},P_{1}^{\circ}}-({\partial G_{chem'}\over \partial \xi})_{T,P}|_{T_{1},P_{1}^{\circ}}$\\

$=({\partial G\over \partial \xi})_{T,P}|_{T_{1},P_{1}^{\circ}}-\Delta G^{\circ}_{chem'}(T_{1})$ $(**)$\\

so from $(*),(**)$ and Lemma \ref{gibbs};\\

$2E(T_{1},P_{1})F=({\partial G\over \partial \xi})_{T,P}|_{T_{1},P_{1}}-({\partial G_{chem'}\over \partial \xi})_{T,P}|_{T_{1},P_{1}}$\\

$=({\partial G\over \partial \xi})_{T,P}|_{T_{1},P_{1}}-(\Delta G^{\circ}_{chem'}(T_{1})+RT_{1}ln(Q_{chem'}(T_{1},P_{1})))$\\

$2E(T_{1},P_{1})F-2E^{\circ}(T_{1})F=({\partial G\over \partial \xi})_{T,P}|_{T_{1},P_{1}}-(\Delta G^{\circ}_{chem'}(T_{1})+RT_{1}ln(Q_{chem'}(T_{1},P_{1})))$\\

$-(({\partial G\over \partial \xi})_{T,P}|_{T_{1},P_{1}^{\circ}}-\Delta G^{\circ}_{chem'}(T_{1}))$\\

$=({\partial G\over \partial \xi})_{T,P}|_{T_{1},P_{1}}-(({\partial G\over \partial \xi})_{T,P}|_{T_{1},P_{1}^{\circ}}-RT_{1}ln(Q_{chem'}(T_{1},P_{1}))$\\

\end{proof}

\begin{defn}{Ideal Solution}\\
\label{feasiblestraight}
\\
\indent We let $pr_{1}$ is the projection onto the first factor, $pr_{12}$ be the projection onto the first two factors, in coordinates $(T,P,n_{1},\ldots,n_{c})$. We define a feasible chemical path $\gamma:[0,1]\rightarrow \mathcal{R}_{>0}^{2+c}$, such that if $n_{i}(t)=pr_{2+i}(t)$, for $1\leq i\leq c$, then;\\

${n_{i}'\over \nu_{i}}={n_{j}'\over \nu_{j}}$, for $1\leq i<j\leq c$\\

where $\{\nu_{1},\ldots,\nu_{c}\}$ are the stoichiometric coefficients. If $n(t)=\sum_{i=1}^{c}n_{i}(t)$, and $x_{i}(t)=a_{i}(t)={n_{i}\over n}(t)$, then $Q(pr_{12}(\gamma(t)))=\prod_{i=1}^{c}a_{i}(t)^{\nu_{i}}$ and ${n_{i,0}\over n_{0}}=f_{i}(pr_{12}(\gamma(0)))$, where $f_{i}=e^{\mu_{i}-\mu_{i}^{\circ}\over RT}$, $1\leq i\leq c$, see Section \ref{path}. Note that $n>0$ and the $x_{i}$ are well defined.\\

We define a chemical equilibrium path to be a feasible chemical path $\gamma$ with the additional property that chemical equilibrium exists at $pr_{12}(\gamma(t))$, for $t\in [0,1]$. We define a quasi-chemical equilibrium path to be a feasible chemical path $\gamma$ with the additional property that $({\partial G\over \partial \xi})_{T,P}|_{pr_{12}([0,1])}$ is constant. We define a dynamic equilibrium path to be a feasible chemical path $\gamma$ with the additional property that $pr_{2+c}(\gamma)'(t)=0$, for $1\leq i\leq c$.\\

We define a straight line feasible path from $(T,P^{\circ})$ to $(T,P)$ to be a map $\gamma:[0,1]\rightarrow \mathcal{R}_{\geq 0}^{2+c}$ such that $pr_{12}\gamma(0)=(T,P^{\circ})$, $pr_{12}\gamma(1)=(T,P)$, $pr_{1}(\gamma(t))=T$. We say that a point $(T,P)$ is a simple dynamic equilibrium point, if it lies on the locus $Q(T,P)=1$.
\end{defn}

\begin{lemma}
\label{implies}
In an ideal solution, a straight line chemical equilibrium path from $(T,P^{\circ})$ to $(T,P)$ is a dynamic equilibrium path. Every $(T,P^{\circ})$ is a simple dynamic equilibrium point.

\end{lemma}
\begin{proof}
By Lemma \ref{equivalences}, and the definition of activities for an ideal solution, we have that;\\

$1=Q(T,P)=\prod_{i=1}^{c}a_{i}^{\nu_{i}}=\prod_{i=1}^{c}x_{i}^{\nu_{i}}$ $(*)$\\

and ${n_{i}'\over \nu_{i}}={n_{j}'\over \nu_{j}}$, for $1\leq i,j\leq c$, $\sum_{i=1}^{c}n_{i}=n$\\

Using the relation $(*)$, differentiating and using the facts that, for $1\leq i\leq c-1$;\\

$n_{i}'={\nu_{i}n_{c}'\over \nu_{c}}$, $n_{i}={\nu_{i}n_{c}\over \nu_{c}}+d_{i}$ $(!)$\\

we obtain that;\\

$(\prod_{i=1}^{c}x_{i}^{\nu_{i}})'=\sum_{i=1}^{c}\nu_{i}x_{i}^{\nu_{i}-1}x_{i}'\prod_{j\neq i}x_{j}^{\nu_{j}}$\\

$=\sum_{i=1}^{c}\nu_{i}x_{i}^{\nu_{i}-1}x_{i}'x_{i}^{-\nu_{i}}$\\

$=\sum_{i=1}^{c}\nu_{i}x_{i}^{-1}x_{i}'$\\

$=\sum_{i=1}^{c}\nu_{i}{n\over n_{i}}{(n_{i}'n-n_{i}n')\over n^{2}}$\\

$=\sum_{i=1}^{c}\nu_{i}({n_{i}'\over n_{i}}-{n'\over n})$\\

$=\sum_{i=1}^{c-1}{\nu_{i}^{2}n_{c}'\over \nu_{i}n_{c}+\nu_{c}d_{i}}+{\nu_{c}n_{c}'\over n_{c}}-
\lambda({\sum_{i=1}^{c}n_{i}'\over \sum_{i=1}^{c}n_{i}})$\\

$=\sum_{i=1}^{c-1}{\nu_{i}^{2}n_{c}'\over \nu_{i}n_{c}+\nu_{c}d_{i}}+{\nu_{c}n_{c}'\over n_{c}}-
\lambda({(\sum_{i=1}^{c-1}{\nu_{i}\over \nu_{c}}+1)n_{c}'\over (\sum_{i=1}^{c-1}{\nu_{i}\over \nu_{c}}+1)n_{c}+\sum_{i=1}^{c-1}d_{i}})$\\

$=0$ $(B)$\\

where $\lambda=\sum_{i=1}^{c}\nu_{i}$. If $\nu_{c}'\neq 0$, we obtain that;\\

$=\sum_{i=1}^{c-1}{\nu_{i}^{2}\over \nu_{i}n_{c}+\nu_{c}d_{i}}+{\nu_{c}\over n_{c}}-
\lambda({(\sum_{i=1}^{c-1}{\nu_{i}\over \nu_{c}}+1)\over (\sum_{i=1}^{c-1}{\nu_{i}\over \nu_{c}}+1)n_{c}+\sum_{i=1}^{c-1}d_{i}})=0$\\

so that;\\

$\sum_{i=1}^{c-1}\nu_{i}^{2}n_{c}(\alpha n_{c}+\beta)\prod_{j\neq i}(\nu_{j}n_{c}+\nu_{c}d_{j})+\nu_{c}(\alpha n_{c}+\beta)\prod_{i=1}^{c-1}(\nu_{i}n_{c}+\nu_{c}d_{i})$\\

$-\lambda n_{c}\prod_{i=1}^{c-1}(\nu_{i}n_{c}+\nu_{c}d_{i})=0$\\

where $\alpha=\sum_{i=1}^{c-1}{\nu_{i}\over \nu_{c}}+1$ and $\beta=\sum_{i=1}^{c-1}d_{i}$\\

which we can write in the form;\\

$\sum_{j=0}^{c}\gamma_{j}\nu_{c}^{j}=0$\\

We have that;\\

$\gamma_{c}=\sum_{i=1}^{c-1}\nu_{i}^{2}\alpha\prod_{j\neq i}\nu_{j}+\alpha \nu_{c}\prod_{i=1}^{c-1}\nu_{i}-\lambda\prod_{i=1}^{c-1}\nu_{i}$\\

$=\alpha\delta\sum_{i=1}^{c-1}\nu_{i}+\alpha\delta\nu_{c}-\lambda\delta$\\

$=\delta(\alpha(\sum_{i=1}^{c}\nu_{i})-\lambda)$\\

$=\delta\lambda(\alpha-1)$\\

Noting that $\delta\neq 0$ and $\alpha-1=\sum_{i=1}^{c-1}{\nu_{i}\over \nu_{c}}$, we have that $\gamma_{c}\neq 0$ iff $\sum_{i=1}^{c-1}\nu_{i}\neq 0$ and $\sum_{i=1}^{c}\nu_{i}\neq 0$. In this case, we obtain a nontrivial polynomial relation $p(n_{c})=0$, so that, by continuity and discreteness of roots, $n_{c}$ is a constant and $n_{c}'=0$. By the connecting relations $(!)$, we obtain that $n_{j}'=0$ as well, for $\leq j\leq c-1$.\\

If $\sum_{i=1}^{c}\nu_{i}=0$ $(G)$ then $\sum_{i=1}^{c-1}\nu_{i}\neq 0$, and, we have, from $(*)$, that;\\

$\prod_{i=1}^{c}n_{i}^{\nu_{i}}=n^{\sum_{i=1}^{c}\nu_{i}}=1$ $(A)$\\

and, following the calculation $(B)$, we obtain;\\

$\sum_{i=1}^{c-1}{\nu_{i}^{2}n_{c}'\over \nu_{i}n_{c}+\nu_{c}d_{i}}+{\nu_{c}n_{c}'\over n_{c}}=0$\\

so that, if $n_{c}'\neq 0$;\\

$\sum_{i=1}^{c-1}{\nu_{i}^{2}\over \nu_{i}n_{c}+\nu_{c}d_{i}}+{\nu_{c}\over n_{c}}=0$\\

and, differentiating $k$ times, for $k\geq 0$, cancelling $n_{c}'$ if $n_{c}'\neq 0$, and using the chain rule, we obtain the relations;\\

$\sum_{i=1}^{c-1}{\nu_{i}^{k+2}\over (\nu_{i}n_{c}+\nu_{c}d_{i})^{k+1}}+{\nu_{c}\over n_{c}^{k+1}}=0$\\

Let $g_{i}={1\over \nu_{i}+{\nu_{c}d_{i}\over n_{c}}}={n_{c}\over \nu_{c}n_{i}}<0$, for $1\leq i\leq c-1$. Then, for $k\geq 0$;\\

$\sum_{i=1}^{c-1}\nu_{i}^{k+2}g_{i}^{k+1}+\nu_{c}=\sum_{i=1}^{c-1}\nu_{i}(\nu_{i}g_{i})^{k+1}+\nu_{c}$\\

$=0$ $(E)$\\

If $g_{i}=-1$, then;\\

${n_{c}\over \nu_{c}n_{i}}={n_{c}\over \nu_{c}({\nu_{i}n_{c}\over \nu_{c}}+d_{i})}={n_{c}\over \nu_{i}n_{c}+\nu_{c}d_{i}}=-1$\\

so that;\\

$n_{c}=-(\nu_{i}n_{c}+\nu_{c}d_{i})$\\

$(1+\nu_{i})n_{c}=\nu_{c}d_{i}$\\

Then if $n_{c}'\neq 0$, we must have that $\nu_{i}=-1$, $d_{i}=0$. Rescaling each $\nu_{i}$ by a factor of $2$, we still have the conditions $(*)$, $(E)$,$(G)$, so we can assume that $|\nu_{i}|\geq 2$, which is a contradiction. Hence, we can assume that $|g_{i}|\neq 1$, for $1\leq i\leq c-1$, so that taking the limit as $k\rightarrow\infty$, with $k$ even, so that $\nu_{i}^{k+2}>0$, $g_{i}^{k+1}<0$, for $1\leq i\leq c-1$, we obtain a contradiction, and conclude that $n_{c}'=0$, and $n_{i}'=0$, for $1\leq i\leq c-1$.\\

If for every $c-1$ element subset $I_{j}\subset \{\nu_{1},\ldots \nu_{c}\}$, $1\leq j\leq c$, we have that $\sum_{i\in I_{j}}\nu_{i}=0$, then clearly;\\

$\sum_{1\leq i\leq c}\nu_{i}=\sum_{i\in I_{j}}\nu_{i}+\nu_{j}$\\

$=\nu_{j}$ for $1\leq j\leq c$\\

which we can exclude. It follows that there exists some $j_{0}$, with $1\leq j_{0}\leq c$, such that $\sum_{i\in I_{j_{0}}}\nu_{i}\neq 0$. Using $n_{j_{0}}$ as the pivot and following the above proof, replacing $n_{c}$ by $n_{j_{0}}$, we can, without loss of generality, assume that $\sum_{i=1}^{c-1}\nu_{i}\neq 0$, and the proof is complete.\\

The second claim follows from the fact in Lemma \ref{equivalences} that $Q(T,P^{\circ})=1$.\\

\end{proof}
\begin{defn}
\label{curves}
For $c\in\mathcal{R}_{>0}$, we define $C_{c}\subset \mathcal{R}^{2}$ to be the zero locus of $Q(T,P)-c$. We define $D\subset \mathcal{R}^{2}$ to be the condition of chemical equilibrium. We define $D_{c}$ to be the zero locus of $({\partial G\over \partial \xi})_{T,P}-c=0$ .

\end{defn}

\begin{lemma}
\label{exists}
For every smooth curve $W\subset \mathcal{R}^{2}$, there exists a locally feasible path $\gamma:[0,1]\rightarrow \mathcal{R}^{2+c}$ with $pr_{12}(\gamma)\subset W$.

\end{lemma}
\begin{proof}
As $W$ is smooth, we can choose a local parametrisation $\delta:[0,1]\rightarrow W$. Let $\epsilon(t)=Q(\delta(t))>0$ and $w=\sum_{i=1}^{c}\nu_{i}$. Without loss of generality, we can assume that $pr_{12}(\gamma(0))=(T_{0},P_{0})$, see Section \ref{path}. We have that;\\

$\prod_{i=1}^{c}x_{i}^{\nu_{i}}(t)=\epsilon(t)$ $(\dag)$\\

iff  $\prod_{i=1}^{c}n_{i}^{\nu_{i}}(t)=\epsilon(t)n^{\sum_{i=1}^{c}\nu_{i}}(t)$\\

iff $\prod_{i=1}^{c-1}n_{i}^{\nu_{i}}(t)n_{c}^{\nu_{c}}(t)=\epsilon(t)(\sum_{i=1}^{c}n_{i})^{\sum_{i=1}^{c}\nu_{i}}(t)$\\

iff $\prod_{i=1}^{c-1}({\nu_{i}\over \nu_{c}}n_{c}+d_{i})^{\nu_{i}}(t)n_{c}^{\nu_{c}}(t)=
\epsilon(t)(\sum_{i=1}^{c}n_{i})^{\sum_{i=1}^{c}\nu_{i}}(t)$\\

iff $\prod_{i=1}^{c-1}({\nu_{i}\over \nu_{c}}n_{c}+d_{i})^{\nu_{i}}(t)n_{c}^{\nu_{c}}(t)=
\epsilon(t)((\sum_{i=1}^{c-1}{\nu_{i}\over \nu_{c}}+1)n_{c}+\sum_{i=1}^{c-1}d_{i})^{w}(t)$\\

Let $d_{i}>0$, for $1\leq i\leq c-1$, then;\\

$\prod_{i=1}^{c}x_{i}^{\nu_{i}}(t)=\epsilon(t)$\\

iff $\prod_{i=1}^{c-1}({\nu_{i}\over \nu_{c}}n_{c}+d_{i})^{\nu_{i}}(t)n_{c}^{\nu_{c}}(t)=
\epsilon(t)((\sum_{i=1}^{c-1}{\nu_{i}\over \nu_{c}}+1)n_{c}+\sigma)^{w}(t)$\\

iff $\prod_{i=1}^{c-1}({\nu_{i}\over \nu_{c}}n_{c}+d_{i})^{\nu_{i}}(t)n_{c}^{\nu_{c}}(t)=\epsilon(t)({w\over \nu_{c}}n_{c}+\sigma)^{w}(t)$\\

where $\sigma=\sum_{i=1}^{c-1}d_{i}>0$. Assume that $w>0$, then we have that;\\

$\prod_{i=1}^{c}x_{i}^{\nu_{i}}(t)=\epsilon(t)$\\

iff $\prod_{i=1}^{p}({\nu_{i}\over \nu_{c}}n_{c}+d_{i})^{\nu_{i}}=\epsilon(t)({w\over \nu_{c}}n_{c}+\sigma)^{w}(t)n_{c}^{-\nu_{c}}\prod_{i=p+1}^{c-1}({\nu_{i}\over \nu_{c}}n_{c}+d_{i})^{-\nu_{i}}$\\

iff $\prod_{i=1}^{p}({\nu_{i}\over \nu_{c}})^{\nu_{i}}n_{c}^{\kappa}+r(n_{c})=\epsilon(t)[({w\over \nu_{c}})^{w}\prod_{i=p+1}^{c-1}({\nu_{i}\over \nu_{c}})^{-\nu_{i}}n_{c}^{w-\nu_{c}-\lambda}+s(n_{c})]$\\

iff $\prod_{i=1}^{p}({\nu_{i}\over \nu_{c}})^{\nu_{i}}n_{c}^{\kappa}+r(n_{c})=\epsilon(t)[({w\over \nu_{c}})^{w}\prod_{i=p+1}^{c-1}({\nu_{i}\over \nu_{c}})^{-\nu_{i}}n_{c}^{\kappa}+s(n_{c})]$\\

iff $\alpha n_{c}^{\kappa}+r(n_{c})=\epsilon(t)(\beta n_{c}^{\kappa}+s(n_{c}))$ $(\dag\dag\dag)$\\

where $\alpha=\prod_{i=1}^{p}({\nu_{i}\over \nu_{c}})^{\nu_{i}}\neq 0$, $\beta=({w\over \nu_{c}})^{w}\prod_{i=p+1}^{c-1}({\nu_{i}\over \nu_{c}})^{-\nu_{i}}\neq 0$, $\lambda=\sum_{i=p+1}^{c-1}\nu_{i}$, $\{r,s\}\subset \mathcal{R}[x]$ have degree less than $\kappa$, $r(0)=\prod_{i=1}^{p}d_{i}>0$, $s(0)=0$, as divisible by $x^{-\nu_{c}}$. Dividing $(\dag\dag\dag)$ by $n_{c}^{\kappa}>0$, we obtain;\\

$\alpha+r_{1}({1\over n_{c}})=\epsilon(t)(\beta+s_{1}({1\over n_{c}}))$ $(C)$\\

where $\{r_{1},s_{1}\}\subset \mathcal{R}[x]$ have degree $\kappa$, with $r_{1}(0)=s_{1}(0)=0$, $deg(r_{1})=\kappa$ and $c_{\kappa}=r(0)=\prod_{i=1}^{p}d_{i}>0$ where $r_{1}=\sum_{j=0}^{\kappa}c_{j}x^{j}$, $deg(s_{1})\leq \kappa+\nu_{c}<\kappa$, so that $lim_{x\rightarrow \infty}{\alpha+r_{1}(x)\over \beta+s_{1}(x)}=\infty$.\\

Let $v(x)=\alpha x^{\kappa}+r(x)$, $w(x)=\beta x^{\kappa}+s(x)$. Then, the roots $v_{i}$ of $r(x)$, $1\leq i\leq p$, are given by $v_{i}=-{d_{i}\nu_{c}\over \nu_{i}}>0$, while the roots $w_{i}$ of $w(x)$ are given by $w_{0}=0$, $w_{i}={-d_{i}\nu_{c}\over \nu_{i}}<0$, $p+1\leq i\leq c-1$ and $w_{c}={-\sigma\nu_{c}\over w}$. If $\nu_{i}=w$, for $1\leq i\leq p$, it would follow that $w>\sum_{i=1}^{p}\nu_{i}=pw$, which is a contradiction, as $p\geq 1$. It follows that we can choose $i_{0}$ with $1\leq i_{0}\leq p$ such that ${-\nu_{c}\over \nu_{i_{0}}}\neq {-\nu_{c}\over w}$ and ${-\nu_{c}d_{i_{0}}\over \nu_{i_{0}}}\neq {-\nu_{c}d_{i_{0}}\over w}$,(\footnote{\label{footnote1}As $n_{c}$ is mobile, the conditions that $n_{i,0}=d_{i}+{\nu_{i}\over \nu_{c}}n_{c,0}$, $1\leq i\leq c-1$, $(B)$, together with the requirement that $\overline{n}_{0}=(n_{1,0},\ldots,n_{c,0})$ lies in $Ker(M)\cap \mathcal{R}_{>0}^{c}$, see Section \ref{path}, places a $1$ dimensional restriction on the tuple $(d_{1},\ldots d_{c-1})$, defined as $pr_{c-1}(W)$, where $W=V\cap g^{-1}(Ker(M))$, $V=\{(d_{1},\ldots d_{c-1},n_{c,0}(d_{1},\ldots d_{c-1})):(d_{1},\ldots d_{c-1})\in\mathcal{R}^{c-1}\}$, and $g(x_{1},\ldots, x_{c})=(y_{1},\ldots, y_{c})$ is the morphism defined by;\\

$y_{c}=x_{c}$\\

$y_{i}=x_{i}+{\nu_{i}\over \nu_{c}}x_{c}$, $(1\leq i\leq c-1)$\\

determined by the conditions $(B)$. Moreover, we can assume that $pr_{c-1}(W)\cap \mathcal{R}^{c-1}\neq \emptyset$. Moving the tuple $(d_{1},\ldots d_{c-1})$ along $pr_{c-1}(W)$, we can fine $d_{i_{0}}>0$, and $d_{j}$, $1\leq j\leq c-1$, $j\neq {i_{0}}$, with $d_{j}>0$, so that ${-\nu_{c}d_{i_{0}}\over \nu_{i_{0}}}\neq {-\nu_{c}\sigma\over w}$ and $v_{i_{0}}$ is the highest root of $r(x)$, so that the roots $w_{i}$ of $w(x)$ do not coincide with the highest root $v_{i_{0}}$ of $r(x)$.}). Dividing by $x^{\kappa}$ doesn't effect the positive roots $v_{i}$, $1\leq i\leq p$ of ${r(x)\over x^{\kappa}}=\alpha+r_{1}({1\over x})$, and the positive roots of $\alpha+r_{1}(x)$, are the positive reciprocals $v_{i}'={1\over v_{i}}$, $1\leq i\leq p$. As $lim_{x\rightarrow \infty}{\alpha+r_{1}(x)\over \beta+s_{1}(x)}=\infty$ and $\epsilon(0)>0$, we can choose $v_{0}$ with $q(v_{0})=\epsilon(0)$ and $v_{0}>v_{i}'$, for $1\leq i\leq p$, where $q(x)={\alpha+r_{1}(x)\over \beta+s_{i}(x)}$. Then ${1\over v_{0}}<{1\over v_{i}'}=v_{i}$, for $1\leq i\leq p$, and as ${\nu_{i}\over \nu_{c}}<0$, ${\nu_{i}\over \nu_{c}}{1\over v_{0}}>{\nu_{i}\over \nu_{c}}v_{i}$, ${\nu_{i}\over \nu_{c}}{1\over v_{0}}+d_{i}>{\nu_{i}\over \nu_{c}}v_{i}+d_{i}=0$, for $1\leq i\leq p$. In particularly, as $n_{c}(0)={1\over v_{0}}$, we have, by the linking relations, that $n_{i}(0)={\nu_{i}\over \nu_{c}}n_{c}(0)+d_{i}={\nu_{i}\over \nu_{c}}{1\over v_{0}}+d_{i}>0$, for $1\leq i\leq p$. Moreover, as ${\nu_{i}\over \nu_{c}}>0$, for $i+1\leq i\leq c-1$, and ${1\over v_{0}}>0$, we have that $n_{i}(0)={\nu_{i}\over \nu_{c}}n_{c}(0)+d_{i}={\nu_{i}\over \nu_{c}}{1\over v_{0}}+d_{i}>0$ as well, for $i+1\leq i\leq p-1$, $(D)$\\

By $(C)$, we have that $q({1\over n_{c}})=\epsilon(t)$ and, by the above construction, it follows that $q(v_{0})=\epsilon(0)$. Consider the real algebraic curve defined by $\theta(x,y)=q(x)-y-\epsilon(0)$, so that $\theta(v_{0},0)=0$. Computing the differential $(q'(x),-1)$, if $q'(v_{0})\neq 0$, we see that the projection $pr_{y}$ is unramified at $(v_{0},0)$, so that we can apply the inverse function theorem, see \cite{dep2}, to obtain a real branch $\gamma(y)$ with $\gamma(0)=v_{0}$ and;\\

$\theta(\gamma(y),y)=q(\gamma(y))-y-\epsilon(0)=0$\\

Replacing $y$ with $\epsilon(t)-\epsilon(0)$, and letting $\delta(t)=\gamma(\epsilon(t)-\epsilon(0))$, we have that;\\

$\theta(\delta(t),\epsilon(t)-\epsilon(0))=q(\delta(t))-(\epsilon(t)-\epsilon(0))-\epsilon(0)=q(\delta(t))-
\epsilon(t)=0$\\

Then $\delta(t)>0$, and we can set $n_{c}={1\over \delta(t)}$, with $n_{c}(0)={1\over v_{0}}$ Using the linkage relations, we can define $n_{i}(t)$, for $1\leq i\leq c-1$ from $n_{c}$. As, by $(D)$, we have that $n_{i}(0)>0$, for $1\leq i\leq c$, by continuity, for sufficiently small $t$, we have that $n_{i}(t)>0$, for $1\leq i\leq c$ as well.\\

If $w<0$, we can take the reciprocal of the relation $(\dag)$, replace $\nu_{i}$ by $-\nu_{i}$ and $\epsilon(t)$ by ${1\over \epsilon(t)}>0$, to get $w>0$. Reordering so that the pivot $\nu_{c}<0$, $\nu_{i}>0$, for $1\leq i\leq p'$, $\nu_{i}<0$, for $p'+1\leq i\leq c-1$, we can carry out the above proof to get the result.\\

If $w=0$, we can carry out the first calculation with $\beta$ replaced by $\prod_{i=p+1}^{c-1}({\nu_{i}\over \nu_{c}})^{-\nu_{i}}\neq 0$.\\

\end{proof}
\begin{lemma}
\label{feasible}
A feasible path $\gamma$ is a dynamic equilibrium path iff $pr(\gamma_{12})\subset C_{f}$, for some $f\in\mathcal{R}_{>0}$, iff ${dQ\over dt}=0$. In particular, for any feasible path $\gamma$ in which $pr_{12}(\gamma)$ is fixed, we have dynamic equilibrium and ${dQ\over dt}=0$.\\

\end{lemma}
\begin{proof}
For the first claim, we have that $f>0$ and if $pr(\gamma_{12})\subset C_{f}$, we have that;\\

$\prod_{i=1}^{c}x_{i}^{\nu_{i}}=f$\\

with the same linkage relations as Lemma \ref{implies}. Now follow through Lemma \ref{implies}, noting that differentiating reduces the constant $f$ to $0$, as in the proof. Conversely, if $pr(\gamma_{12})\not\subset C_{f}$, then we have that;\\

$(\prod_{i=1}^{c}x_{i}^{\nu_{i}})'|_{0}=Q(\gamma(t))'|_{0}$\\

$=(grad(Q)|_{\gamma(0)}\centerdot \gamma'(0))\neq 0$\\

but, if $\gamma$ is a dynamic equilibrium path, then clearly each $n_{i}$ is constant, $1\leq i\leq c$, $n$ is constant and $x_{i}$ is constant, so that $x_{i}'=0$, for $1\leq i\leq c$ and $(\prod_{i=1}^{c}x_{i}^{\nu_{i}})'|_{0}=0$, which is a contradiction. For the second claim, if  $pr(\gamma_{12})\subset C_{c}$, for some $c\in\mathcal{R}$, it follows from Definition \ref{curves} and the proof of Lemma \ref{implies}, that $Q$ is constant and ${dQ\over dt}=0$. Conversely, if ${dQ\over dt}=0$, then $Q$ is constant along $pr(\gamma_{12})$, so that $pr(\gamma_{12})\subset C_{c}$, for some $c\in\mathcal{R}$.  The final claim follows from the fact that $Q$ depends only on the coordinates $(T,P)$, so that ${dQ\over dt}=0$, and the first claim.

\end{proof}

\begin{lemma}
\label{standardpressure2}
We have that the condition of chemical equilibrium defines a $1$-dimensional curve $D$ in the state space $(T,P)$. Similarly, the conditions that $Q(T,P)=c$ define $1$-dimensional curves $C_{c}$ in $(T,P)$, and if $\gamma:[0,1]\rightarrow (T,P,n_{1},\ldots n_{c})$ is a path, such that $pr_{12}(\gamma)$ lies in $C_{c}$, then it must be a dynamic equilibrium path. Let $D'$ be a component of $D$, then $Q$ is constant along $D'$ iff ${\Delta G^{\circ}\over T}$ is constant along $D'$. Let $C_{c}'$ be a component of $C_{c}$, then $({\partial G\over \partial \xi})|_{T,P}=0$ along $C_{c}'$ iff ${\Delta G^{\circ}\over T}=-Rln(c)$. Assuming that ${\Delta G^{\circ}\over T}$ is non constant, we have that $Q$ is constant along $D'$ iff $pr_{1}(D')$ is a fixed temperature $T$, and  $({\partial G\over \partial \xi})|_{T,P}=0$ along $C_{c}'$ iff $pr_{1}(C_{c}')$ is a fixed temperature $T$. The only feasible paths which are both chemical and dynamic equilibrium paths are straight line chemical equilibrium paths. There exists a feasible dynamic equilibrium path, with $pr_{12}(\gamma)\subset P=P^{\circ}$.
\end{lemma}
\begin{proof}
For the first part, either use the fact that $({\partial G\over \partial \xi})|_{T,P}$ only depends on $(T,P)$ and differentiability properties, or the result from Lemma \ref{equivalences} that chemical equilibrium is defined by $Q(T,P)-e^{-\Delta G^{\circ}\over RT}=0$, and the fact that $\Delta G^{\circ}$ depends on $T$. For the second part, either use differentiability properties of $Q(T,P)$ or the fact from Lemma \ref{equivalences} that $Q=1$ iff $({\partial G\over \partial \xi})|_{T,P}-\Delta G^{\circ}=0$. The second claim is clear from Lemma \ref{feasible}.
For the third claim, we have, by Lemma \ref{equivalences}, that along $D'$, $Q=e^{\Delta G^{\circ}\over RT}$, so that clearly $Q$ is constant along $D'$ iff ${\Delta G^{\circ}\over T}$ is constant. The fourth claim is clear by the fact that $({\partial G\over \partial \xi})|_{T,P}-\Delta G^{\circ}=RTln(c)$ along $C_{c}'$. The fifth and sixth claims follow immediately from the fact that ${\Delta G^{\circ}\over T}$ is a function of $T$ and is non constant. For the seventh claim, if $\gamma$ is a chemical and dynamic equilibrium path, then $pr_{12}(\gamma)\subset D$, and, by Lemma \ref{feasible}, $pr_{12}(\gamma)\subset C_{c}$, for some $c\in\mathcal{R}$. It follows that $Q$ is constant along the $pr_{12}(\gamma)\subset D'$ for some component $D'$, and then, by the fifth claim, $pr_{1}(\gamma)\subset pr_{1}(D')$ is a fixed temperature $T$, so that $\gamma$ is a straight line chemical equilibrium path. For the final claim, we have that $Q(T,P^{\circ})=1$, by Lemma \ref{implies}, so that $P=P^{\circ}$ lies in $C_{1}$. It follows, by Lemma \ref{feasible}, that a feasible path $\gamma$ with $pr_{12}(\gamma)\subset P=P^{\circ}$ is a dynamic equilibrium path.
\end{proof}
\end{section}
\begin{section}{Ideal Solutions}
\label{errorterms}
\begin{rmk}
\label{ideal2}
We have, using the phase rule for an ideal solution in equilibrium with its vapour, and using the ideal gas law, see \cite{OM}, that;\\

$\mu_{i}^{(g)}=\mu_{i}^{\circ(g)}+RTln({P_{i}\over P^{\circ}})$\\

$\mu_{i}^{(sol)}=\mu_{i}^{\circ(sol)}+RTln({P_{i}\over P^{\circ}})$ $(*)$\\

By the definition of an ideal solution, we have that;\\

$\mu_{i}=\mu_{i}^{*}+RTln(x_{i})$ $(**)$\\

where, by $\mu_{i}^{*}(T,P)$, we mean the chemical potential of substance $i$ on its own, at temperature and pressure $(T,P)$. By Raoult's law $P_{i}=x_{i}P_{i}^{*}$, see \cite{M}, combined with $(**)$, we obtain;\\

$\mu_{i}=\mu_{i}^{*}+RTln(x_{i})$\\

$=\mu_{i}^{*}+RTln({P_{i}\over P_{i}^{*}})$\\

$=\mu_{i}^{*}+RTln({P_{i}\over P^{\circ}})-RTln({P_{i}^{*}\over P^{\circ}})$ $(***)$\\

Combining $(*),(***)$, we obtain that;\\

$\mu_{i}^{*}=\mu_{i}-RTln({P_{i}\over P^{\circ}})+RTln({P_{i}^{*}\over P^{\circ}})$\\

$=(\mu_{i}^{\circ}+RTln({P_{i}\over P^{\circ}}))-RTln({P_{i}\over P^{\circ}})+RTln({P_{i}^{*}\over P^{\circ}})$\\

$=\mu_{i}^{\circ}+RTln({P_{i}^{*}\over P^{\circ}})$ $(\dag)$\\

Letting $P_{i}^{*}=P^{\circ}$, we obtain that $\mu_{i}^{*}(T,P')=\mu_{i}^{\circ}$, $(****)$, where $(T,P')$ is the temperature and pressure at which the equilibrium pressure $P_{i}'^{*}=P^{\circ}$. From $(**)$, the fact that $\mu_{i}^{*}(T,P)\simeq \mu_{i}^{*}(T,P')$ and $(****)$, we obtain that;\\

$\mu_{i}\simeq \mu_{i}^{\circ}+RTln(x_{i})$, $(*****)$\\

as a very good approximation. This avoids the contradiction that $x_{i}=1$ for a solution involving more than one component, at $P=P^{\circ}$. To make the results here more precise, we need to compute the error term, but the proof is still consistent if we allow that $n_{i}(T,P)\rightarrow 0$ as $P\rightarrow P^{\circ}$, so that $x_{i}={n_{i}\over n}\rightarrow 1$, and $x_{i}$ is not defined at $P=P^{\circ}$.\\

More specifically, we have that;\\

$\mu_{i}^{*}(T,P)=\mu_{i}^{*}(T,P')+\delta$\\

where $\delta=\mu_{i}^{*}(T,P)-\mu_{i}^{*}(T,P')$, so that;\\

$\mu_{i}=\mu_{i}^{\circ}+RTln(x_{i})+\delta$\\

For Raoult's law, see \cite{M}, we also need an approximation. We have that, by the definition of an ideal solution, the phase rule, Dalton's law that each gas in a mixture of ideal gases behaves as if it were alone in the container at the equilibrium pressures $\{P_{i},P_{i}^{*}\}$, see \cite{M}, that;\\

$\mu_{i}=\mu_{i}^{*}+RTln(x_{i})$\\

$=\mu_{i}^{*}(T,P_{i}^{*})+RTln(x_{i})+\epsilon$\\

$=\mu_{i}^{\circ(g)}+RTln({P_{i}^{*}\over P^{\circ}})+RTln(x_{i})+\epsilon$\\

$=\mu_{i}^{\circ(g)}+RTln({P_{i}\over P^{\circ}})$\\

so that;\\

$RTln(x_{i})=RTln({P_{i}\over P^{\circ}})-RTln({P_{i}^{*}\over P^{\circ}})-\epsilon$\\

where $\epsilon=\mu_{i}^{*}(T,P)-\mu_{i}^{*}(T,P_{i}^{*})$, so that $(***)$ becomes;\\

$\mu_{i}=\mu_{i}^{*}+RTln(x_{i})$\\

$=\mu_{i}^{*}+RTln({P_{i}\over P^{\circ}})-RTln({P_{i}^{*}\over P^{\circ}})-\epsilon$ $(***)'$\\

Combining $(*),(***)'$, we obtain that;\\

$\mu_{i}^{*}=\mu_{i}-RTln({P_{i}\over P^{\circ}})+RTln({P_{i}^{*}\over P^{\circ}})+\epsilon$\\

$=(\mu_{i}^{\circ}+RTln({P_{i}\over P^{\circ}}))-RTln({P_{i}\over P^{\circ}})+RTln({P_{i}^{*}\over P^{\circ}})+\epsilon$\\

$=\mu_{i}^{\circ}+RTln({P_{i}^{*}\over P^{\circ}})+\epsilon$ $(\dag)'$\\

Letting $P_{i}^{*}=P^{\circ}$ again, we obtain that $\mu_{i}^{*}(T,P')=\mu_{i}^{\circ}+\epsilon$, $(****)'$\\

From $(**)$, $(****)'$, we obtain that;\\

$\mu_{i}=\mu_{i}^{*}+RTln(x_{i})$\\

$=\mu_{i}^{*}(T,P')+\delta+RTln(x_{i})$\\

$=\mu_{i}^{\circ}+\epsilon+\delta+RTln(x_{i})$\\

$=\mu_{i}^{\circ}+RTln(x_{i})+\gamma_{i}$\\

where $\gamma_{i}=\epsilon+\delta=\mu_{i}^{*}(T,P)-\mu_{i}^{*}(T,P_{i}^{*})+\mu_{i}^{*}(T,P)-\mu_{i}^{*}(T,P')$\\

$=2\mu_{i}^{*}(T,P)-\mu_{i}^{*}(T,P_{i}^{*})-\mu_{i}^{*}(T,P')\simeq 0$\\

Using Lemma \ref{differential}, we have that $dG=-SdT+VdP$, so that, if temperature is fixed, $dG=VdP$, then, for the Gibbs energy function of substance $i$ on it own, in the liquid phase;\\

$\mu_{i}^{*}(T,P)-\mu_{i}^{*}(T,P_{i}^{*})={G(T,P,n)-G(T,P_{i}^{*},n)\over n}$\\

$={1\over n}\int_{P_{i}^{*}}^{P}dG$\\

$={1\over n}\int_{P_{i}^{*}}^{P}VdP$\\

$={1\over n}\int_{P_{i}^{*}}^{P}{nN_{A}m_{i}\over \kappa_{i}(T,P)}dP$\\

$\simeq {N_{A}m_{i}(P-P_{i}^{*})\over \kappa}$\\

$=V_{m,i}(P-P_{i}^{*})$\\

where $\kappa(T,P)$ is the density of substance $i$ in the liquid phase, and which we assume to be approximately constant, and $V_{m.i}$ is the molar volume. Similarly;\\

$\mu_{i}^{*}(T,P)-\mu_{i}^{*}(T,P')\simeq {N_{A}m_{i}(P-P')\over \kappa}$\\

$=V_{m,i}(P-P')$\\

so that;\\

$\gamma_{i}(P)\simeq V_{m,i}(2P-P_{i}^{*}-P')\simeq 0$\\

\end{rmk}
We reformulate Lemmas \ref{gibbs}, \ref{equivalences}, \ref{vanhoffhelmholtz} and \ref{eqlines} with this error term;\\

\begin{lemma}
\label{gibbs2}

In the ideal solution case, for the energy function $G$ involving only $c$ uncharged species;\\

$({\partial G\over \partial \xi})_{T,P}=\Delta G^{\circ}+RTln(Q)+\epsilon$\\

where $\epsilon(P)=\sum_{i=1}^{c}\nu_{i}\gamma_{i}(P)\simeq 0$ and $\gamma_{i}(P)\simeq 0$ is the error term for the $i$'th uncharged species in Remark \ref{ideal2}.

\end{lemma}
\begin{proof}
By Lemma \ref{equivalences}, we have that;\\

$({\partial G\over \partial \xi})_{T,P}=\sum_{i=1}^{c}\nu_{i}\mu_{i}$\\

$\Delta G^{\circ}=\sum_{i=1}^{c}\nu_{i}\mu_{i}^{\circ}$, $(*)$\\

Using $(*)$, the fact that $\mu_{i}=\mu_{i}^{\circ}+RTln(a_{i})+\gamma_{i}$, and Definition \ref{constants}, we have that;\\

$({\partial G\over \partial \xi})_{T,P}-\Delta G^{\circ}=\sum_{i=1}^{c}\nu_{i}(\mu_{i}-\mu_{i}^{\circ})$\\

$=\sum_{i=1}^{c}\nu_{i}(\mu_{i}^{\circ}+RTln(a_{i})+\gamma_{i}-\mu_{i}^{\circ})$\\

$=\sum_{i=1}^{c}\nu_{i}RTln(a_{i})+\sum_{i=1}^{c}\nu_{i}\gamma_{i}$\\

$=RTln(\prod_{i=1}^{c}a_{i}^{\nu_{i}})+\sum_{i=1}^{c}\nu_{i}\gamma_{i}=RTln(Q)+\epsilon$\\

\end{proof}
\begin{lemma}
\label{equivalences2}

For an ideal solution, we have, using the definition of $\epsilon(P)$ in Lemma \ref{gibbs}, and the error terms $\gamma_{i}(P)$, $1\leq i\leq c$ in Remark \ref{ideal2}, that;\\

$({\partial G\over \partial \xi})_{T,P}=\sum_{i=1}^{c}\nu_{i}\mu_{i}$\\

$\Delta G^{\circ}=\sum_{i=1}^{c}\nu_{i}\mu_{i}^{\circ}$\\

At chemical equilibrium $T,P$, $({\partial G\over \partial \xi})_{T,P}=0$ and at $T,P^{0}$, $\Delta G^{\circ}=0$.\\

If chemical and electrical chemical equilibrium exists at $(T,P^{\circ})$ and $(T,P)$, $Q(T,P)=e^{-\epsilon(P)\over RT}\simeq 1$ and $E=E^{\circ}$. Conversely, if $Q(T,P)=e^{-\epsilon(P)\over RT}\simeq 1$ and chemical equilibrium exists at $(T,P^{\circ})$ then chemical equilibrium exists at $(T,P)$.\\

Chemical equilibrium exists at $(T,P)$ iff $Q(T,P)=e^{-\Delta G^{\circ}-\epsilon(P)\over RT}$\\

We always have that $Q(T,P^{\circ})=e^{-\delta\over RT}\simeq 1$, where;\\

$\delta=\epsilon(P^{\circ})=\sum_{i=1}^{c}\nu_{i}\gamma_{i}(P^{\circ})$\\

\end{lemma}

\begin{proof}
For the first claim, we have, using the definition of $\xi$, that;\\

$dn_{i}=\nu_{i}d\xi$, $(1\leq i\leq c)$, $(*)$\\

By Lemma \ref{differential}, fixing $T$ and $P$, and using $(*)$, we have that;\\

$dG=\sum_{i=1}^{c}\mu_{i}d n_{i}$\\

$=(\sum_{i=1}^{c}\mu_{i}\nu_{i})d\xi$, $(\dag)$\\

so that;\\

$({\partial G\over \partial \xi})_{T,P}=\sum_{i=1}^{c}\mu_{i}\nu_{i}$, $(\dag\dag)$\\

The second claim from the first, as;\\

$\Delta G^{\circ}(T)=\int_{0}^{1}({\partial G\over \partial \xi})_{T,P^{\circ}}$\\

$=\int_{0}^{1}(\sum_{i=1}^{c}\nu_{i}\mu_{i}^{\circ}(T))d\xi$\\

$=\sum_{i=1}^{c}\nu_{i}\mu_{i}^{\circ}(T)\int_{0}^{1}d\xi$\\

$=\sum_{i=1}^{c}\nu_{i}\mu_{i}^{\circ}(T)$\\

noting that $({\partial G\over \partial \xi})_{T,P^{\circ}}$ doesn't vary with $\xi$. For the third claim, at chemical equilibrium, $T,P$, noting again that $({\partial G\over \partial \xi})_{T,P}$ doesn't vary with $\xi$, and using $(\dag,\dag\dag)$, we have that;\\

$dG=({\partial G\over \partial \xi})_{T,P}=0$, (independently of $\xi$) $(\dag\dag\dag)$\\

At chemical equilibrium $T,P^{\circ}$, using the first and second claims, and $(\dag\dag\dag)$, we have that;\\

$dG=({\partial G\over \partial \xi})_{T,P^{\circ}}$\\

$=\sum_{i=1}^{c}\nu_{i}\mu_{i}^{\circ}$\\

$=\Delta G^{\circ}=0$\\

For the second to last claim, and the first direction, we have, by Lemma \ref{gibbs}, that $RTln(Q)=-\epsilon\simeq 0$, so that $Q(T,P)=e^{-\epsilon(P)\over RT}\simeq 1$, and, by Lemma \ref{nernst2}, that $E-E^{\circ}=-{RTln(Q)\over 2F}-{\epsilon(P)\over 2F}={\epsilon(P)\over 2F}-{\epsilon(P)\over 2F}=0$. For the converse, we have by Lemma \ref{gibbs}, using the fact that $Q(T,P)=e^{-\epsilon(P)\over RT}\simeq 1$;\\

$({\partial G\over \partial \xi})_{T,P}=\Delta G^{\circ}-\epsilon(P)+\epsilon(P)=\Delta G^{\circ}$\\

and, if chemical equilibrium exists at $(T,P^{\circ})$, then, as $Q(T,P^{\circ})=e^{-\epsilon(P^{\circ})\over RT}$ we have that;\\

$({\partial G\over \partial \xi})_{T,P^{\circ}}=\Delta G^{\circ}+RTln(Q(T,P^{\circ})+\epsilon(P^{\circ})=\Delta G^{\circ}=0$\\

so that $({\partial G\over \partial \xi})_{T,P}=0$\\

For the penultimate claim, in one direction, use Lemma \ref{gibbs}, together with the fact that $({\partial G\over \partial \xi})_{T,P}=0$ and rearrange, the converse is also clear, applying $ln$.\\

For the final claim, we have, by the definition of activities, that;\\

$\mu_{i}=\mu_{i}^{\circ}+RTln(a_{i})+\gamma_{i}(P)$\\

so that $RTln(a_{i}(T,P^{\circ}))=-\gamma_{i}(P^{\circ})$. Now use the definition of $Q$ in Definition \ref{constants}.\\

\end{proof}

\begin{lemma}
\label{van't Hoff,Gibbs-Helmholtz2}

Along a chemical equilibrium path, we have that;\\

$ln({Q(T_{2})\over Q(T_{1})})={1\over R}\int_{T_{1}}^{T_{2}}{\Delta H^{\circ}\over T^{2}}dT++{1\over R}({\epsilon(P(T_{1}))\over T_{1}}-{\epsilon(P(T_{2}))\over T_{2}})$\\

${\Delta G^{\circ}(T_{2})\over T_{2}}-{\Delta G^{\circ}(T_{1})\over T_{1}}=-\int_{T_{1}}^{T_{2}}{\Delta H^{\circ}\over T^{2}}dT$\\

In particularly, if $\Delta H^{\circ}$ is temperature independent;\\

$ln({Q(T_{2})\over Q(T_{1})})=-{\Delta H^{\circ}\over R}({1\over T_{2}}-{1\over T_{1}})+{1\over R}({\epsilon(P(T_{1}))\over T_{1}}-{\epsilon(P(T_{2}))\over T_{2}})$\\

${\Delta G^{\circ}(T_{2})\over T_{2}}-{\Delta G^{\circ}(T_{1})\over T_{1}}=\Delta H^{\circ}({1\over T_{2}}-{1\over T_{1}})$\\

$\Delta G^{\circ}(T_{1})={T_{1}\over T_{2}}\Delta G^{\circ}(T_{2})-({T_{1}\over T_{2}}-1)\Delta H^{\circ}$\\

For $c\in\mathcal{R}$, Let $D_{c}$ intersect the line $P=P^{\circ}$ at $(T_{1},P^{\circ})$, then, for $(T_{2},P)\in D_{c}$, we have that;\\

$Q(T_{2},P)=e^{\Delta G^{\circ}(T_{1})-\Delta G^{\circ}(T_{2})-\epsilon(P(T_{2}))\over RT_{2}}$ $(\dag\dag\dag)$\\

$c=\Delta G^{\circ}(T_{1})$\\

$ln({Q(T_{2})\over Q(T_{1})})=ln(Q(T_{2}))={1\over R}\int_{T_{1}}^{T_{2}}{\Delta H^{\circ}-c\over T^{2}}dT-({\epsilon(P(T_{2}))\over RT_{2}}-{\epsilon(P(T_{1}))\over RT_{1}})$\\

${\Delta G^{\circ}(T_{2})-\Delta G^{\circ}(T_{1})\over T_{2}}=-\int_{T_{1}}^{T_{2}}{\Delta H^{\circ}-c\over T^{2}}dT$\\

and if $\Delta H^{\circ}$ is temperature independent;\\

$ln({Q(T_{2})\over Q(T_{1})})=ln(Q(T_{2}))=-({\Delta H^{\circ}-c\over R})({1\over T_{2}}-{1\over T_{1}})-({\epsilon(P(T_{2}))\over RT_{2}}-{\epsilon(P(T_{1}))\over RT_{1}})$ $(\dag\dag)$\\

${\Delta G^{\circ}(T_{2})-\Delta G^{\circ}(T_{1})\over T_{2}}=(\Delta H^{\circ}-c)({1\over T_{2}}-{1\over T_{1}})$ $(\dag)$\\

$\Delta G^{\circ}(T_{1})={T_{1}\over T_{2}}\Delta G^{\circ}(T_{2})-\Delta H^{\circ}({T_{1}\over T_{2}}-1)$ $(\dag\dag\dag\dag)$\\
\end{lemma}

\begin{proof}
By Lemma \ref{equivalences}, we have that;\\

$\Delta G^{\circ}=\sum_{i=1}^{c}\nu_{i}\mu_{i}^{\circ}$\\

so that differentiating with respect to $T$;\\

${d(\Delta G^{\circ})\over dT}=\sum_{i=1}^{c}\nu_{i}{d\mu_{i}^{\circ}\over dT}$\\

$=\sum_{i=1}^{c}\nu_{i}({\partial \mu_{i}^{\circ}\over \partial T})_{P,n}$\\

By Euler reciprocity, we have that;\\

$({\partial \mu_{i}^{\circ}\over \partial T})_{P,n}=-({\partial S^{\circ}\over \partial n_{i}})_{T,P,n'}=-\overline{S}^{\circ}_{i}$\\

so that, noting $\overline{S}^{\circ}_{i}$ is independent of $n_{i}$, so we can replace $\overline{S}^{\circ}_{i}$ by $\overline{S}^{\circ}_{m,i}$, the absolute molar entropy of substance $i$, and using thermodynamic arguments;\\

${d(\Delta G^{\circ})\over dT}=-\sum_{i=1}^{c}\nu_{i}\overline{S}^{\circ}_{i}=$\\

$=-\sum_{i=1}^{c}\nu_{i}\overline{S}^{\circ}_{m,i}$\\

$=-\Delta S^{\circ}$ $(*)$\\

Using the product rule, $(*)$ and the definition of enthalpy, we have that;\\

${d\over dT}({\Delta G^{\circ}\over T})={1\over T}{d(\Delta G^{\circ})\over dT}-{1\over T^{2}}\Delta G^{\circ}$\\

$=-{\Delta S^{\circ}\over T}-{\Delta G^{\circ}\over T^{2}}$\\

$=-{\Delta (ST+G)^{\circ}\over T^{2}}$\\

$=-{\Delta H^{\circ}\over T^{2}}$ $(**)$\\

By Lemma \ref{equivalences2}, along a chemical equilibrium path, we have that $Q=e^{-\Delta G^{\circ}-\epsilon(P)\over RT}$, so that $ln(Q)={-\Delta G^{\circ}-\epsilon(P)\over RT}$. It follows from $(**)$ that;\\

${d ln(Q)\over dT}={d\over dT}({-\Delta G^{\circ}\over RT})-{d\over dT}({\epsilon(P)\over RT})$\\

$={\Delta H^{\circ}\over RT^{2}}-{d\over dT}({\epsilon(P)\over RT})$\\

It follows, integrating between $T_{1}$ and $T_{2}$, and using the fundamental theorem of calculus, that;\\

$ln({Q(T_{2})\over Q(T_{1})})=ln(Q)(T_{2})-ln(Q)(T_{1})$\\

$=\int_{T_{1}}^{T_{2}}{d ln(Q)\over dT}dT$\\

$={1\over R}\int_{T_{1}}^{T_{2}}[{\Delta H^{\circ}\over T^{2}}-{d\over dT}({\epsilon(P)\over RT})]dT$\\

$={-\Delta H^{\circ}\over RT_{2}}+{\Delta H^{\circ}\over RT_{1}}-({\epsilon(P(T_{2}))\over RT_{2}}-{\epsilon(P(T_{1}))\over RT_{1}})$ $(P)$\\

so that, rearranging, we obtain the first claim. Using the fact, by Lemma \ref{gibbs2}, that;\\

$ln(Q(T_{2}))={-\Delta G^{\circ}(T_{2})-\epsilon(P(T_{2}))\over RT_{2}}$\\

$ln(Q(T_{1}))={-\Delta G^{\circ}(T_{1})-\epsilon(P(T_{1}))\over RT_{1}}$\\

we obtain, substituting into $(P)$, canceling $R$, and performing the integration, if $\Delta H^{\circ}$ is temperature independent, that;\\

${-\Delta G^{\circ}(T_{2})-\epsilon(P(T_{2}))\over RT_{2}}-{-\Delta G^{\circ}(T_{1})-\epsilon(P(T_{1}))\over RT_{1}}={-\Delta H^{\circ}\over RT_{2}}+{\Delta H^{\circ}\over RT_{1}}-({\epsilon(P(T_{2}))\over RT_{2}}-{\epsilon(P(T_{1}))\over RT_{1}})$\\

so that;\\

${\Delta G^{\circ}(T_{2})\over T_{2}}-{\Delta G^{\circ}(T_{1})\over T_{1}}=\Delta H^{\circ}({1\over T_{2}}-{1\over T_{1}})$ $(Q)$\\

For the fifth claim, rearrange $(Q)$. If $D_{c}$ intersects the line $P=P^{\circ}$ at $(T_{1},P^{\circ})$, for the sixth $(\dag\dag\dag)$ and seventh claims, we have, using Lemma \ref{gibbs} and the fact from Lemma \ref{equivalences2} that $Q(T_{1},P^{\circ})=e^{-\delta\over RT_{1}}$;\\

$({\partial G\over \partial \xi})_{T_{2},P}=\Delta G^{\circ}(T_{2})+RT_{2}ln(Q(T_{2},P))+\epsilon(P(T_{2}))$\\

$=({\partial G\over \partial \xi})_{T_{1},P^{\circ}}$\\

$=\Delta G^{\circ}(T_{1})+RT_{1}ln(Q(T_{1},P^{\circ}))+\epsilon(P(T_{1}))$\\

$=\Delta G^{\circ}(T_{1})+RT_{1}ln(Q(T_{1},P^{\circ}))+\epsilon(P^{\circ})$\\

$=\Delta G^{\circ}(T_{1})-\delta+\epsilon(P^{\circ})$\\

$=\Delta G^{\circ}(T_{1})=c$\\

so that, again rearranging, we obtain the result. Along $D_{c}$, we have, using Lemma \ref{gibbs}, that;\\

$ln(Q(T))={c-\Delta G^{\circ}(T)-\epsilon(P(T))\over RT}$\\

so that, using the first part;\\

${d ln(Q)\over dT}={d\over dT}({c-\Delta G^{\circ}(T)-\epsilon(P(T))\over RT})$\\

$={-c\over RT^{2}}+{d\over dT}({-\Delta G^{\circ}(T)\over RT})-{d\over dT}({\epsilon(P(T))\over RT})$\\

$={\Delta H^{\circ}-c\over RT^{2}}-{d\over dT}({\epsilon(P(T))\over RT})$\\

so that, performing the integration, using the fact that $Q(T_{1},P^{\circ})=e^{-\delta\over RT_{1}}$;\\

$ln(Q(T_{2}))-ln(Q(T_{1}))=ln(Q(T_{2}))+{\delta\over RT_{1}}={1\over R}\int_{T_{1}}^{T_{2}}[{\Delta H^{\circ}-c\over T^{2}}-{d\over dT}({\epsilon(P(T))\over RT})]dT$\\

$={1\over R}\int_{T_{1}}^{T_{2}}{\Delta H^{\circ}-c\over T^{2}}dT-({\epsilon(P(T_{2}))\over RT_{2}}-{\epsilon(P(T_{1}))\over RT_{1}})$\\

We have that, by Lemma \ref{gibbs2};\\

$ln (Q(T_{2}))={c-\Delta G^{\circ}(T_{2})-\epsilon(P(T_{2}))\over RT_{2}}$\\

$ln(Q(T_{1}))=-{\delta\over RT_{1}}$\\

so that;\\

$ln (Q(T_{2}))-ln(Q(T_{1}))={c-\Delta G^{\circ}(T_{2})-\epsilon(P(T_{2}))\over RT_{2}}+{\delta\over RT_{1}}$\\

$={\Delta G^{\circ}(T_{1})-\Delta G^{\circ}(T_{2})-\epsilon(P(T_{2}))\over RT_{2}}+{\delta\over RT_{1}}$\\

$={1\over R}\int_{T_{1}}^{T_{2}}[{\Delta H^{\circ}-c\over T^{2}}-({\epsilon(P(T_{2}))\over RT_{2}}-{\epsilon(P(T_{1}))\over RT_{1}})$\\

$={-1\over R}(\Delta H^{\circ}-c)({1\over T_{2}}-{1\over T_{1}})-({\epsilon(P(T_{2}))\over RT_{2}}-{\delta\over RT_{1}})$\\

so that, cancelling $R$ and the the error terms;\\

${\Delta G^{\circ}(T_{2})-\Delta G^{\circ}(T_{1})\over T_{2}}=(\Delta H^{\circ}-c)({1\over T_{2}}-{1\over T_{1}})$\\

$=(\Delta H^{\circ}-\Delta G^{\circ}(T_{1}))({1\over T_{2}}-{1\over T_{1}})$\\

so that, rearranging again;\\

$\Delta G^{\circ}(T_{1})({1\over T_{1}}+{1\over T_{2}}-{1\over T_{2}})={\Delta G^{\circ}(T_{1})\over T_{1}}$\\

$={\Delta G^{\circ}(T_{2})\over T_{2}}-\Delta H^{\circ}({1\over T_{2}}-{1\over T_{1}})$\\

to obtain;\\

$\Delta G^{\circ}(T_{1})={T_{1}\over T_{2}}\Delta G^{\circ}(T_{2})-\Delta H^{\circ}({T_{1}\over T_{2}}-1)$\\

\end{proof}
\begin{lemma}
\label{eqlines2}
If there exists a component $D_{c}$, $c\in\mathcal{R}$, which projects onto a closed bounded subinterval $I$ of the line $P=P^{\circ}$, not containing $0$, and intersects $P=P^{\circ}$ at $(T_{1},P^{\circ})$, with $T_{1}>0$, then, for $T_{2}\in I$, $\Delta G^{\circ}$ is linear, with;\\

$\Delta G^{\circ}(T_{2})=T_{2}({(\Delta G^{\circ}(T_{1})-\Delta H^{\circ})\over T_{1}})+\Delta H^{\circ}$\\

for $T_{2}\in I$. If $\epsilon\neq 0$, have that;\\

$({dG\over d\xi})_{T,P}=\lambda+\epsilon ln(P)+\beta T$\\

where $\{\lambda,\epsilon,\beta\}\subset \mathcal{R}$ and $\{\beta,\epsilon\}$ can be effectively determined, and we have that the activity coefficient is given by;\\

$Q(T_{1},P')=e^{{\epsilon ln({P'\over P'^{\circ}})-\epsilon(P')\over RT_{1}}}$\\

and the dynamic equilibrium paths are given by;\\

$({P'\over P'^{\circ}})^{\epsilon\over RT_{1}}=ce^{\epsilon(P')}$\\

for $c\in\mathcal{R}_{\geq 0}$, see Definition \ref{constants}, while the quasi-chemical equilibrium paths are given by;\\

$\lambda+\epsilon ln(P')+\beta T_{1}=c$\\

for $c\in\mathcal{R}$.\\

If $\epsilon=0$;\\

$({dG\over d\xi})_{T,P}=\lambda+\beta T+\sigma ln(T)$\\

where $\{\lambda,\beta,\sigma\}\subset \mathcal{R}$, and $\{\beta,\sigma\}$ can be effectively determined. The activity coefficient $Q$ is given by;\\

$Q(T_{1},P')=e^{-\epsilon(P')\over RT_{1}}$\\

The dynamic equilibrium paths are given by;\\

$e^{-\epsilon(P')\over RT_{1}}=c$\\

for $c\in\mathcal{R}_{\geq 0}$, see Definition \ref{constants}, while the quasi-chemical equilibrium paths are given by;\\

$\lambda+\epsilon ln(P')+\beta T_{1}=c$\\

for $c\in\mathcal{R}$.\\

\end{lemma}
\begin{proof}
For the first claim, by Lemma \ref{van't Hoff,Gibbs-Helmholtz2}, we have that;\\

$\Delta G^{\circ}(T_{2})={T_{2}\over T_{1}}\Delta G^{\circ}(T_{1})-\Delta H^{\circ}({T_{2}\over T_{1}}-1)$\\

$=T_{2}({(\Delta G^{\circ}(T_{1})-\Delta H^{\circ})\over T_{1}})+\Delta H^{\circ}$\\

For the next claim, by Lemma \ref{equivalences2} and the proof of Lemma \ref{eqlines}, we have that;\\

$({\partial ({dG\over d\xi})_{T,P}\over \partial T})_{P}=({\partial (\sum \nu_{i}\mu_{i})\over \partial T})_{P}$\\

$=\sum \nu_{i}({\partial \mu_{i}\over \partial T})_{P}$\\

$=\sum \nu_{i}({\partial \mu_{i}\over \partial T})_{P,n}$\\

$=-\sum \nu_{i}\overline{S}_{m,i}$ $(*)$\\

Again, to compute $\overline{S}_{m,i}$, we have by the first law of thermodynamics;\\

$dQ=dU+dL=dU+pdV$\\

where $L$ is the work done by the system. We can assume that the liquid mixture is in thermal equilibrium with a mixture of ideal gases in the vapour phase, and using the ideal gas law, the definition of temperature for ideal gases, obtain the calculation of internal energy for the mixture;\\

$U(T,P,n_{1},\ldots,n_{c})=\sum_{i=1}^{c}({3\over 2}N_{A}n_{i}kT-N_{A}n_{i}m_{i}\rho_{i})$\\

where $m_{i}$ is the molecular mass of species $i$, $\rho_{i}$ is the specific latent heat of evaporation of species $i$, which we assume is independent of temperature $T$. By a result in \cite{dep1}, using the fact that entropy difference is independent of path, we have that $Q$ is independent of $P$. We then have;\\

$dU=\sum_{i=1}^{c}{3\over 2}N_{A}kTdn_{i}+\sum_{i=1}^{c}{3\over 2}N_{A}kn_{i}dT-\sum_{i=1}^{c}N_{A}m_{i}\rho_{i}dn_{i}$\\

$dQ=\sum_{i=1}^{c}{3\over 2}N_{A}kTdn_{i}+\sum_{i=1}^{c}{3\over 2}N_{A}kn_{i}dT-\sum_{i=1}^{c}N_{A}m_{i}\rho_{i}dn_{i}+dL$\\

${dQ\over T}=\sum_{i=1}^{c}{3\over 2}N_{A}kdn_{i}+\sum_{i=1}^{c}{3\over 2}N_{A}kn_{i}{dT\over T}-\sum_{i=1}^{c}N_{A}m_{i}\rho_{i}{dn_{i}\over T}+{g(T,\overline{n})dT\over T}$\\

$+\sum_{i=1}^{c}h_{i}(T,\overline{n}){dn_{i}\over T}$\\

$({dQ\over T})_{n',T,P}={3\over 2}N_{A}kdn_{i}-N_{A}m_{i}\rho_{i}{dn_{i}\over T}+h_{i}(T,\overline{n}){dn_{i}\over T}$\\

It follows that;\\

$\overline{S}_{m,i}=\int_{\Delta n_{i}=1}({dQ\over T})_{n',T,P}={3\over 2}N_{A}k-{N_{A}m_{i}\rho_{i}\over T}+{k_{i}(T)\over T}$ $(**)$\\

So that, from $(*)$;\\

$({\partial ({dG\over d\xi})_{T,P}\over \partial T})_{P}=-\sum_{i=1}^{c} \nu_{i}({3\over 2}N_{A}k-{N_{A}m_{i}\rho_{i}\over T})-\sum_{i=1}^{c}\nu_{i}{k_{i}(T)\over T}$\\

$=-{3\over 2}N_{A}k(\sum_{i=1}^{c} \nu_{i})+{N_{A}\over T}\sum_{i=1}^{c}\nu_{i}\mu_{i}\rho_{i}-\sum_{i=1}^{c}\nu_{i}{k_{i}(T)\over T}$\\

$=-{3\over 2}N_{A}k(\sum_{i=1}^{c} \nu_{i})+{N_{A}\over T}\sum_{i=1}^{c}\nu_{i}\mu_{i}\rho_{i}-{G(T)\over T}$ $(***)$\\

From $(***)$, which is uniform $P$, we see that $({dG\over d\xi})_{T,P}$ is of the form $\alpha(P)+\beta T+\gamma ln(T)-\int {G(T)\over T}$, $(B)$, where $\{\beta,\gamma\}\subset\mathcal{R}$, and, assuming that $({dG\over d\xi})_{T,P}$ is differentiable, $\alpha \in C^{1}(\mathcal{R})$. By a similar calculation, we have that;\\

$({\partial ({dG\over d\xi})_{T,P}\over \partial P})_{T}=({\partial (\sum \nu_{i}\mu_{i})\over \partial P})_{T}$\\

$=\sum_{i=1}^{c} \nu_{i}({\partial \mu_{i}\over \partial P})_{T}$\\

$=\sum_{i=1}^{c} \nu_{i}({\partial \mu_{i}\over \partial P})_{T,n}$\\

$=\sum_{i=1}^{c} \nu_{i}({\partial V\over \partial n_{i}})_{T,P,n'}$\\

$=\sum_{i=1}^{c} \nu_{i}\overline{V}_{i}$\\

$=\sum_{i=1}^{c}\nu_{i}{N_{A}m_{i}\over \kappa_{i}(T,P)}$ $(A)$\\

where $\kappa_{i}$ is the density of substance $i$. We also have that;\\

$P(\sum_{i=1}^{c}\nu_{i}{N_{A}m_{i}\over \kappa_{i}(T,P)})=P(\sum_{i=1}^{c} \nu_{i}\overline{V}_{i})=G(T)$ $(dL=PdV)$ $(C)$\\

and from $(A),(B),(C)$;\\

$P({\partial ({dG\over d\xi})_{T,P}\over \partial P})_{T}=G(T)$\\

$=P\alpha'(P)$\\

so that $G(T)=\epsilon$, $\alpha(P)=\lambda+\epsilon ln(P)$\\

$({dG\over d\xi})_{T,P}$ is of the form;\\

$\alpha(P)+\beta T+\gamma ln(T)-\int {G(T)\over T}$\\

$=\lambda+\epsilon ln(P)+\beta T+\gamma ln(T)-\epsilon ln(T)$\\

$=\lambda+\epsilon ln(P)+\beta T+\sigma ln(T)$ $(D)$\\

where $\sigma=\gamma-\epsilon$, $\{\beta,\epsilon,\lambda,\sigma\}\subset\mathcal{R}$.\\

If $\epsilon=0$, then $({dG\over d\xi})_{T,P}$ is independent of $P$, and the components $D_{c}$ are all straight line paths. In this case, if $D_{c}$ intersects the line $P=P^{\circ}$ at $(T_{1},P^{\circ})$, then, for all $P>0$;\\

$c=\Delta G^{\circ}(T_{1})+RT_{1}ln(Q(T_{1},P)+\epsilon(P)$\\

$=\Delta G^{\circ}(T_{1})$\\

implies that $RT_{1}ln(Q(T_{1},P)=-\epsilon(P)$, so that $Q(T_{1},P)=e^{-\epsilon(P)\over RT_{1}}$, $(X)$. From $(D)$, we have that;\\

$({dG\over d\xi})_{T,P}=\lambda+\beta T+\sigma ln(T)$, $(Y)$\\

The calculation of the dynamical and chemical equilibrium paths follows easily, from the equations $Q=c$, for $c\in\mathcal{R}_{\geq 0}$ and $({dG\over d\xi})_{T,P}=c$, for $c\in\mathcal{R}$, using $(X),(Y)$.\\

If $\epsilon\neq 0$, for any $c\in\mathcal{R}$, we can solve the equation;\\

$\lambda+\epsilon ln(P)+\beta T+\sigma ln(T)=c$\\

for any given $T>0$ and an appropriate choice of $P(T)$. In particularly, there exists a component $D_{c}$ projecting onto the line $P=P^{0}$. Calculating limits at $\{+\infty,-\infty\}$, we have that for $\beta>0,\sigma>0$ or $\beta<0,\sigma<0$, we can solve the equation;\\

$\lambda+\epsilon ln(P^{\circ})+\beta T+\sigma ln(T)=c$ $(Z)$\\

for $T$. If $\beta>0,\sigma<0$ or $\beta<0,\sigma>0$, observing that $(\beta T+\sigma ln(T))'=\beta+{\sigma\over T}$, $(\beta T+\sigma ln(T))''=-{\sigma\over T^{2}}$, so there exists a min/max at $T={-\sigma\over \beta}$, we have that, if;\\

$-\sigma+\sigma ln({-\sigma\over \beta})\leq c-\lambda-\epsilon ln(ln(P^{\circ}))$\\

$-\sigma+\sigma ln({-\sigma\over \beta})\geq c-\lambda-\epsilon ln(ln(P^{\circ}))$\\

we can again solve the equation $(Z)$ for $T$, so that, for an appropriate choice of $c$, there exists an intersection of the component $D_{c}$ with the line $P=P^{\circ}$.\\

By the first part, $\Delta G^{\circ}$ is linear, with;\\

$\Delta G^{\circ}(T_{2})=T_{2}({(\Delta G^{\circ}(T_{1})-\Delta H^{\circ})\over T_{1}})+\Delta H^{\circ}$\\

for an intersection at $(T_{1},P^{\circ})$. We also have, using $(D)$, that;\\

$\Delta G^{\circ}(T_{2})=({\partial G\over \partial \xi})_{T,P}(T_{2},P^{\circ})$\\

$=\lambda+\epsilon ln(P^{\circ})+\beta T_{2}+\sigma ln(T_{2})$\\

so that, equating coefficients;\\

$\sigma=0$\\

$\lambda+\epsilon ln(P^{\circ})=\Delta H^{\circ}$\\

$\beta={(\Delta G^{\circ}(T_{1})-\Delta H^{\circ})\over T_{1}}$\\

$\Delta G^{\circ}(T_{1})=\beta T_{1}+\Delta H^{\circ}$\\

We can then, using Lemma \ref{gibbs2}, obtain a formula for the activity coefficient;\\

$Q(T_{1},P')=e^{{(({\partial G\over \partial \xi})_{T,P}|_{T_{1},P'}-\Delta G^{\circ}(T_{1}))-\epsilon(P')\over RT_{1}}}$\\

$=e^{{(\Delta H^{\circ}-\epsilon ln(P'^{\circ})+\epsilon ln(P')+\beta T_{1}-(\beta T_{1}+\Delta H^{\circ}))-\epsilon(P')\over RT_{1}}}$\\

$=e^{{\epsilon ln({P'\over P'^{\circ}})-\epsilon(P')\over RT_{1}}}$ $(W)$\\

as required. The claim about the coefficients being determined is clear from the proof. The determination of the dynamical and quasi-chemical equilibrium lines, see Definitions \ref{constants} and Lemma \ref{feasible}, follows from a simple rearrangement of the formulas $Q(T_{1},P')=c$, for some $c\in\mathcal{R}_{\geq 0}$, using $(W)$, and $({dG\over d\xi})_{T,P}=c$, for some $c\in \mathcal{R}$, using $(D)$, with $\sigma=0$.

\end{proof}

\begin{lemma}
\label{rates}
Let notation be as in Lemma \ref{eqlines}, then if $\epsilon\neq 0$, with $Q(T,P)=({P\over P^{\circ}})^{\epsilon\over RT}e^{-\epsilon(P)\over RT}$, then, using the definition of $grad$ in \cite{BK};\\

$grad(Q)(T,P)=({-\epsilon ln({P\over P^{\circ}})+\epsilon(P)\over RT^{2}}({P\over P^{\circ}})^{\epsilon\over RT}e^{-\epsilon(P)\over RT},{{\epsilon\over P}-\epsilon'(P)\over RT}({P\over P^{\circ}})^{\epsilon\over RT}e^{-\epsilon(P)\over RT})$\\

In particular the paths of maximal reaction, for the region $|grad(Q)(T,P)|>1$, $Q(T,P)>0$, are given by;\\

$\int {\epsilon P ln({P\over P^{\circ}})-P\epsilon(P)\over P\epsilon'(P)-\epsilon}dP=-{T^{2}\over 2}+c$\\

for $c\in\mathcal{R}$\\

If $\epsilon(P)=0$;\\

$grad(Q)(T,P)=({-\epsilon ln({P\over P^{\circ}})\over RT^{2}}({P\over P^{\circ}})^{\epsilon\over RT},{\epsilon\over RTP}({P\over P^{\circ}})^{\epsilon\over RT})$\\

and the paths of maximal reaction, for the region $|grad(Q)(T,P)|>1$, $Q(T,P)>0$, are given by;\\

$P^{2}({ln(P)\over 2}-{ln(P^{\circ})\over 2}-{1\over 4})+{T^{2}\over 2}=c$\\

for $c\in\mathcal{R}$. If $\epsilon=0$, with with $Q(T,P)=e^{-\epsilon(P)\over RT}$, then;\\

$grad(Q)(T,P)=({\epsilon(P)\over RT^{2}}e^{-\epsilon(P)\over RT},-{\epsilon'(P)\over RT}e^{-\epsilon(P)\over RT})$\\

and the paths of maximal reaction, for the region $|grad(Q)(T,P)|>1$, $Q(T,P)>0$, are given by;\\

$\int {\epsilon(P)\over \epsilon'(P)}dP={-T^{2}\over 2}+c$\\

for $c\in\mathcal{R}$.\\
\end{lemma}

\begin{proof}
The determination of $grad(Q)(T,P)=({\partial Q\over \partial T},{\partial Q\over \partial P})$ is a simple application of the chain rule and the formula for $Q$. By the definition of the extent $\xi$ of a reaction, see Definition \ref{constants}, we have that, for $1\leq i\leq c$;\\

$n_{i}(t)=\nu_{i}\xi(t)+n_{i,0}$\\

$n(t)=\sum_{j=1}^{c}n_{i}(t)$\\

$=\sum_{i=1}^{c}(\nu_{i}\xi(t)+n_{i,0})$\\

$=\alpha\xi(t)+\beta$\\

where $\alpha=\sum_{i=1}^{c}\nu_{i}$ and $\beta=\sum_{i=1}^{c}n_{i,0}$\\

so that;\\

$x_{i}(t)={n_{i}(t)\over n(t)}={\nu_{i}\xi(t)+n_{i,0}\over \alpha\xi(t)+\beta}$\\

It follows that for a feasible path $\gamma$;\\

$\prod_{i=1}^{c}x_{i}^{\nu_{i}}(t)=Q(\gamma_{12}(t))$\\

$=\prod_{i=1}^{c}({\nu_{i}\xi(t)+n_{i,0}\over \alpha\xi(t)+\beta})^{\nu_{i}}$\\

$={\prod_{i=1}^{c}(\nu_{i}\xi(t)+n_{i,0})^{\nu_{i}}\over (\alpha\xi(t)+\beta)^{c}}$\\

$=G_{\gamma}(\xi(t))$\\

where $G_{\gamma}(x)={\prod_{i=1}^{c}(\nu_{i}x+n_{i,0})^{\nu_{i}}\over (\alpha x+\beta)^{c}}$\\

We have that $\xi(0)=0$, and, as we can assume that $\beta>0$, we have that;\\

$G_{\gamma}'(0)={\sum_{i=1}^{c}\nu_{i}(\prod_{j\neq i}n_{j,0})\over \beta^{c}}-{c\prod_{j=1}^{c}n_{j,0}\over \beta^{c+1}}$\\

$={1\over \beta^{c+1}}(\beta(\sum_{i=1}^{c}\nu_{i}(\prod_{j\neq i}n_{j,0})-c\prod_{j=1}^{c}n_{j,0}))$\\

$={\prod_{j=1}^{c}n_{j,0}\over \beta^{c+1}}(\beta{\sum_{i=1}^{c}\nu_{i}\over n_{i,0}}-c)$\\

$={\prod_{j=1}^{c}n_{j,0}\over \beta^{c+1}}(n_{init}{\sum_{i=1}^{c}\nu_{i}\over n_{i,0}}-c)$\\

$={\prod_{j=1}^{c}n_{j,0}\over \beta^{c+1}}(\sum_{i=1}^{c}{\nu_{i}\over x_{i,init}}-c)$\\

$={\prod_{j=1}^{c}n_{j,0}\over \beta^{c+1}}((\sum_{i=1}^{c}\nu_{i}log(x_{i}))'_{init}-c)$\\

$={\prod_{j=1}^{c}n_{j,0}\over \beta^{c+1}}(log(\prod_{i=1}^{c}x_{i}^{\nu_{i}})'_{init}-c)$\\

$={\prod_{j=1}^{c}n_{j,0}\over \beta^{c+1}}(log(Q(\gamma_{12}(t)))'|_{0}-c)$\\

$={\prod_{j=1}^{c}n_{j,0}\over \beta^{c+1}}({grad(Q)\centerdot \gamma_{12}'(0)\over Q(\gamma_{12}(0))}-c)$\\

so that;\\

$G_{\gamma}'(0)=0$ iff ${grad(Q)\centerdot \gamma_{12}'(0)\over Q(\gamma_{12}(0))}=c$\\

which we can exclude by an appropriate parametrisation of the feasible path $\gamma$, without altering the direction of $\gamma_{12}'(0)$. By the inverse function theorem, we can invert $G_{\gamma}$ locally, to obtain that $\xi(t)=(G_{\gamma}^{-1}\circ Q)(\gamma_{12}(t))$. Then;\\

$\xi'(0)=(G_{\gamma}^{-1})'|_{Q(T_{0},P_{0})}grad(Q)(T_{0},P_{0})\centerdot \gamma_{12}'(0)$\\

$={grad(Q)(T_{0},P_{0})\centerdot \gamma_{12}'(0)\over G_{\gamma}'(0)}$\\

$=grad(Q)(T_{0},P_{0})\centerdot \gamma_{12}'(0){\beta^{c+1}\over \prod_{j=1}^{c}n_{j,0}}{Q(\gamma_{12}(0))\over grad(Q)\centerdot \gamma_{12}'(0)-cQ(\gamma_{12}(0))}$\\

$={\alpha_{1}\beta_{1}[grad(Q)(T_{0},P_{0})\centerdot \gamma_{12}'(0)]\over grad(Q)(T_{0},P_{0})\centerdot \gamma_{12}'(0)-c\beta_{1}}$\\

where $\gamma_{12}(0)=(T_{0},P_{0})$, $\alpha_{1}(T_{0},P_{0})={\beta^{c+1}\over \prod_{j=1}^{c}n_{j,0}}$, $\beta_{1}(T_{0},P_{0})=Q(\gamma_{12}(0))$.\\

Writing $\gamma_{12}'(0)=\lambda(cos(\theta),sin(\theta))$, we have that;\\

$\xi'(0)={\lambda\alpha_{1}\beta_{1}[{\partial Q\over \partial T}|_{(T_{0},P_{0})}cos(\theta)+{\partial Q\over \partial T}|_{(T_{0},P_{0})}sin(\theta)]\over \lambda[{\partial Q\over \partial T}|_{(T_{0},P_{0})}cos(\theta)+{\partial Q\over \partial T}|_{(T_{0},P_{0})}sin(\theta)]-c\beta_{1}}=h(\lambda,\theta)={\lambda\alpha_{1}\beta_{1}r(\theta)\over \lambda r(\theta)-c\beta_{1}}$\\

where $r(\theta)={\partial Q\over \partial T}|_{(T_{0},P_{0})}cos(\theta)+{\partial Q\over \partial T}|_{(T_{0},P_{0})}sin(\theta)$. We have that;\\

${\partial h\over \partial \lambda}={\alpha_{1}\beta_{1}r(\theta)\over \lambda r(\theta)-c\beta_{1}}-{\lambda\alpha_{1}\beta_{1}r^{2}(\theta)\over (\lambda r(\theta)-c\beta_{1})^{2}}$\\

so that;\\

${\partial h\over \partial \lambda}=0$ iff $\alpha_{1}\beta_{1}r(\theta)(\lambda r(\theta)-c\beta_{1})-\lambda\alpha_{1}\beta_{1}r^{2}(\theta)=0$\\

iff $-c\alpha_{1}\beta_{1}^{2}r(\theta)=0$\\

so that, as $\alpha_{1}\neq 0$, $Q(T_{0},P_{0})\neq 0$, and $(cos(\theta),sin(\theta))$ is not tangent to the dynamic equilibrium path at $(T_{0},P_{0})$, then $h(\lambda,\theta)$ is monotonic in $\lambda$.\\

${\partial h\over \partial \theta}={\lambda \alpha_{1}\beta_{1}r'(\theta)\over \lambda r(\theta)-c\beta_{1}}-{\lambda^{2}\alpha_{1}\beta_{1}r'(\theta)\over (\lambda r(\theta)-c\beta_{1})^{2}}$\\

so that;\\

${\partial h\over \partial \theta}=0$ iff $\lambda \alpha_{1}\beta_{1}r'(\theta)(\lambda r(\theta)-c\beta_{1})-\lambda^{2}\alpha_{1}\beta_{1}r'(\theta)=0$\\

iff $\lambda \alpha_{1}\beta_{1}(\lambda r(\theta)-c\beta_{1})-\lambda^{2}\alpha_{1}\beta_{1}=0$\\

iff $r(\theta)={\lambda^{2}\alpha_{1}\beta_{1}+c\lambda \alpha_{1}\beta_{1}^{2}\over \lambda^{2} \alpha_{1}\beta_{1}}$\\

$=1+{c\beta_{1}\over \lambda}$\\

If $|1+{c\beta_{1}\over \lambda}|\leq |grad(Q)(T_{0},P_{0})|$, and $|grad(Q)(T_{0},P_{0})|>1$, we can solve $r(\theta)=1+{c\beta_{1}\over \lambda}$, for $\lambda>0$, so that, when $r(\theta(\lambda))=1+{c\beta_{1}\over \lambda}$;\\

$h(\lambda,\theta(\lambda))={\lambda \alpha_{1}\beta_{1}(1+{c\beta_{1}\over \lambda})\over \lambda(1+{c\beta_{1}\over \lambda})-c\beta_{1}}$\\

$={\lambda\alpha_{1}\beta_{1}+c\alpha_{1}\beta_{1}^{2}\over \lambda}$\\

$=\alpha_{1}\beta_{1}+{c\alpha_{1}\beta_{1}^{2}\over \lambda}$\\

$=\alpha_{1}\beta_{1}(1+{c\beta_{1}\over \lambda})$\\

so a maximum/minimum occurs when $|1+{c\beta_{1}\over \lambda}|=|grad(Q)(T_{0},P_{0})|$, in which case $(cos(\theta),sin(\theta))$ is parallel to $grad(Q)(T_{0},P_{0})$, and perpendicular to the tangent of the level curve of $Q$, through $(T_{0},P_{0})$. We, therefore, have to solve the paired differential equation;\\

${dT\over dt}={-\epsilon ln({P(t)\over P^{\circ}})+\epsilon(P(t))\over RT(t)^{2}}({P(t)\over P^{\circ}})^{\epsilon\over RT(t)}e^{-\epsilon(P(t))\over RT(t)}$\\

${dP\over dt}={{\epsilon\over P(t)}-\epsilon'(P(t))\over RT(t)}({P(t)\over P^{\circ}})^{\epsilon\over RT(t)}e^{-\epsilon(P(t))\over RT(t)}$\\

so that;\\

${dP\over dT}={{dP\over dt}\over {dT\over dt}}$\\

$={{{\epsilon\over P}-\epsilon'(P)\over RT}({P\over P^{\circ}})^{\epsilon\over RT}e^{-\epsilon(P)\over RT}\over {-\epsilon ln({P\over P^{\circ}})+\epsilon(P)\over RT^{2}}({P\over P^{\circ}})^{\epsilon\over RT}e^{-\epsilon(P)\over RT}}$\\

$={T({\epsilon\over P}-\epsilon'(P))\over -\epsilon ln({P\over P^{\circ}})+\epsilon(P)}$\\

Separating variables, we obtain that;\\

${\epsilon ln({P\over P^{\circ}})-\epsilon(P)\over \epsilon'(P)-{\epsilon\over P}}dP=-TdT$\\

which has an implicit solution given by;\\

$\int {\epsilon P ln({P\over P^{\circ}})-P\epsilon(P)\over P\epsilon'(P)-\epsilon}dP=-{T^{2}\over 2}+c$\\

for $c\in\mathcal{R}$. By a result due to \cite{BdiP}, we have that these implicit solutions are integral curves for $grad(Q)$ as required. If $\epsilon(P)=0$, then $\epsilon'(P)=0$ and the implicit solutions are given by;\\

$\int P ln({P\over P^{\circ}})dP=-{T^{2}\over 2}+c$\\

for $c\in\mathcal{R}$. We have, integrating by parts, that;\\

$\int P ln({P\over P^{\circ}})=\int Pln(P)-Pln(P^{\circ})$\\

$={P^{2}ln(P)\over 2}-\int {P\over 2}-{P^{2}ln(P^{\circ})\over 2}$\\

$={P^{2}ln(P)\over 2}-{P^{2}\over 4}-{P^{2}ln(P^{\circ})\over 2}$\\

$=P^{2}({ln(P)\over 2}-{ln(P^{\circ})\over 2}-{1\over 4})$\\

so that the implicit solutions are given by;\\

$P^{2}({ln(P)\over 2}-{ln(P^{\circ})\over 2}-{1\over 4})+{T^{2}\over 2}=c$\\

for $c\in\mathcal{R}$ as required. The determination of $grad(Q)(T,P)$ when $\epsilon=0$ is again a simple application of the chain rule. As before, we compute;\\

${dP\over dT}={{-\epsilon'(P)\over RT}e^{-\epsilon(P)\over RT}\over {\epsilon(P)\over RT^{2}}e^{-\epsilon(P)\over RT}}$\\

$={-\epsilon'(P)T\over \epsilon(P)}$\\

so that, separating variables;\\

${\epsilon(P)\over \epsilon'(P)}dP=-TdT$\\

and the implicit solutions are given by;\\

$\int {\epsilon(P)\over \epsilon'(P)}dP={-T^{2}\over 2}+c$\\

\end{proof}
\begin{rmk}
\label{atmospheric}
The fact that, in the case $\epsilon(P)=0$, the paths of maximal reaction depend on an arbitrary choice of $P^{\circ}$ suggest that some approximation is needed in the formula;\\

$\mu_{i}(T,P)=\mu_{i}^{\circ}(T)+RTlog(x_{i}(T,P))$\\

for ideal solutions. Of course, once $P^{\circ}$ is fixed, $\epsilon(P)$ depends on this choice of $P^{\circ}$ as well.
\end{rmk}
\end{section}
\begin{section}{Electrochemistry with error terms and ideal solution}
\label{idealelectrochemistry}

We consider the reaction $H_{2}(g)+2AgCl(s)+2e^{-}(R)\rightarrow 2HCl+2Ag(s)+2e^{-}(L)$, for the standard cell, even though the uncharged species probably don't form an ideal solution. The reader can easily reformulate the results in the context of an ideal solution, by just changing the electron count, see Section \ref{electrochemistry}.

\begin{lemma}{The Nernst Equation for the Standard Cell}\\
\label{nernst2}\\

At electrical chemical equilibrium $(T,P)$ and $(T,P^{\circ})$;\\

$(E-E^{\circ})(T,P)=-{RTln(Q(T,P))\over 2F}-{\epsilon(P)\over 2F}$\\

\end{lemma}
\begin{proof}
For $c$ substances, with $c'$ the number of the charged species,using Definition \ref{constants}, we have that the electrostatic potential energy;\\

$U_{el}= \sum_{i=1}^{c'}\phi(\overline{x}_{i})q_{i}$, where $q_{i}=N_{i}ez_{i}=N_{A}n_{i}ez_{i}$\\

where $\{\overline{x}_{i}:1\leq i\leq c'\}$ are the positions of the charged species, $N_{i}$ is the number of particles at $\overline{x}_{i}$.\\

 We have that;\\

$U=U_{chem}+U_{el}$, so that;\\

$G(T,P,n_{1},\ldots,n_{c})=U+PV-TS$\\

$=U_{chem}+U_{el}+PV-TS$\\

$=U_{el}+G_{chem}$\\

$=\sum_{j=1}^{c}\phi(\overline{x}_{j})q_{j}+G_{chem}$\\

$=\sum_{j=1}^{c}\phi(\overline{x}_{j})N_{A}n_{j}ez_{j}+G_{chem}$\\

so that;\\

$\mu_{i}=({\partial G\over \partial n_{i}})_{T,P}$\\

$=({\partial (\sum_{j=1}^{c}\phi(\overline{x}_{j})N_{A}n_{j}ez_{j}+G_{chem})\over \partial n_{i}})_{T,P}$\\

$=\mu_{i,chem}$, $(c'+1\leq i\leq c)$\\

$=\mu_{i,chem}+{\partial (\phi(\overline{x}_{i})N_{A}n_{i}ez_{i})\over \partial n_{i}}$, $(1\leq i\leq c')$\\

$=\mu_{i,chem}+\phi(\overline{x}_{i})N_{A}ez_{i}$\\

$=\mu_{i,chem}+\phi(\overline{x}_{i})F z_{i}$, $(*)$\\

We consider the standard cell reaction $H_{2}(g)+2AgCl(s)+2e^{-}(R)\rightarrow 2HCl+2Ag(s)+2e^{-}(L)$. At electrical chemical equilibrium, similarly to Lemma \ref{equivalences}, generalised to a collection involving charged species, using $(*)$, we have that;\\

$({\partial G\over \partial \xi})_{T,P}={\sum}_{i=1}^{c}\nu_{i}\mu_{i}$\\

$=2\mu(HCl)+2\mu(Ag)-\mu(H_{2})-2\mu(AgCl)+2\mu(e^{-}(L))-2\mu(e^{-}(R))$\\

$=({\partial G_{chem'}\over \partial \xi})_{T,P}+2\mu(e^{-}(L))-2\mu(e^{-}(R))$\\

$=({\partial G_{chem'}\over \partial \xi})_{T,P}+((2\mu_{chem}(e^{-}(L))-2F\phi(L))-(2\mu_{chem}(e^{-}(L))-2F\phi(R)))$\\

$=({\partial G_{chem'}\over \partial \xi})_{T,P}+2F(\phi(R)-\phi(L))$\\

$=({\partial G_{chem'}\over \partial \xi})_{T,P}+2EF=0$ $(\dag)$ \\

where $G_{chem'}$ is the Gibbs energy restricted to the uncharged species. By Lemmas \ref{gibbs2} and \ref{equivalences2}, we have that;\\

$({\partial G_{chem'}\over \partial \xi})_{T,P^{\circ}}=\sum_{i=c'+1}^{c}\nu_{i}\mu_{i}^{\circ}$\\

$=(\Delta G_{chem'}^{\circ}+RTln(Q_{chem'}(T,P^{\circ}))+\epsilon(P^{\circ}))$\\

$=(\Delta G_{chem'}^{\circ}-\epsilon(P^{\circ}))+\epsilon(P^{\circ}))$\\

$=\Delta G_{chem'}^{\circ}$, $(\dag\dag)$\\

From $(\dag),(\dag\dag)$, we obtain;\\

$2E^{\circ}F=-({\partial G_{chem'}\over \partial \xi})_{T,P^{\circ}}$\\

$=-\Delta G_{chem'}^{\circ}$ $(\dag\dag\dag)$\\

Similarly, we have that;\\

$({\partial G_{chem'}\over \partial \xi})_{T,P}=\sum_{i=c'+1}^{c}\nu_{i}\mu_{i}$\\

$=(\Delta G_{chem'}^{\circ}+RTln(Q_{chem'}(T,P))+\epsilon(P))$, $(\dag\dag\dag\dag)$\\

so from $(\dag),(\dag\dag\dag\dag))$, we obtain that;\\

$2EF=-({\partial G_{chem'}\over \partial \xi})_{T,P}$\\

$=-(\Delta G_{chem'}^{\circ}+RTln(Q_{chem'}(T,P))+\epsilon(P))$, $(\sharp)$\\

Combining $(\sharp), (\dag\dag\dag)$, we obtain that;\\

$2EF-2E^{\circ}F=-(\Delta G_{chem'}^{\circ}+RTln(Q_{chem'}(T,P))+\epsilon(P))-(-\Delta G_{chem'}^{\circ})$\\

$=-RTln(Q_{chem'}(T,P))-\epsilon(P)$\\

so that;\\

$E-E^{\circ}=-{RTln(Q_{chem'}(T,P))\over 2F}-{\epsilon(P)\over 2F}$\\

\end{proof}
\begin{lemma}
\label{delta2}

At electrical chemical equilibrium $(T,P)$ and $(T,P^{\circ})$, and chemical equilibrium $(T,P)$;\\

$\Delta G^{\circ}=2F(E-E^{0})$\\

\end{lemma}
\begin{proof}
By Lemma \ref{nernst2}, we have that;\\

$E-E^{\circ}=-{RTln(Q)\over 2F}-{\epsilon(P)\over 2F}$, $(*)$\\

and, by Lemma \ref{gibbs2}, we have that;\\

$0=({\partial G\over \partial \xi})_{T,P}=\Delta G^{\circ}+RTln(Q)+\epsilon(P)$, $(**)$\\

Rearranging $(*),(**)$, we obtain the result.\\

\end{proof}
\begin{lemma}
\label{allt2}
If $\epsilon=0$, we have, for all $T_{1}>0$, that;\\

$({\partial G\over \partial \xi})_{T,P}|_{(T_{1},P_{1})}=({\partial G\over \partial \xi})_{T,P}|_{(T_{1},P_{1}^{\circ})}$\\

iff;\\

$E(T_{1},P_{1})=E(T_{1},P_{1}^{\circ})=E^{\circ}(T_{1})$\\

where $G$ is the Gibbs energy function for the charged and uncharged species.
\end{lemma}
\begin{proof}
By $(\dag)$ of Lemma \ref{nernst2}, we have that;\\

$({\partial G\over \partial \xi})_{T,P}=({\partial G_{chem'}\over \partial \xi})_{T,P}+2EF$ $(*)$\\

By Lemma \ref{eqlines2}, we have that $({\partial G_{chem'}\over \partial \xi})_{T,P}$ is independent of $P$, in particularly, we have that;\\

$({\partial G_{chem'}\over \partial \xi})_{T_{1},P_{1}}=({\partial G_{chem'}\over \partial \xi})_{T_{1},P_{1}^{\circ}}$ $(**)$\\

so that, combining $(*),(**)$, we obtain the result.
\end{proof}

\begin{lemma}
\label{alltelectrical2}
We have, for all $T_{1}>0,P_{1}>0$, that;\\

$2F(E(T_{1},P_{1})-E^{\circ}(T_{1}))=({\partial G\over \partial \xi})_{T,P}|_{(T_{1},P_{1})}-({\partial G\over \partial \xi})_{T,P}|_{(T_{1},P_{1}^{\circ})}-RT_{1}ln(Q(T_{1},P_{1}))-\epsilon(P_{1})$\\

\end{lemma}
\begin{proof}
Following the proof of Lemma \ref{nernst2}, we have that;\\

$({\partial G\over \partial \xi})_{T,P}|_{T_{1},P_{1}}=({\partial G_{chem'}\over \partial \xi})_{T,P}|_{T_{1},P_{1}}+2E(T_{1},P_{1})F$ $(*)$\\

$2E^{\circ}(T_{1})F=({\partial G\over \partial \xi})_{T,P}|_{T_{1},P_{1}^{\circ}}-({\partial G_{chem'}\over \partial \xi})_{T,P}|_{T_{1},P_{1}^{\circ}}$\\

$=({\partial G\over \partial \xi})_{T,P}|_{T_{1},P_{1}^{\circ}}-\Delta G^{\circ}_{chem'}(T_{1})$ $(**)$\\

so from $(*),(**)$ and Lemma \ref{gibbs2};\\

$2E(T_{1},P_{1})F=({\partial G\over \partial \xi})_{T,P}|_{T_{1},P_{1}}-({\partial G_{chem'}\over \partial \xi})_{T,P}|_{T_{1},P_{1}}$\\

$=({\partial G\over \partial \xi})_{T,P}|_{T_{1},P_{1}}-(\Delta G^{\circ}_{chem'}(T_{1})+RT_{1}ln(Q_{chem'}(T_{1},P_{1}))+\epsilon(P_{1}))$\\

$2E(T_{1},P_{1})F-2E^{\circ}(T_{1})F=({\partial G\over \partial \xi})_{T,P}|_{T_{1},P_{1}}-(\Delta G^{\circ}_{chem'}(T_{1})+RT_{1}ln(Q_{chem'}(T_{1},P_{1}))+\epsilon(P_{1}))$\\

$-(({\partial G\over \partial \xi})_{T,P}|_{T_{1},P_{1}^{\circ}}-\Delta G^{\circ}_{chem'}(T_{1}))$\\

$=({\partial G\over \partial \xi})_{T,P}|_{T_{1},P_{1}}-(({\partial G\over \partial \xi})_{T,P}|_{T_{1},P_{1}^{\circ}}-RT_{1}ln(Q_{chem'}(T_{1},P_{1}))-\epsilon(P_{1})$\\

\end{proof}
\end{section}
\begin{section}{dilute solutions}
\label{errorterms2}
\begin{defn}
\label{catalyzers}

As mentioned in Definition \ref{constants}, we can consider an electrolyte as a solute in a dilute solution. Sometimes the solvent is involved in an electrolytic reaction, for example;\\

$2H_{2}O+4e^{-}(R)\rightarrow O_{2}+2H_{2}+4e^{-}(L)$ $(*)$\\

and sometimes not, as in the standard cell, where we can consider $H_{2}O$ as the solvent not involved in the reaction. In Lemmas 0.4-0.12, for the standard cell, we can replace $Q$ defined as $\prod_{i=1}^{c}a_{i}^{\nu_{i}}$ by $a_{0}(\prod_{i=1}^{c}a_{i}^{\nu_{i}})$, considering $H_{2}O$ as substance $0$.\\

We ideally have that $\mu_{i}=\mu_{i}^{\circ}+RT ln(a_{i})$, $0\leq i\leq c$, $(\dag)$, when we define the activities $a_{i}$, for $0\leq i\leq c$, which, when $(\dag)$ holds, involves a contradiction.

\end{defn}

\begin{lemma}
\label{proofs}
In Lemmas 0.4-0.12, for the standard cell, and considering a dilute solution with no interaction of the solvent, replacing $Q$ defined as $\prod_{i=1}^{c}a_{i}^{\nu_{i}}$ by;\\

$a_{0}(\prod_{i=1}^{c}a_{i}^{\nu_{i}})$\\

If we assume without approximation that $\mu_{i}=\mu_{i}^{\circ}+RT ln(a_{i})$, $0\leq i\leq c$, then;\\

Lemma \ref{nernst}; $E-E^{\circ}=-{RTln(Q)\over 2F}+{RTln({a_{0}(T,P)\over a_{0}(T,P^{\circ})})\over 2F}$\\

Lemma \ref{gibbs}; $({\partial G\over \partial \xi})_{T,P}=\Delta G^{\circ}+RTln({Q\over a_{0}(T,P)})$\\

where $\Delta G^{\circ}$ is the Gibbs energy change for $1$ mole of reaction without the solvent.\\

Lemma \ref{delta}; $\Delta G^{\circ}=2F(E-E^{\circ})+RTln(a_{0}(T,P^{\circ}))$\\

Lemma \ref{equivalences}; The same with the modification that if chemical and electrical equilibrium exist at $(T,P^{\circ})$ and $(T,P)$, $Q(T,P)=a_{0}(T,P)$ and $E-E^{\circ}=-{RTln(a_{0}(T,P^{\circ}))\over 2F}$. Conversely, if $Q(T,P)={a_{0}(T,P)\over a_{0}(T,P^{\circ})}$ and chemical equilibrium exists at $(T,P^{\circ})$ then chemical equilibrium exists at $(T,P)$.\\

Chemical equilibrium exists at $(T,P)$ iff $Q(T,P)=a_{0}e^{-\Delta G^{\circ}\over RT}$\\

We always have that $Q(T,P^{\circ})=1$\\

Lemma \ref{vanhoffhelmholtz}; The same, with the modification that along a chemical equilibrium path, we have that;\\

$ln({Q(T_{2})\over Q(T_{1})})={1\over R}\int_{T_{1}}^{T_{2}}{\Delta H^{\circ}\over T^{2}}dT+ln({a_{0}(T_{2},P_{2})\over a_{0}(T_{1},P_{1})})$\\

and, if $\Delta H^{\circ}$ is temperature independent;\\

$ln({Q(T_{2})\over Q(T_{1})})=-{\Delta H^{\circ}\over R}({1\over T_{2}}-{1\over T_{1}})+ln({a_{0}(T_{2},P_{2})\over a_{0}(T_{1},P_{1})})$\\

For $c\in\mathcal{R}$, if $D_{c}$ intersect the line $P=P^{\circ}$ at $(T_{1},P^{\circ})$, then, for $(T_{2},P)\in D_{c}$, we have that;\\

$Q(T_{2},P)=e^{\Delta G^{\circ}(T_{1})-\Delta G^{\circ}(T_{2})\over RT_{2}}a_{0}(T_{1},P^{\circ})^{-T_{1}\over T_{2}}a_{0}(T_{2},P)$ $(\dag\dag\dag)'$\\

$c=\Delta G^{\circ}(T_{1})-RT_{1}ln(a_{0}(T_{1},P^{\circ})$\\

$ln({Q(T_{2})\over Q(T_{1})})=ln(Q(T_{2}))={1\over R}\int_{T_{1}}^{T_{2}}{\Delta H^{\circ}-c\over T^{2}}dT+ln({a_{0}(T_{2},P)\over a_{0}(T_{1},P^{\circ})})$\\

${\Delta G^{\circ}(T_{2})-\Delta G^{\circ}(T_{1})\over T_{2}}=-\int_{T_{1}}^{T_{2}}{\Delta H^{\circ}-c\over T^{2}}dT+R(1-{T_{1}\over T_{2}})ln(a_{0}(T_{1},P^{\circ}))$\\

and if $\Delta H^{\circ}$ is temperature independent;\\

$ln({Q(T_{2})\over Q(T_{1})})=ln(Q(T_{2}))=-({\Delta H^{\circ}-c\over R})({1\over T_{2}}-{1\over T_{1}})+ln({a_{0}(T_{2},P)\over a_{0}(T_{1},P^{\circ})})$ $(\dag\dag)'$\\

${\Delta G^{\circ}(T_{2})-\Delta G^{\circ}(T_{1})\over T_{2}}=(\Delta H^{\circ}-c)({1\over T_{2}}-{1\over T_{1}})+R(1-{T_{1}\over T_{2}})ln(a_{0}(T_{1},P^{\circ}))$ $(\dag)'$\\

to obtain;\\

$\Delta G^{\circ}(T_{1})={T_{1}\over T_{2}}\Delta G^{\circ}(T_{2})-\Delta H^{\circ}({T_{1}\over T_{2}}-1)$\\

again.\\

Lemma \ref{eqlines}; The same, with the modification that if $\epsilon\neq 0$;\\

$Q(T_{2},P')=e^{\epsilon ln({P'\over P'^{\circ}})\over RT_{2}}a_{0}(T_{2},P')$\\

and the dynamic equilibrium paths are given by;\\

$a_{0}(T_{2},P')({P'\over P'^{\circ}})^{\epsilon\over RT_{2}}=c$\\

for $c\in\mathcal{R}_{\geq 0}$, while if $\epsilon=0$, $Q(T_{2},P')={a_{0}(T_{2},P')\over a_{0}(T_{2},P'^{\circ})}$, the dynamic equilibrium lines are given by;\\

${a_{0}(T_{2},P')\over a_{0}(T_{2},P'^{\circ})}=c$\\

for $c\in\mathcal{R}_{\geq 0}$ and the quasi-chemical equilibrium lines are given by;\\

$\lambda+\beta T_{2}+\sigma ln(T_{2})=c$\\

for $c\in\mathcal{R}$.

\end{lemma}
\begin{proof}
Following the proof of Lemma \ref{equivalences}, we note that for Gibbs function $G$ with $c+1$ species, including the solvent, substance $0$, as $dn_{0}=0$, that;\\

$dG=\sum_{i=0}^{c}\mu_{i}dn_{i}$\\

$=\sum_{i=1}^{c}\mu_{i}dn_{i}$\\

so the first three claims in Lemma \ref{equivalences} go through as before. Going through the proof of Lemma \ref{gibbs}, we then obtain that;\\

$({\partial G\over \partial \xi})_{T,P}=\Delta G^{\circ}+RTln(\prod_{i=1}^{c}a_{i}^{\nu_{i}})$\\

$=\Delta G^{\circ}+RTln({\prod_{i=0}^{c}a_{i}^{\nu_{i}}\over a_{0}(T,P)})$\\

$=\Delta G^{\circ}+RTln({Q\over a_{0}(T,P)})$\\

Going back through the proof of Lemma \ref{equivalences}, we obtain that $RTln({Q\over a_{0}(T,P)})=0$, so that $Q(T,P)=a_{0}(T,P)$. Going through the proof of Lemma \ref{nernst}, using the fact that $\mu_{i}=\mu_{i}^{\circ}+RTln(a_{i})$, for $0\leq i\leq c$, so that $Q(T,P^{\circ})=1$, we have that;\\

$({\partial G_{chem'}\over \partial \xi})_{T,P^{\circ}}=\Delta G_{chem'}^{\circ}+RTln(Q_{chem'}(T,P^{\circ})$\\

$=\Delta G_{chem'}^{\circ}+RTln({Q\over a_{0}(T,P^{\circ})})$\\

$=\Delta G_{chem'}^{\circ}-RTln(a_{0}(T,P^{\circ}))$ $(\dag\dag)'$\\

where $G_{chem'}$ is the Gibbs energy restricted to the uncharged species without the solvent. Using $(\dag)$ from Lemma \ref{nernst} and $(\dag\dag)'$, we obtain;\\

$2E^{\circ}F=-\Delta G_{chem'}^{\circ}+RTln(a_{0}(T,P^{\circ}))$ $(\dag\dag\dag)'$\\

Similarly, we obtain;\\

$({\partial G_{chem'}\over \partial \xi})_{T,P}=\Delta G_{chem'}^{\circ}+RTln(Q_{chem'}(T,P)$\\

$=\Delta G_{chem'}^{\circ}+RTln({Q\over a_{0}(T,P)})$\\

so that, from $(\dag)$ from Lemma \ref{nernst};\\

$2EF=-(\Delta G_{chem'}^{\circ}+RTln({Q\over a_{0}(T,P)}))$ $(\sharp)'$\\

Combining $(\sharp)'$, $(\dag\dag\dag)'$, we obtain;\\

$2EF-2E^{\circ}F=-(\Delta G_{chem'}^{\circ}+RTln({Q\over a_{0}(T,P)}))-(-\Delta G_{chem'}^{\circ}+RTln(a_{0}(T,P^{\circ})))$\\

$=-RTln({Q\over a_{0}(T,P)})-RTln(a_{0}(T,P^{\circ}))$\\

$=-RTln(Q)+RTln({a_{0}(T,P)\over a_{0}(T,P^{\circ})})$\\

which gives the result.of Lemma \ref{nernst}. Going back through the proof of Lemma \ref{equivalences} again, we then have that;\\

$E-E^{\circ}=-{RTln(Q)\over 2F}+{RTln({a_{0}(T,P)\over a_{0}(T,P^{\circ})})\over 2F}$\\

$=-{RTln(a_{0}(T,P))\over 2F}+{RTln({a_{0}(T,P)\over a_{0}(T,P^{\circ})})\over 2F}$\\

$=-{RTln(a_{0}(T,P^{\circ}))\over 2F}$\\

For the converse claim, we have by the modification of Lemma \ref{gibbs}, and the facts that $Q(T,P)={a_{0}(T,P)\over a_{0}(T,P^{\circ})}$,$Q(T,P^{\circ})=1$, $({\partial G\over \partial \xi})_{T,P^{\circ}}=0$, that;\\

$({\partial G\over \partial \xi})_{T,P}=\Delta G^{\circ}+RTln({Q\over a_{0}(T,P)})$\\

$=\Delta G^{\circ}-RTln(a_{0}(T,P^{\circ}))$\\

$({\partial G\over \partial \xi})_{T,P^{\circ}}=\Delta G^{\circ}+RTln({Q\over a_{0}(T,P^{\circ})})$\\

$=0$\\

so that we have chemical equilibrium at $(T,P)$. For the penultimate claim of Lemma \ref{equivalences}, rearrange the formula from the modification of Lemma \ref{gibbs}, with the definition of chemical equilibrium;\\

$({\partial G\over \partial \xi})_{T,P}=\Delta G^{\circ}+RTln({Q\over a_{0}(T,P)})=0$\\

The last claim is clear from;\\

$\mu_{i}=\mu_{i}^{\circ}+RTln(a_{i})$, for $0\leq i\leq c$\\

For Lemma \ref{delta}, we have by the modification of Lemma \ref{nernst}, that;\\

$E-E^{\circ}=-{RTln(Q)\over 2F}+{RTln({a_{0}(T,P)\over a_{0}(T,P^{\circ})})\over 2F}$\\

and, by the modification of Lemma \ref{gibbs}, that;\\

$0=({\partial G\over \partial \xi})_{T,P}=\Delta G^{\circ}+RTln(Q)-RTln(a_{0}(T,P))$\\

so that;\\

$\Delta G^{\circ}=RTln(a_{0}(T,P))-RTln(Q)$\\

$=RTln(a_{0}(T,P))+2F(E-E^{\circ})-RTln({a_{0}(T,P)\over a_{0}(T,P^{\circ})})$\\

$=2F(E-E^{\circ})+RTln(a_{0}(T,P^{\circ}))$\\

For the modification of Lemma \ref{vanhoffhelmholtz}, the first part of the proof goes through with the chemical potentials $\mu_{i},1\leq i\leq c$, defined relative to the Gibbs energy including the solvent. By the modification of Lemma \ref{equivalences}, we have that $Q=a_{0}e^{-\Delta G^{\circ}\over RT}$ along a chemical equilibrium path, so that $ln(Q)={-\Delta G^{\circ}\over RT}+ln(a_{0}(T,P))$, $(K)'$. It follows that;\\

${d ln(Q)\over dT}={d\over dT}({-\Delta G^{\circ}\over RT})+{d\over dT}(ln(a_{0}(T,P)))$\\

$={\Delta H^{\circ}\over RT^{2}}+{d\over dT}(ln(a_{0}(T,P)))$\\

It follows, integrating between $T_{1}$ and $T_{2}$, using $(K)'$ and the fundamental theorem of calculus, that;\\

$ln({Q(T_{2})\over Q(T_{1})})=ln(Q)(T_{2})-ln(Q)(T_{1})$\\

$={-\Delta G^{\circ}(T_{2})\over RT_{2}}+{\Delta G^{\circ}(T_{1})\over RT_{1}}+ln(a_{0}(T_{2},P_{2}))-ln(a_{0}(T_{1},P_{1}))$\\

$=\int_{T_{1}}^{T_{2}}{d ln(Q)\over dT}dT$\\

$={1\over R}\int_{T_{1}}^{T_{2}}{\Delta H^{\circ}\over T^{2}}+\int_{T_{1}}^{T_{2}}{d\over dT}(ln(a_{0}(T,P)))dT$\\

$={1\over R}\int_{T_{1}}^{T_{2}}{\Delta H^{\circ}\over T^{2}}+ln(a_{0}(T_{2},P_{2}))-ln(a_{0}(T_{1},P_{1}))$ $(P)'$\\

so that, rearranging, we obtain the first claim. Using the fact, by the modification of Lemma \ref{gibbs}, that;\\

$ln(Q(T_{2}))=-{\Delta G^{\circ}(T_{2})\over RT_{2}}+ln(a_{0}(T_{2},P_{2}))$\\

$ln(Q(T_{1}))=-{\Delta G^{\circ}(T_{1})\over RT_{1}}+ln(a_{0}(T_{1},P_{1}))$\\

we obtain, substituting into $(P)'$, canceling $R$, and performing the integration, if $\Delta H^{\circ}$ is temperature independent, that;\\

${\Delta G^{\circ}(T_{2})\over T_{2}}-{\Delta G^{\circ}(T_{1})\over T_{1}}=-\int_{T_{1}}^{T_{2}}{\Delta H^{\circ}\over T^{2}}=\Delta H^{\circ}({1\over T_{2}}-{1\over T_{1}})$ $(Q)'$\\

For the fifth claim, rearrange $(Q)'$. If $D_{c}$ intersects the line $P=P^{\circ}$ at $(T_{1},P^{\circ})$, for the sixth $(\dag\dag\dag)$ and seventh claims, we have, using the modification of Lemma \ref{gibbs} and the fact from Lemma \ref{equivalences} that $Q(T_{1},P^{\circ})=1$;\\

$({\partial G\over \partial \xi})_{T_{2},P}=\Delta G^{\circ}(T_{2})+RT_{2}ln(Q(T_{2},P))-RT_{2}ln(a_{0}(T_{2},P))$\\

$=({\partial G\over \partial \xi})_{T_{1},P^{\circ}}$\\

$=\Delta G^{\circ}(T_{1})+RT_{1}ln(Q(T_{1},P^{\circ}))-RT_{1}ln(a_{0}(T_{1},P^{\circ}))$\\

$=\Delta G^{\circ}(T_{1})-RT_{1}ln(a_{0}(T_{1},P^{\circ}))=c$\\

so that, again rearranging, we obtain the result. Along $D_{c}$, we have, using Lemma \ref{gibbs}, that;\\

$ln(Q)={c-\Delta G^{\circ}\over RT}+ln(a_{0}(T,P))$\\

so that, using the first part;\\

${d ln(Q)\over dT}={d\over dT}({c-\Delta G^{\circ}\over RT})+{d\over dT}(ln(a_{0}(T,P)))$\\

$={-c\over RT^{2}}+{d\over dT}({-\Delta G^{\circ}\over RT})+{d\over dT}(ln(a_{0}(T,P)))$\\

$={\Delta H^{\circ}-c\over RT^{2}}+{d\over dT}(ln(a_{0}(T,P)))$\\

so that, performing the integration, using the fact that $Q(T_{1},P^{\circ})=1$;\\

$ln(Q(T_{2}))-ln(Q(T_{1}))=ln(Q(T_{2}))={1\over R}\int_{T_{1}}^{T_{2}}{\Delta H^{\circ}-c\over T^{2}}dT+ln(a_{0}(T_{2},P_{2}))-ln(a_{0}(T_{1},P^{\circ}))$\\

We have that, by the modification of Lemma \ref{gibbs};\\

$ln (Q(T_{2}))={c-\Delta G^{\circ}(T_{2})\over RT_{2}}+ln(a_{0}(T_{2},P))$\\

$ln(Q(T_{1}))=0$\\

so that, using the formula for $c$;\\

$ln (Q(T_{2}))=ln (Q(T_{2}))-ln(Q(T_{1}))$\\

$={c-\Delta G^{\circ}(T_{2})\over RT_{2}}+ln(a_{0}(T_{2},P))$\\

$={\Delta G^{\circ}(T_{1})-RT_{1}ln(a_{0}(T_{1},P^{\circ}))-\Delta G^{\circ}(T_{2})\over RT_{2}}+ln(a_{0}(T_{2},P))$\\

$={1\over R}\int_{T_{1}}^{T_{2}}{\Delta H^{\circ}-c\over T^{2}}dT+ln(a_{0}(T_{2},P_{2}))-ln(a_{0}(T_{1},P^{\circ}))$\\

$={-1\over R}(\Delta H^{\circ}-c)({1\over T_{2}}-{1\over T_{1}})+ln(a_{0}(T_{2},P_{2}))-ln(a_{0}(T_{1},P^{\circ}))$\\

and rearranging;\\

${\Delta G^{\circ}(T_{1})-RT_{1}ln(a_{0}(T_{1},P^{\circ}))-\Delta G^{\circ}(T_{2})\over RT_{2}}={-1\over R}(\Delta H^{\circ}-c)({1\over T_{2}}-{1\over T_{1}})-ln(a_{0}(T_{1},P^{\circ}))$\\

${\Delta G^{\circ}(T_{2})-\Delta G^{\circ}(T_{1})\over T_{2}}+{RT_{1}ln(a_{0}(T_{1},P^{\circ}))\over T_{2}}$\\

$=(\Delta H^{\circ}-c)({1\over T_{2}}-{1\over T_{1}})+Rln(a_{0}(T_{1},P^{\circ}))$\\

${\Delta G^{\circ}(T_{2})-\Delta G^{\circ}(T_{1})\over T_{2}}$\\

$=(\Delta H^{\circ}-c)({1\over T_{2}}-{1\over T_{1}})+Rln(a_{0}(T_{1},P^{\circ}))-{RT_{1}ln(a_{0}(T_{1},P^{\circ}))\over T_{2}}$\\

$=(\Delta H^{\circ}-(\Delta G^{\circ}(T_{1})-RT_{1}ln(a_{0}(T_{1},P^{\circ}))))({1\over T_{2}}-{1\over T_{1}})+Rln(a_{0}(T_{1},P^{\circ}))-{RT_{1}ln(a_{0}(T_{1},P^{\circ}))\over T_{2}}$\\

$=(\Delta H^{\circ}-\Delta G^{\circ}(T_{1}))({1\over T_{2}}-{1\over T_{1}})+RT_{1}ln(a_{0}(T_{1},P^{\circ}))({1\over T_{2}}-{1\over T_{1}})+(R-{RT_{1}\over T_{2}})ln(a_{0}(T_{1},P^{\circ}))$\\

$=(\Delta H^{\circ}-\Delta G^{\circ}(T_{1}))({1\over T_{2}}-{1\over T_{1}})+ln(a_{0}(T_{1},P^{\circ}))(RT_{1}({1\over T_{2}}-{1\over T_{1}})+(R-{RT_{1}\over T_{2}}))$\\

$=(\Delta H^{\circ}-\Delta G^{\circ}(T_{1}))({1\over T_{2}}-{1\over T_{1}})$\\

so that, rearranging again;\\

$\Delta G^{\circ}(T_{1})({1\over T_{1}}+{1\over T_{2}}-{1\over T_{2}})={\Delta G^{\circ}(T_{1})\over T_{1}}$\\

$={\Delta G^{\circ}(T_{2})\over T_{2}}-\Delta H^{\circ}({1\over T_{2}}-{1\over T_{1}})$\\

to obtain;\\

$\Delta G^{\circ}(T_{1})={T_{1}\over T_{2}}\Delta G^{\circ}(T_{2})-\Delta H^{\circ}({T_{1}\over T_{2}}-1)$\\

For the modification of Lemma \ref{eqlines}, be careful to use the restricted summation for $c$ substances, in the calculation of $({\partial({{\partial G\over \partial \xi}})_{T,P}\over \partial T})_{P}$. while the calculation for $dU$ involves $c+1$ substances, including the solvent. If $\epsilon=0$, we have that;\\

$c=\Delta G^{\circ}(T_{2})+RT_{2}ln(Q(T_{2},P'))-RT_{2}ln(a_{0}(T_{2},P'))$\\

$=\Delta G^{\circ}(T_{2})-RT_{2}ln(a_{0}(T_{2},P'^{\circ}))$\\

so that rearranging, $Q(T_{2},P')={a_{0}(T_{2},P')\over a_{0}(T_{2},P'^{\circ})}$\\

and the claims about dynamic and quasi-chemical equilibrium lines follows from $Q=c$ and $({\partial G\over \partial \xi})_{T,P}=c$, for $c\in\mathcal{R}_{\geq 0}$ and $c\in\mathcal{R}$ respectively. If $\epsilon\neq 0$, we use the modification of Lemma \ref{gibbs}, to obtain the formula for the activity coefficient;\\

$Q(T_{2},P')=e^{({\partial G\over \partial \xi})_{T,P}|{T_{2},P'}-\Delta G^{\circ}(T_{2})+RT_{2}ln(a_{0}(T_{2},P'))\over RT_{2}}$\\

$=e^{\epsilon ln({P'\over P'^{\circ}})\over RT_{2}}a_{0}(T_{2},P')$\\

Again, the determination of the dynamical and quasi-chemical equilibrium lines follows again from rearrangement of $Q=c$ and $({\partial G\over \partial \xi})_{T,P}=c$, for $c\in\mathcal{R}_{\geq 0}$ and $c\in\mathcal{R}$ respectively.

\end{proof}
\begin{lemma}
\label{proofs2}
In Lemmas 0.4-0.12, for the case of a reaction like $(*)$ in Definition \ref{catalyzers}, and considering a dilute solution with interaction of the solvent, replacing $Q$ defined as $\prod_{i=1}^{c}a_{i}^{\nu_{i}}$ by;\\

$\prod_{i=0}^{c}a_{i}^{\nu_{i}}$\\

If we assume without approximation that $\mu_{i}=\mu_{i}^{\circ}+RTln(a_{i})$, $0\leq i\leq c$, then the proofs go through as before, with the modification that we have $c+1$ rather than $c$ substances.
\end{lemma}
\begin{rmk}
\label{reference}
In the case of a solvent with no interaction, if we define the activities by $a_{i}=x_{i}$, $0\leq i\leq c$, with the definition of $Q$ as $a_{0}\prod_{1\leq i\leq c}a_{i}^{\nu_{i}}$, then we need to modify the proofs of Lemmas \ref{implies} and \ref{feasible}. This is done in Lemma \ref{henrys2}. For the existence of a feasible path, where we require that the $n_{0}$ term is fixed, see Remark \ref{zero}, we need to change Lemma \ref{exists}. Letting $d_{0}>0$ denote the fixed molar amount of the solvent, we obtain the relation, modifying the proof of Lemma \ref{exists};\\

$d_{0}\prod_{i=1}^{c-1}({\nu_{i}\over \nu_{c}}n_{c}+d_{i})^{\nu_{i}}(t)n_{c}^{\nu_{c}}(t)=
\epsilon(t)((\sum_{i=1}^{c-1}{\nu_{i}\over \nu_{c}}+1)n_{c}+\sum_{i=0}^{c-1}d_{i})^{w}(t)$\\

where $w=1+\sum_{i=1}^{c}\nu_{i}$.\\

In the $w>0$ case, absorb the constant $d_{0}$ into $\epsilon(t)$, by setting $\epsilon_{1}(t)={\epsilon(t)\over d_{0}}>0$, and redefine $\sigma=\sum_{i=0}^{c-1}d_{i}$. Then use the proof of Lemma \ref{exists}, noting that if $\nu_{i}=w$, for $1\leq i\leq p$, then $w=1+\sum_{i=1}^{c}\nu_{i}>\sum_{i=1}^{p}\nu_{i}=\sum_{i=1}^{p}w=pw$ which is a contradiction again. If $w<0$, then use reciprocality again to reduce to $w>0$, replacing $d_{0}$ with ${1\over d_{0}}>0$. The $w=0$ case is again similar, using the $w>0$ calculation.\\
\end{rmk}

\end{section}
\begin{section}{Dilute solutions with Henry's Law for Solutes, Raoult's Law for the Solvent and Interaction of the Solvent}
\label{henryssolutesidealsolvent}
\begin{defn}
\label{catalyzers2}

As mentioned in Definition \ref{constants}, we can consider an electrolyte as a solute in a dilute solution and define the activities $a_{i}$, $0\leq i\leq c$, by;\\

$a_{0}=x_{0}\simeq 1$\\

$a_{i}=x_{i}$ $(1\leq i\leq c)$\\

and define;\\

$Q=a_{0}\prod_{i=1}^{c}a_{i}^{\nu_{i}}\simeq \prod_{i=1}^{c}a_{i}^{\nu_{i}}$ (no interaction of the solvent)\\

$Q=\prod_{i=0}^{c}a_{i}^{\nu_{i}}$ (interaction of the solvent)\\

We ideally have that $\mu_{i}=\mu_{i}^{\circ}+RTln(a_{i})$, $0\leq i\leq c$, which involves a contradiction. By Henry's Law, we have that $P_{i}=k_{i}x_{i}$, $1\leq i\leq c$, so that, by Henry's Law, phase equilibrium and the ideal gas law;\\

$\mu_{i}^{(sol)}=\mu_{i}^{(g)\circ}+RT ln({P_{i}\over P^{\circ}})$\\

$=\mu_{i}^{(g)\circ}+RTln({k_{i}x_{i}\over P^{\circ}})$\\

$=\mu_{i}^{(H)\circ}+RTln(x_{i})$ $(*)$\\

where $\mu_{i}^{(H)\circ}=\mu_{i}^{(g)\circ}+RTln({k_{i}\over P^{\circ}})$ $(\dag)$. From $(*)$, we obtain that;\\

$\mu_{i}^{\circ}=\mu_{i}^{(H)\circ}+RTln(x_{i}(T,P^{\circ}))$ $(\dag\dag)$\\

and, for $1\leq i\leq c$;\\

$\mu_{i}=\mu_{i}^{\circ}-RTln(x_{i}(T,P^{\circ}))+RTln(x_{i})$\\

$=\mu_{i}^{\circ}+RTln(x_{i})+\kappa(T)$ $(\dag\dag\dag)$\\

where, from phase equilibrium and $(\dag),(\dag\dag)$;\\

$\kappa(T)=-RTln(x_{i}(T,P^{\circ}))$\\

$=\mu_{i}^{(H)\circ}-\mu_{i}^{\circ}$\\

$=\mu_{i}^{(g)\circ}+RTln({k_{i}\over P^{\circ}})-\mu_{i}^{(sol)\circ}$\\

$=RTln({k_{i}\over P^{\circ}})$\\

We can use Raoult's law for the solvent as $P_{0}=P_{0}^{*}x_{0}$, see \cite{M}. Then, using equilibrium with an ideal gas mixture;\\

$\mu_{0}^{(sol)}=\mu_{0}^{\circ(g)}+RTln({P_{0}\over P^{\circ}})$\\

$=\mu_{0}^{\circ(g)}+RTln({P_{0}^{*}x_{0}\over P^{\circ}})$\\

$\mu_{0}^{\circ(g)}+RTln(x_{0})+RTln({P_{0}^{*}\over P^{\circ}})$\\

$=\mu_{0}^{*}(T,P_{0}^{*})+RTln(x_{0})$\\

$=\mu_{0}^{*}(T,P)+RTln(x_{0})+\theta(T,P)$ $(**)$\\

where $\theta(T,P)=\mu_{0}^{*}(T,P_{0}^{*})-\mu_{0}^{*}(T,P)\simeq 0$\\

so that;\\

$RTln(x_{0})=RTln({P_{i}\over P^{\circ}})-RTln({P_{i}^{*}\over P^{\circ}})$\\

and;\\

$\mu_{0}=\mu_{0}^{*}+RTln(x_{0})+\theta$\\

$=\mu_{0}^{*}+RTln({P_{0}\over P^{\circ}})-RTln({P_{0}^{*}\over P^{\circ}})+\theta$ $(***)'$\\

We also have, using the phase rule for the solvent in equilibrium with an ideal gas mixture, that;\\

$\mu_{0}^{(g)}=\mu_{0}^{\circ(g)}+RTln({P_{0}\over P^{\circ}})$\\

$\mu_{0}^{(sol)}=\mu_{0}^{\circ(sol)}+RTln({P_{0}\over P^{\circ}})$ $(*)$\\

Combining $(*),(***)'$, we obtain that;\\

$\mu_{0}^{*}=\mu_{0}-RTln({P_{0}\over P^{\circ}})+RTln({P_{0}^{*}\over P^{\circ}})-\theta$\\

$=(\mu_{0}^{\circ}+RTln({P_{0}\over P^{\circ}}))-RTln({P_{0}\over P^{\circ}})+RTln({P_{0}^{*}\over P^{\circ}})-\theta$\\

$=\mu_{0}^{\circ}+RTln({P_{0}^{*}\over P^{\circ}})-\theta$ $(\dag)'$\\

Letting $P_{0}^{*}=P^{\circ}$, we obtain that $\mu_{0}^{*}(T,P')=\mu_{0}^{\circ}-\theta$, for the corresponding $P'$, $(****)'$\\

From $(**)$, $(****)'$, we obtain that;\\

$\mu_{0}=\mu_{0}^{*}+RTln(x_{0})+\theta$\\

$=\mu_{0}^{*}(T,P')+\delta+RTln(x_{0})+\theta$\\

$=\mu_{0}^{\circ}-\theta+\delta+RTln(x_{0})+\theta$\\

$=\mu_{0}^{\circ}+RTln(x_{0})+\delta$\\

where $\delta=\mu_{0}^{*}(T,P)-\mu_{0}^{*}(T,P')\simeq 0$\\

Using the same calculation as before, we have that;\\

$\mu_{0}=\mu_{0}^{\circ}+RTln(x_{0})+\gamma(P)$\\

\end{defn}

\begin{lemma}
\label{henrys3}
In the case of dilute solutions, with interaction of the solvent, a feasible path $\gamma$ is a dynamic equilibrium path iff $pr(\gamma_{12})\subset C_{c}$, for some $c\in\mathcal{R}_{>0}$ iff ${dQ\over dt}=0$.
\end{lemma}
\begin{proof}
We have that;\\

$\prod_{i=0}^{c}a_{i}^{\nu_{i}}=\prod_{i=0}^{c}x_{i}^{\nu_{i}}$\\

$=Q=c=f$\\

where $x_{i}={n_{i}\over n}$. Now follow through the proof of Lemma \ref{implies}, as we are differentiating, the proof works with a constant $f>0$.
\end{proof}

We reformulate Lemmas \ref{gibbs2} to \ref{rates} in this context, assuming Henry's law for the solutes and the solvent an ideal solution;\\

\begin{lemma}
\label{henrys5}
In the dilute solution case, with interaction of the solvent, for the energy function $G$ involving  $c+1$ uncharged species;\\

$({\partial G\over \partial \xi})_{T,P}=\Delta G^{\circ}+RTln(Q)+\epsilon$\\

where $\epsilon(T,P)=\nu_{0}\gamma_{0}(P)+\sum_{i=1}^{c}\nu_{i}\kappa_{i}(T)\simeq 0$, $\gamma_{0}(P)\simeq 0$ and $\kappa_{i}(T)\simeq 0$ are the error term for the $i$'th uncharged species in Definition \ref{catalyzers2}, $1\leq i\leq c$.
\end{lemma}
\begin{proof}
The proof is clear from Lemma \ref{gibbs2}.

\end{proof}
\begin{lemma}
\label{henrys6}
For a dilute solution, with interaction of the solvent, we have, using the definition of $\epsilon(T,P)$ in Lemma \ref{henrys5}, the error term $\gamma_{0}(P)$ and the error terms $\kappa_{i}(T)$, $1\leq i\leq c$ in Definition \ref{catalyzers}, that the same results as Lemma \ref{equivalences2} hold, replacing $\epsilon(P)$ with $\epsilon(T,P)$ and $\delta$ with $\epsilon(T,P^{\circ})=\nu_{0}\gamma_{0}(P^{\circ})+\sum_{i=1}^{c}\nu_{i}\kappa_{i}(T)$\\

\end{lemma}
\begin{proof}
The proof is clear from the proof of Lemma \ref{equivalences2}.

\end{proof}
\begin{lemma}
\label{henrys7}
For a dilute solution, with interaction of the solvent, we have the same result as Lemma \ref{van't Hoff,Gibbs-Helmholtz2} hold, replacing $\epsilon(P(T))$ by $\epsilon(T,P(T))$ along the quasi-chemical equilibrium lines.

\end{lemma}
\begin{proof}
The proof is clear from the proof of Lemma \ref{van't Hoff,Gibbs-Helmholtz2}

\end{proof}
\begin{lemma}
\label{henrys8}
For a dilute solution, with interaction of the solvent, we have the same results as Lemma \ref{eqlines2} hold, replacing $\epsilon(P')$ by $\epsilon(T_{1},P')$. In particularly, if $\epsilon\neq 0$, we have that;\\

$Q(T,P)=e^{\epsilon ln({P\over P^{\circ}})-\epsilon(T,P)\over RT}$\\

and, if $\epsilon=0$, we have that;\\

$Q(T,P)=e^{-\epsilon(T,P)\over RT}$\\

\end{lemma}
\begin{proof}
The proof is clear from the proof of Lemma \ref{eqlines2}.

\end{proof}
\begin{lemma}
\label{henrys9}
For a dilute solution, with interaction of the solvent, we have that, if $\epsilon\neq 0$;\\

$grad(Q)(T,P)=(({-\epsilon ln({P\over P^{\circ}})\over RT^{2}}+{\epsilon(T,P)\over RT^{2}}-{{\partial \epsilon\over \partial T}(T,P)\over RT})({P\over P^{\circ}})^{\epsilon\over RT}e^{-\epsilon(T,P)\over RT},$\\

$({\epsilon P^{\circ}\over RTP}-{{\partial \epsilon\over \partial P}(T,P)\over RT})({P\over P^{\circ}})^{\epsilon\over RT}e^{-\epsilon(T,P)\over RT})$\\

and, if $\epsilon=0$;\\

$grad(Q)(T,P)=(({\epsilon(T,P)\over RT^{2}}-{{\partial \epsilon\over \partial T}(T,P)\over RT})e^{-\epsilon(T,P)\over RT},{-{\partial \epsilon\over \partial P}(T,P)\over RT}e^{-\epsilon(T,P)\over RT})$\\

The paths of maximal reaction in the region $|grad(Q)(T,P)|>1$, $Q(T,P)>0$, are given by implicit solutions to the differential equations;\\

${dP\over dT}={\epsilon TP^{\circ}-PT{\partial\epsilon\over \partial P}(T,P)\over (-\epsilon Pln({P\over P^{\circ}})+P\epsilon(T,P)-PT{\partial \epsilon\over \partial T}(T,P))}$\\

${dP\over dT}={-T{\partial \epsilon\over \partial P}(T,P)\over \epsilon(T,P)-T{\partial \epsilon\over \partial T}(T,P)}$\\

respectively.\\

\end{lemma}
\begin{proof}
The computation of $grad(Q)$ in both cases is a simple application of the chain and product rules. Following the method of Lemma \ref{rates}, noting the claim about maximal reaction is still valid with the same definition of $Q$, if $\epsilon\neq 0$, we compute;\\

${dP\over dT}={({\epsilon P^{\circ}\over RTP}-{{\partial \epsilon\over \partial P}(T,P)\over RT})({P\over P^{\circ}})^{\epsilon\over RT}e^{-\epsilon(T,P)\over RT}\over ({-\epsilon ln({P\over P^{\circ}})\over RT^{2}}+{\epsilon(T,P)\over RT^{2}}-{{\partial \epsilon\over \partial T}(T,P)\over RT})({P\over P^{\circ}})^{\epsilon\over RT}e^{-\epsilon(T,P)\over RT}}$\\

$={\epsilon TP^{\circ}-PT{\partial\epsilon\over \partial P}(T,P)\over (-\epsilon Pln({P\over P^{\circ}})+P\epsilon(T,P)-PT{\partial \epsilon\over \partial T}(T,P))}$\\

and, if $\epsilon=0$;\\

${dP\over dT}={{{-{\partial \epsilon\over \partial P}(T,P)\over RT}e^{-\epsilon(T,P)\over RT}}\over ({\epsilon(T,P)\over RT^{2}}-{{\partial \epsilon\over \partial T}(T,P)\over RT})e^{-\epsilon(T,P)\over RT}}$\\

$={-T{\partial \epsilon\over \partial P}(T,P)\over \epsilon(T,P)-T{\partial \epsilon\over \partial T}(T,P)}$\\

\end{proof}
\end{section}
\begin{section}{Dilute solutions with Fugacity and Interaction of the Solvent}
\label{fug}
\begin{defn}
\label{catalyzers3}

As mentioned in Definition \ref{constants}, we can consider an electrolyte as a solute in a dilute solution and define the activities $a_{i}$, $0\leq i\leq c$, by;\\

$a_{0}=\gamma_{0}x_{0}\simeq 1$\\

$a_{i}=\gamma_{i}x_{i}$ $(1\leq i\leq c)$\\

We can define the activity coefficient  $Q=a_{0}\prod_{i=1}^{c}a_{i}^{\nu_{i}}\simeq \prod_{i=1}^{c}a_{i}^{\nu_{i}}$, but we will adopt a new convention, see below.\\

We have that $\mu_{i}=\mu_{i}^{\circ}+RT ln(a_{i})$, $0\leq i\leq c$, which involves the contradiction with the definition of activity for ideal solutions. By the approximation of Henry's Law for the solutes, we have that $P_{i}=k_{i}x_{i}\gamma_{i}$, $1\leq i\leq c$, (convention (II)), see \cite{M}, so that, by the approximation of Henry's Law, phase equilibrium and the gas law with fugacity $\delta_{i}$;\\

$\mu_{i}^{(sol)}=\mu_{i}^{(g)\circ}+RT ln({\delta_{i}P_{i}\over P^{\circ}})$\\

$=\mu_{i}^{(g)\circ}+RTln({k_{i}x_{i}\gamma_{i}\over P^{\circ}})+RTln(\delta_{i}(T,P))$\\

$=\mu_{i}^{(H)\circ}+RTln(\gamma_{i}x_{i})+RTln(\delta_{i}(T,P))$ $(*)$\\

where $\mu_{i}^{(H)\circ}=\mu_{i}^{(g)\circ}+RTln({k_{i}\over P^{\circ}})$ $(\dag)$. From $(*)$, we obtain that;\\

$\mu_{i}^{\circ}=\mu_{i}^{(H)\circ}+RTln(\gamma_{i}x_{i}(T,P^{\circ}))+RTln(\delta_{i}(T,P^{\circ})$ $(\dag\dag)$\\

and, for $1\leq i\leq c$;\\

$\mu_{i}=\mu_{i}^{\circ}-RTln(\gamma_{i}x_{i}(T,P^{\circ}))-RTln(\delta_{i}(T,P^{\circ})+RTln(\gamma_{i}x_{i})+RTln(\delta_{i}(T,P))$\\

$=\mu_{i}^{\circ}+RTln(\gamma_{i}x_{i})+\kappa(T,P)$ $(\dag\dag\dag)$\\

where, from phase equilibrium and $(\dag),(\dag\dag)$;\\

$\kappa(T,P)=-RTln(\gamma_{i}x_{i}(T,P^{\circ}))+RT(ln(\delta_{i}(T,P))-ln(\delta_{i}(T,P^{\circ}))$\\

$=\mu_{i}^{(H)\circ}-\mu_{i}^{\circ}+RTln(\delta_{i}(T,P^{\circ}))+RTln({\delta_{i}(T,P)\over \delta_{i}(T,P^{\circ})})$\\

$=\mu_{i}^{(g)\circ}+RTln({k_{i}\over P^{\circ}})-\mu_{i}^{(sol)\circ}+RTln(\delta_{i}(T,P))$\\

$=RTln({k_{i}\over P^{\circ}})+RTln(\delta_{i}(T,P))$\\

$=RTln({k_{i}\delta_{i}(T,P)\over P^{\circ}})$\\

We can measure the correction in Raoult's law for the solvent by $P_{0}=\gamma_{0}P_{0}^{*}x_{0}$, (convention I), see \cite{M}. Then, using the gas law with fugacity $\delta_{0}$, and the correction $\sigma$ for the difference of the chemical potential between a gas in a non ideal mixture and on its own, we have, at equilibrium, that;\\

$\mu_{0}^{(sol)}=\mu_{0}^{\circ(g)}+RTln({\delta_{0}P_{0}\over P^{\circ}})$\\

$=\mu_{0}^{\circ(g)}+RTln({\delta_{0}\gamma_{0}P_{0}^{*}x_{0}\over P^{\circ}})$\\

$\mu_{0}^{\circ(g)}+RTln(\gamma_{0}x_{0})+RTln({\delta_{0}P_{0}^{*}\over P^{\circ}})$\\

$=\mu_{0}(T,P_{0}^{*})+RTln(\gamma_{0}x_{0})$\\

$=\mu_{0}^{*}(T,P_{0}^{*})+\sigma(T,P_{0}^{*})+RTln(\gamma_{0}x_{0})$\\

$=\mu_{0}^{*}(T,P)+RTln(\gamma_{0}x_{0})+\theta(T,P)$ $(**)$\\

where $\theta(T,P)=\mu_{0}^{*}(T,P_{0}^{*})-\mu_{0}^{*}(T,P)+\sigma(T,P_{0}^{*})$\\

so that;\\

$RTln(\gamma_(0)x_{0})=RTln({\delta_{0}P_{0}\over P^{\circ}})-RTln({\delta_{0}P_{0}^{*}\over P^{\circ}})$\\

and;\\

$\mu_{0}=\mu_{0}^{*}+RTln(\gamma_{0}x_{0})+\theta$\\

$=\mu_{0}^{*}+RTln({\delta_{0}P_{0}\over P^{\circ}})-RTln({\delta_{0}P_{0}^{*}\over P^{\circ}})+\theta$ $(***)'$\\

We also have, using the phase rule for the solvent in equilibrium, that;\\

$\mu_{0}^{(g)}=\mu_{0}^{\circ(g)}+RTln({\delta_{0}P_{0}\over P^{\circ}})$\\

$\mu_{0}^{(sol)}=\mu_{0}^{\circ(sol)}+RTln({\delta_{0}P_{0}\over P^{\circ}})$ $(*)$\\

Combining $(*),(***)'$, we obtain that;\\

$\mu_{0}^{*}=\mu_{0}-RTln({\delta_{0}P_{0}\over P^{\circ}})+RTln({\delta_{0}P_{0}^{*}\over P^{\circ}})-\theta$\\

$=(\mu_{0}^{\circ}+RTln({\delta_{0}P_{0}\over P^{\circ}}))-RTln({\delta_{0}P_{0}\over P^{\circ}})+RTln({\delta_{0}P_{0}^{*}\over P^{\circ}})-\theta$\\

$=\mu_{0}^{\circ}+RTln({\delta_{0}P_{0}^{*}\over P^{\circ}})-\theta$ $(\dag)'$\\

Letting $\delta_{0}P_{0}^{*}=P^{\circ}$, we obtain that $\mu_{0}^{*}(T,P')=\mu_{0}^{\circ}-\theta$, for the corresponding $P'$, $(****)'$\\

From $(**)$, $(****)'$, we obtain that;\\

$\mu_{0}=\mu_{0}^{*}+RTln(\gamma_{0}x_{0})+\theta$\\

$=\mu_{0}^{*}(T,P')+\delta+RTln(\gamma_{0}x_{0})+\theta$\\

$=\mu_{0}^{\circ}-\theta+\delta+RTln(\gamma_{0}x_{0})+\theta$\\

$=\mu_{0}^{\circ}+RTln(\gamma_{0}x_{0})+\delta$\\

where $\delta=\mu_{0}^{*}(T,P)-\mu_{0}^{*}(T,P')\simeq 0$\\

Using the same calculation as before, we have that;\\

$\mu_{0}=\mu_{0}^{\circ}+RTln(\gamma_{0}x_{0})+\lambda(P)$\\

We can define a new activity coefficient by $Z=\prod_{i=0}^{c}b_{i}^{\nu_{i}}$, where;\\

$b_{0}=x_{0}$\\

$b_{i}=x_{i}$, $1\leq i\leq c$\\

From the above, we have that;\\

$\mu_{0}(T,P)=\mu_{0}^{\circ}+RTln(x_{0})+RTln(\gamma_{0}(T,P))+\delta_{0}(P)$\\

$=\mu_{0}^{\circ}+RTln(b_{0})+\phi_{0}(T,P)$\\

and, for $1\leq i\leq c$;\\

$\mu_{i}(T,P)=\mu_{i}^{\circ}+RTln(x_{i})+RTln(\gamma_{i}(T,P))+\kappa_{i}(T)$\\

$=\mu_{i}^{\circ}+RTln(b_{i})+\psi_{i}(T,P)$\\

\end{defn}

\begin{lemma}
\label{fugacity3}
In the case of dilute solutions, with interaction of the solvent, a feasible path $\gamma$ is a dynamic equilibrium path iff $pr(\gamma_{12})\subset C_{f}$, for some $f\in\mathcal{R}_{>0}$ iff ${dZ\over dt}=0$.
\end{lemma}
\begin{proof}
We have that;\\

$\prod_{i=0}^{c}a_{i}^{\nu_{i}}=\prod_{i=0}^{c}x_{i}^{\nu_{i}}$\\

$=Z=f$\\

Now copy the proof of Lemma \ref{henrys3}.

\end{proof}

We reformulate Lemmas \ref{gibbs2} to \ref{rates} in this context, assuming the approximation to Henry's law for the solutes and the approximation to Raoult's law for the solvent;\\

\begin{lemma}
\label{fugacity5}
In the dilute solution case, with interaction of the solvent, for the energy function $G$ involving  $c+1$ uncharged species;\\

$({\partial G\over \partial \xi})_{T,P}=\Delta G^{\circ}+RTln(Z)+\epsilon(T,P)$\\

where $\epsilon(T,P)=\nu_{0}\phi_{0}(T,P)+\sum_{i=1}^{c}\nu_{i}\psi_{i}(T,P)$, and $\phi_{0}$ is the error term for the solvent in Definition \ref{catalyzers3}, and $\psi_{i}$, for $1\leq i\leq c$ are the error terms for the solutes in Definition \ref{catalyzers3}.
\end{lemma}
\begin{proof}
The proof is clear from Lemma \ref{gibbs2}.

\end{proof}
\begin{lemma}
\label{fugacity6}
For a dilute solution, with interaction of the solvent, we have, using the definition of $\epsilon(T,P)$ in Lemma \ref{fugacity5}, the error term $\phi_{0}(T,P)$ and the error terms $\psi_{i}(T,P)$, $1\leq i\leq c$ in Definition \ref{catalyzers3}, that the same results as Lemma \ref{equivalences2} hold, replacing $\epsilon(P)$ with $\epsilon(T,P)$ and $\delta$ with $\epsilon(T,P^{\circ})=\nu_{0}\phi_{0}(T,P^{\circ})+\sum_{i=1}^{c}\nu_{i}\psi_{i}(T,P^{\circ})$\\

\end{lemma}
\begin{proof}
The proof is clear from the proof of Lemma \ref{equivalences2}.

\end{proof}
\begin{lemma}
\label{fugacity7}
For a dilute solution, with interaction of the solvent, we have the same result as Lemma \ref{van't Hoff,Gibbs-Helmholtz2} hold, replacing $\epsilon(P(T))$ by $\epsilon(T,P(T))$ along the quasi-chemical equilibrium lines.

\end{lemma}
\begin{proof}
The proof is again clear from the proof of Lemma \ref{van't Hoff,Gibbs-Helmholtz2}

\end{proof}
\begin{lemma}
\label{fugacity8}
For a dilute solution, with interaction of the solvent, we have the same results as Lemma \ref{eqlines2} hold, replacing $\epsilon(P')$ by $\epsilon(T_{1},P')$. In particularly, if $\epsilon\neq 0$, we have that;\\

$Z(T,P)=e^{\epsilon ln({P\over P^{\circ}})-\epsilon(T,P)\over RT}$\\

and, if $\epsilon=0$, we have that;\\

$Z(T,P)=e^{-\epsilon(T,P)\over RT}$\\

\end{lemma}
\begin{proof}
The proof is again clear from the proof of Lemma \ref{eqlines2}.

\end{proof}
\begin{lemma}
\label{fugacity9}
For a dilute solution, with interaction of the solvent, we have that, if $\epsilon\neq 0$;\\

$grad(Z)(T,P)=(({-\epsilon ln({P\over P^{\circ}})\over RT^{2}}+{\epsilon(T,P)\over RT^{2}}-{{\partial \epsilon\over \partial T}(T,P)\over RT})({P\over P^{\circ}})^{\epsilon\over RT}e^{-\epsilon(T,P)\over RT},$\\

$({\epsilon P^{\circ}\over RTP}-{{\partial \epsilon\over \partial P}(T,P)\over RT})({P\over P^{\circ}})^{\epsilon\over RT}e^{-\epsilon(T,P)\over RT})$\\

and, if $\epsilon=0$;\\

$grad(Z)(T,P)=(({\epsilon(T,P)\over RT^{2}}-{{\partial \epsilon\over \partial T}(T,P)\over RT})e^{-\epsilon(T,P)\over RT},{-{\partial \epsilon\over \partial P}(T,P)\over RT}e^{-\epsilon(T,P)\over RT})$\\

The paths of maximal reaction in the region $|grad Z(T,P)|>1$, $Z(T,P)>0$, are given by implicit solutions to the differential equations;\\

${dP\over dT}={\epsilon TP^{\circ}-PT{\partial\epsilon\over \partial P}(T,P)\over (-\epsilon Pln({P\over P^{\circ}})+P\epsilon(T,P)-PT{\partial \epsilon\over \partial T}(T,P))}$\\

${dP\over dT}={-T{\partial \epsilon\over \partial P}(T,P)\over \epsilon(T,P)-T{\partial \epsilon\over \partial T}(T,P)}$\\

respectively.\\

\end{lemma}
\begin{proof}
The proof is the same as Lemma \ref{henrys9}, replacing $Q$ by $Z$, noting that $Z$ is defined the same way in terms of activities.\\

\end{proof}
\end{section}
\begin{section}{Electrochemistry with Error Terms, Fugacity and Interaction of the Solvent}
\label{electrochemistry}
Using the new error term $\epsilon(T,P)$ and the activity coefficient $Z$ from Section \ref{fug}, we have that;\\

\begin{lemma}{The Nernst Equation for Catalyzers}
\label{efi1}\\

At electrical chemical equilibrium $(T,P)$ and $(T,P^{\circ})$;\\

$(E-E^{\circ})(T,P)=-{RTln(Z(T,P))\over 4F}-{\epsilon(T,P)\over 4F}$\\

\end{lemma}
\begin{proof}
Just follow the proof of Lemma \ref{nernst2}, replacing $\epsilon(P)$ with $\epsilon(T,P)$ and use the fact that  the catalyzer reaction $2H_{2}O+4 e^{-}(R)\rightarrow O_{2}+2H_{2}+4e^{-}(L)$ occurs with $4$ electrons rather than $2$.\\

\end{proof}
\begin{lemma}
\label{efi2}

At electrical chemical equilibrium $(T,P)$ and $(T,P^{\circ})$, and chemical equilibrium $(T,P)$;\\

$\Delta G^{\circ}=4F(E-E^{0})$\\

\end{lemma}
\begin{proof}
Follow the proof of Lemma \ref{delta2}, replacing $\epsilon(P)$ with $\epsilon(T,P)$, noting the remark in Lemma \ref{efi1}.\\

\end{proof}
\begin{lemma}
\label{efi3}
If $\epsilon=0$, we have, for all $T_{1}>0$, that;\\

$({\partial G\over \partial \xi})_{T,P}|_{(T_{1},P_{1})}=({\partial G\over \partial \xi})_{T,P}|_{(T_{1},P_{1}^{\circ})}$\\

iff;\\

$E(T_{1},P_{1})=E(T_{1},P_{1}^{\circ})=E^{\circ}(T_{1})$\\

where $G$ is the Gibbs energy function for the charged and uncharged species.
\end{lemma}
\begin{proof}
Follow the proof of Lemma \ref{allt2}, replacing the result of Lemma 1.5, that $({\partial G_{chem'}\over \partial \xi})_{T,P}$ is independent of $P$,  with the corresponding same result in Lemma \ref{fugacity8}.
\end{proof}

\begin{lemma}
\label{efi4}
We have, for all $T_{1}>0,P_{1}>0$, that;\\

$4F(E(T_{1},P_{1})-E^{\circ}(T_{1}))=({\partial G\over \partial \xi})_{T,P}|_{(T_{1},P_{1})}-({\partial G\over \partial \xi})_{T,P}|_{(T_{1},P_{1}^{\circ})}-RT_{1}ln(Z(T_{1},P_{1}))-\epsilon(T_{1},P_{1})$ $(*)$\\

\end{lemma}
\begin{proof}
Follow the proof of Lemma \ref{alltelectrical2}, replacing $\epsilon(P)$ with $\epsilon(T,P)$.\\

\end{proof}
\begin{rmk}
\label{strategy}
The result of Lemma \ref{efi4} combined with the determination of the activity coefficient $Z$ in Lemma \ref{fugacity8} and the error term $\epsilon(T,P)$ in Lemma \ref{fugacity5} can be use to determine the unknown quantity $({\partial G\over \partial \xi})_{T,P}$. We can measure the potential difference between the cathode and anode along the dynamical equilibrium paths provided by Lemma \ref{fugacity8} and then use the formula $(*)$ in Lemma \ref{efi4}. Once this is determined, we then alter the power supply, in accordance with $(*)$, to push the reaction along the paths of maximal reaction given in Lemma \ref{fugacity9}. This should improve the efficiency of the production of hydrogen and oxygen, in the case of the catalyzer reaction, given by $2H_{2}O+4 e^{-}(R)\rightarrow O_{2}+2H_{2}+4e^{-}(L)$.\\
\end{rmk}
\end{section}
\begin{section}{Dilute solutions with Henry's Law for Solutes, Raoult's Law for the Solvent and No Solvent Interaction}
\label{nointeraction}
As mentioned in Definition \ref{constants}, we can consider an electrolyte as a solute in a dilute solution and define the activities $a_{i}$, $0\leq i\leq c$, by;\\

$a_{0}=x_{0}\simeq 1$\\

$a_{i}=x_{i}$ $(1\leq i\leq c)$\\

We can either define the activity coefficient by;\\

$W=\prod_{i=1}^{c}a_{i}^{\nu_{i}}$\\

or, use the more conventional definition;\\

$Q=a_{0}\prod_{i=1}^{c}a_{i}^{\nu_{i}}\simeq \prod_{i=1}^{c}a_{i}^{\nu_{i}}$\\

We will consider both cases.\\

\begin{rmk}
\label{zero}
In the first case, if $a_{0}$ is assumed constant, we have to redefine a feasible path by using coordinates $(T,P,n_{0},n_{1},\ldots,n_{c})$ and letting $\gamma:[0,1]\rightarrow \mathcal{R}_{>0}^{3+c}$, such that if $n_{i}(t)=pr_{3+i}(t)$, for $0\leq i\leq c$, then;\\

${n_{i}'\over \nu_{i}}={n_{j}'\over \nu_{j}}$, for $1\leq i<j\leq c$\\

where $\{\nu_{1},\ldots,\nu_{c}\}$ are the stoichiometric coefficients. If $n(t)=\sum_{i=0}^{c}n_{i}(t)$, and $x_{i}(t)=a_{i}(t)={n_{i}\over n}(t)$, $0\leq i\leq c$, then $Q(pr_{12}(t))=\prod_{i=1}^{p}a_{i}(t)^{\nu_{i}}$ and $n_{0}>0$ is a fixed constant. Note that $n>0$ and the $x_{i}$ are well defined, $0\leq i\leq c$. The existence of feasible paths follows easily from the proof of Lemma \ref{exists}, where we are free to take any $n_{0}>0$.\\

\end{rmk}

\begin{lemma}
\label{henrys2nointeraction}
In the case of dilute solutions, with no interaction of the solvent, a feasible path $\gamma$ is a dynamic equilibrium path iff $pr(\gamma_{12})\subset C_{f}$, for some $f\in\mathcal{R}_{>0}$ iff ${dW\over dt}=0$.
\end{lemma}
\begin{proof}
We have that;\\

$W=\prod_{i=1}^{c}a_{i}^{\nu_{i}}=\prod_{i=1}^{c}x_{i}^{\nu_{i}}$\\

$=W=f$ $(*)$\\

If $\gamma$ is a feasible path, then $pr_{12}(\gamma)\subset W_{>0}$, otherwise, we could find $(T,P)$, with $x_{i}(T,P)\leq 0$, contradicting the fact that $n_{i}>0$, $n>0$. It follows that $f>0$. With $f>0$, follow through the proof of Lemma \ref{implies}, replacing $\beta$ with $\sum_{i=0}^{c}d_{i}$, where $d_{0}=n_{0}$. Clearly $n_{0}'=0$ so we obtain the first direction. The rest of the proof follows from Lemma \ref{feasible}, using the additional fact that $n_{0}'=0$. \\

\end{proof}

\begin{lemma}
\label{henrys1}
In the case of dilute solutions, with no interaction of the solvent, and $a_{0}$ assumed constant, a feasible path $\gamma$ is a dynamic equilibrium path iff $pr(\gamma_{12})\subset C_{c}$, for some $c\in\mathcal{R_{>0}}$ iff ${dQ\over dt}=0$.
\end{lemma}
\begin{proof}
We have that;\\

$\prod_{i=1}^{c}a_{i}^{\nu_{i}}={Q\over a_{0}}$\\

$=\prod_{i=1}^{c}x_{i}^{\nu_{i}}$\\

$={c\over a_{0}}=d$\\

If $\gamma$ is a feasible path, then $pr_{12}(\gamma)\subset Q_{>0}$, otherwise, we could find $(T,P)$, with $x_{i}(T,P)\leq 0$, contradicting the fact that $n_{i}>0$, for $0\leq i\leq c$, $n>0$. It follows that $c>0$, $d>0$. With $d>0$, follow through the proof of Lemma \ref{implies} again, getting the other directions from Lemma \ref{feasible}.\\

\end{proof}
\begin{lemma}
\label{henrys2}
In the case of dilute solutions, with no interaction of the solvent, a feasible path $\gamma$ is a dynamic equilibrium path iff $pr(\gamma_{12})\subset C_{f}$, for some $f\in\mathcal{R}_{>0}$ iff ${dQ\over dt}=0$.
\end{lemma}
\begin{proof}
With the same caveat as in Lemma \ref{henrys1}, we have that;\\

$a_{0}\prod_{i=1}^{c}a_{i}^{\nu_{i}}=x_{0}\prod_{i=1}^{c}x_{i}^{\nu_{i}}$\\

$=\prod_{i=0}^{c}x_{i}^{\mu_{i}}$\\

$=Q=f$ $(*)$\\

where $x_{i}={n_{i}\over n}$, $\mu_{0}=1$, $\mu_{i}=\nu_{i}$, $1\leq i\leq c$.\\

Using the relation $(*)$, differentiating and using the facts that, for $1\leq i\leq c-1$;\\

$n_{i}'={\mu_{i}n_{c}'\over \mu_{c}}$, $n_{i}={\mu_{i}n_{c}\over \mu_{c}}+d_{i}$, $n_{0}'=0$, $n_{0}=d_{0}$\\

we obtain that;\\

$(\prod_{i=0}^{c}x_{i}^{\mu_{i}})'=\sum_{i=0}^{c}\mu_{i}x_{i}^{\mu_{i}-1}x_{i}'\prod_{j\neq i}x_{j}^{\mu_{j}}$\\

$=f\sum_{i=0}^{c}\mu_{i}x_{i}^{\mu_{i}-1}x_{i}'x_{i}^{-\mu_{i}}$\\

$=f\sum_{i=0}^{c}\mu_{i}x_{i}^{-1}x_{i}'$\\

$=0$\\

so that;\\

$=\sum_{i=0}^{c}\mu_{i}{n\over n_{i}}{(n_{i}'n-n_{i}n')\over n^{2}}$\\

$=\sum_{i=0}^{c}\mu_{i}({n_{i}'\over n_{i}}-{n'\over n})$\\

$=\sum_{i=1}^{c-1}{\mu_{i}^{2}n_{c}'\over \nu_{i}n_{c}+\mu_{c}d_{i}}+{\mu_{c}n_{c}'\over n_{c}}-
\lambda({\sum_{i=0}^{c}n_{i}'\over \sum_{i=0}^{c}n_{i}})$\\

$=\sum_{i=1}^{c-1}{\mu_{i}^{2}n_{c}'\over \mu_{i}n_{c}+\mu_{c}d_{i}}+{\mu_{c}n_{c}'\over n_{c}}-\lambda
({(\sum_{i=1}^{c-1}{\mu_{i}\over \mu_{c}}+1)n_{c}'\over (\sum_{i=1}^{c-1}{\nu_{i}\over \nu_{c}}+1)n_{c}+\sum_{i=0}^{c-1}d_{i}})$\\

$=0$ $(B)$\\

where $\lambda=\sum_{i=0}^{c}\mu_{i}=1+\sum_{i=1}^{c}\nu_{i}$\\

Following the proof of Lemma \ref{implies}, replacing $\beta$ with $\sum_{i=0}^{c-1}d_{i}$, we have, if $\sum_{i=1}^{c-1}\mu_{i}=\sum_{i=1}^{c-1}\nu_{i}\neq 0$ and $\lambda=1+\sum_{i=1}^{c}\mu_{i}=1+\sum_{i=1}^{c}\nu_{i}\neq 0$, then $n_{i}'=0$, for $1\leq i\leq c$, and clearly we have that $n_{0}'=0$. Similarly, if $\lambda=\sum_{i=0}^{c}\mu_{i}=1+\sum_{i=1}^{c}\nu_{i}=0$, we obtain the relation;\\

$\prod_{i=0}^{c}n_{i}^{\mu_{i}}=f$\\

Again, following the proof of Lemma \ref{implies}, if $n_{c}'\neq 0$, we obtain the relation;\\

$\sum_{i=1}^{c-1}{\mu_{i}^{2}\over \mu_{i}n_{c}+\mu_{c}d_{i}}+{\mu_{c}\over n_{c}}=\sum_{i=1}^{c-1}{\nu_{i}^{2}\over \nu_{i}n_{c}+\nu_{c}d_{i}}+{\nu_{c}\over n_{c}}$\\

$=0$\\

and, by the proof there, we obtain that $n_{i}'=0$, for $1\leq i\leq c$. As $n_{0}'=0$, we obtain the result. We are left with the case $\sum_{i=1}^{c-1}\mu_{i}=\sum_{i=1}^{c-1}\nu_{i}=0$. As in the proof of Lemma \ref{implies}, we can assume this choosing the appropriate pivot. The other directions in the Lemma follow from a simple modification of Lemma \ref{feasible}.

\end{proof}

\begin{lemma}
\label{henrys4nointeraction}
In Lemmas \ref{nernst} to \ref{alltelectrical}, for the standard cell, and considering a dilute solution with no interaction of the solvent, replacing $Q$ defined as $\prod_{i=1}^{c}a_{i}^{\nu_{i}}$ by the same definition of $W$;\\

If we assume without approximation that $\mu_{i}=\mu_{i}^{\circ}+RT ln(a_{i})$, $1\leq i\leq c$, then;\\

the same results as Lemma \ref{proofs}, setting $a_{0}(T,P)=1$, with the caveat that in the final claim $W=1$ and any path is a dynamical equilibrium line.\\
\end{lemma}
\begin{proof}
The proof is clear.
\end{proof}

\begin{lemma}
\label{henrys4}
In Lemmas \ref{nernst} to \ref{alltelectrical}, for the standard cell, and considering a dilute solution with no interaction of the solvent, replacing $Q$ defined as $\prod_{i=1}^{c}a_{i}^{\nu_{i}}$ by;\\

$a_{0}(\prod_{i=1}^{c}a_{i}^{\nu_{i}})=x_{0}(\prod_{i=1}^{c}a_{i}^{\nu_{i}})$\\

If we assume without approximation that $\mu_{i}=\mu_{i}^{\circ}+RT ln(a_{i})$, $0\leq i\leq c$, then;\\

the same results as Lemma \ref{proofs}, with the new definition of $a_{0}(T,P)=x_{0}(T,P)$.\\
\end{lemma}
\begin{proof}
The proof is clear.
\end{proof}

We reformulate Lemmas \ref{gibbs2} to \ref{rates} in this context, assuming Henry's law for the solutes and the solvent an ideal solution;\\

\begin{lemma}
\label{henrys5nointeraction}
In the dilute solution case, with no interaction of the solvent, for the energy function $G$ involving  $c+1$ uncharged species, including the solvent;\\

$({\partial G\over \partial \xi})_{T,P}=\Delta G^{\circ}+RTln(W)+\epsilon$\\

where $\epsilon(T,P)=\sum_{i=1}^{c}\nu_{i}\kappa_{i}(T)\simeq 0$ and $\kappa_{i}(T)\simeq 0$ are the error terms for the $i$'th uncharged species in Definition \ref{catalyzers2}, $1\leq i\leq c$.
\end{lemma}
\begin{proof}
The proof is clear from Lemma \ref{gibbs2} and the fact that, as $dn_{0}=0$;\\

$dG=\sum_{i=0}^{c}\mu_{i}dn_{i}=\sum_{i=1}^{c}\mu_{i}dn_{i}$\\

so that $({\partial G\over \partial \xi})_{T,P}=\sum_{i=1}^{c}\nu_{i}\mu_{i}$ and $\Delta G^{\circ}=\sum_{i=1}^{c}\nu_{i}\mu_{i}^{\circ}$.\\

\end{proof}
\begin{lemma}
\label{henrys6nointeraction}
For a dilute solution, with no interaction of the solvent, we have, using the definition of $\epsilon(T)=\sum_{i=1}^{c}\nu_{i}\kappa_{i}(T)$, where the error terms $\kappa_{i}(T)$, $1\leq i\leq c$ occur in Definition \ref{catalyzers2}, that the same results as Lemma \ref{equivalences2} hold, replacing $\epsilon(P)$ with $\epsilon(T)$ and $\delta$ with $\epsilon(T)=\sum_{i=1}^{c}\nu_{i}\kappa_{i}(T)$ again, and using $W$ instead of $Q$.\\

\end{lemma}
\begin{proof}
The proof is clear from the proof of Lemma \ref{equivalences2}, using the observation from Lemma \ref{henrys5nointeraction}.

\end{proof}
\begin{lemma}
\label{henrys7nointeraction}
For a dilute solution, with no interaction of the solvent, we have the same result as Lemma \ref{van't Hoff,Gibbs-Helmholtz2} hold, replacing $\epsilon(P(T))$ by $\epsilon(T)$ along the quasi-chemical equilibrium lines, and using $W$ instead of $Q$.

\end{lemma}
\begin{proof}
The proof is again clear from the proof of Lemma \ref{van't Hoff,Gibbs-Helmholtz2}.\\

\end{proof}
\begin{lemma}
\label{henrys8nointeraction}
For a dilute solution, with no interaction of the solvent, we have the same results as Lemma \ref{eqlines2} hold, replacing $\epsilon(P')$ by $\epsilon(T_{1})$. In particularly, if $\epsilon\neq 0$, we have that;\\

$W(T,P)=e^{\epsilon ln({P\over P^{\circ}})-\epsilon(T)\over RT}$\\

and, if $\epsilon=0$, we have that;\\

$W(T,P)=e^{-\epsilon(T)\over RT}$\\

\end{lemma}
\begin{proof}
The proof is clear from the proof of Lemma \ref{eqlines2}.

\end{proof}
\begin{lemma}
\label{henrys9nointeraction}
For a dilute solution, with no interaction of the solvent, we have that, if $\epsilon\neq 0$;\\

$grad(W)(T,P)=(({-\epsilon ln({P\over P^{\circ}})\over RT^{2}}+{\epsilon(T)\over RT^{2}}-{{d\epsilon\over d T}(T)\over RT})({P\over P^{\circ}})^{\epsilon\over RT}e^{-\epsilon(T)\over RT},$\\

$({\epsilon P^{\circ}\over RTP})({P\over P^{\circ}})^{\epsilon\over RT}e^{-\epsilon(T)\over RT})$\\

and, if $\epsilon=0$;\\

$grad(W)(T,P)=(({\epsilon(T)\over RT^{2}}-{{d\epsilon\over dT}(T)\over RT})e^{-\epsilon(T)\over RT},0)$\\

The paths of maximal reaction in the region $|grad(W)(T,P)|>1$, $W(T,P)>0$, are given by implicit solutions to the differential equations;\\

${dP\over dT}={\epsilon TP^{\circ}\over (-\epsilon Pln({P\over P^{\circ}})+P\epsilon(T)-PT{d\epsilon\over d T}(T))}$\\

${dP\over dT}=0$\\

respectively.\\

\end{lemma}
\begin{proof}
The computation follows easily from the proof of Lemma \ref{henrys9}, replacing $\epsilon(P,T)$ by $\epsilon(T)$. We also note that in the proof of maximal reaction, see Lemma \ref{rates}, we have to change $\beta$ to $\sum_{i=0}^{c+1}n_{i,0}$, where $n_{i,0}$ is the fixed molar amount of the solvent. This effects $\alpha_{1}$ but we still have that $\alpha_{1}\neq 0$ and the rest of the proof is unchanged.\\
\end{proof}
\begin{rmk}
\label{other}
We can also formulate versions of Lemmas \ref{henrys5nointeraction} to \ref{henrys9nointeraction} for the activity coefficient $Q$ instead of $W$, mentioned in the introduction to this section \ref{nointeraction}. However, although the proof should go through, it is more difficult, and left as an exercise for the reader, combining the methods of Sections \ref{errorterms} and \ref{errorterms2}. However, it seems unnecessary when we can derive the main results with the coefficient $W$.

\end{rmk}
\end{section}
\begin{section}{Dilute solutions with Fugacity and No Solvent Interaction}
\label{fug1}
Again, we can define activity coefficients either by $W=\prod_{i=1}^{c}b_{i}^{\nu_{i}}$, or the more conventional $Z=b_{0}\prod_{i=1}^{c}b_{i}^{\nu_{i}}$ where;\\

$b_{0}=x_{0}$\\

$b_{i}=x_{i}$, $1\leq i\leq c$\\

See the introductions to Section \ref{fug} and Section \ref{nointeraction} with the Remark \ref{zero}. We will again consider both cases.

\begin{lemma}
\label{fugacity2nointeraction}
In the case of dilute solutions, with no interaction of the solvent, a feasible path $\gamma$ is a dynamic equilibrium path iff $pr(\gamma_{12})\subset C_{f}$, for some $f\in\mathcal{R}_{>0}$ iff ${dW\over dt}=0$.
\end{lemma}
\begin{proof}
Copy the proof of Lemma \ref{henrys2nointeraction}.
\end{proof}

\begin{lemma}
\label{fugacity1}
In the case of dilute solutions, with no interaction of the solvent, and $b_{0}$ assumed constant, a feasible path $\gamma$ is a dynamic equilibrium path iff $pr(\gamma_{12})\subset C_{f}$, for some $f\in\mathcal{R}_{>0}$ iff ${dZ\over dt}=0$.
\end{lemma}
\begin{proof}
We have that;\\

$\prod_{i=1}^{c}b_{i}^{\nu_{i}}={Z\over b_{0}}$\\

$=\prod_{i=1}^{c}x_{i}^{\nu_{i}}$\\

$={f\over b_{0}}=d$\\

Now copy the proof of Lemma \ref{henrys1}.

\end{proof}
\begin{lemma}
\label{fugacity2}
In the case of dilute solutions, with no interaction of the solvent, a feasible path $\gamma$ is a dynamic equilibrium path iff $pr(\gamma_{12})\subset C_{f}$, for some $f\in\mathcal{R}_{>0}$ iff ${dZ\over dt}=0$.
\end{lemma}
\begin{proof}
We have that;\\

$b_{0}\prod_{i=1}^{c}b_{i}^{\nu_{i}}=x_{0}\prod_{i=1}^{c}x_{i}^{\nu_{i}}$\\

$=Z=f$\\

Now copy the proof of Lemma \ref{henrys2}.

\end{proof}

\begin{lemma}
\label{fugacity4nointeraction}
In Lemmas \ref{nernst} to \ref{alltelectrical}, for the standard cell, and considering a dilute solution with no interaction of the solvent, replacing $Q$ defined as $\prod_{i=1}^{c}a_{i}^{\nu_{i}}$ by the same definition of $W$;\\

If we assume without approximation that $\mu_{i}=\mu_{i}^{\circ}+RT ln(b_{i})$, $1\leq i\leq c$, then;\\

the same results as Lemma \ref{proofs}, setting $a_{0}(T,P)=1$, with the same caveat as Lemma \ref{henrys4nointeraction}.
\end{lemma}
\begin{proof}
The proof is clear, as in Lemma \ref{henrys4nointeraction}.\\
\end{proof}

\begin{lemma}
\label{fugacity4}
In Lemmas \ref{nernst} to \ref{alltelectrical}, for the standard cell, and considering a dilute solution with no interaction of the solvent, replacing $Q$ defined as $\prod_{i=1}^{c}a_{i}^{\nu_{i}}$ by;\\

$Z=b_{0}(\prod_{i=1}^{c}b_{i}^{\nu_{i}})=x_{0}(\prod_{i=1}^{c}b_{i}^{\nu_{i}})$\\

If we assume without approximation that $\mu_{i}=\mu_{i}^{\circ}+RT ln(b_{i})$, $0\leq i\leq c$, then;\\

the same results as Lemma \ref{proofs}, with the new definition of $b_{0}(T,P)=x_{0}(T,P)$ replacing $a_{0}(T,P)$ in the Lemma.\\
\end{lemma}
\begin{proof}
The proof is clear, as in Lemma \ref{henrys4}.\\
\end{proof}
We reformulate Lemmas \ref{gibbs2} to \ref{rates} in this context, assuming the approximation to Henry's law for the solutes and the approximation to Raoult's law for the solvent;\\

\begin{lemma}
\label{fugacity5nointeraction}
In the dilute solution case, with no interaction of the solvent, for the energy function $G$ involving  $c+1$ uncharged species;\\

$({\partial G\over \partial \xi})_{T,P}=\Delta G^{\circ}+RTln(W)+\epsilon(T,P)$\\

where $\epsilon(T,P)=\sum_{i=1}^{c}\nu_{i}\psi_{i}(T,P)$, and $\psi_{i}$, for $1\leq i\leq c$ are the error terms for the solutes in Definition \ref{catalyzers3}.
\end{lemma}
\begin{proof}
The proof is clear from Lemma \ref{gibbs2}, with the same observation as in Lemma \ref{henrys5nointeraction}.

\end{proof}
\begin{lemma}
\label{fugacity6nointeraction}
For a dilute solution, with no interaction of the solvent, we have, using the definition of $\epsilon(T,P)$ in Lemma \ref{fugacity5nointeraction}, the error terms $\psi_{i}(T,P)$, $1\leq i\leq c$ in Definition \ref{catalyzers3}, that the same results as Lemma \ref{equivalences2} hold, replacing $\epsilon(P)$ with $\epsilon(T,P)$ and $\delta$ with $\epsilon(T,P^{\circ})=\sum_{i=1}^{c}\nu_{i}\psi_{i}(T,P^{\circ})$, using $W$ instead of $Q$.\\

\end{lemma}
\begin{proof}
The proof is clear from the proof of Lemma \ref{equivalences2}, using Lemma \ref{fugacity5nointeraction}.\\

\end{proof}
\begin{lemma}
\label{fugacity7nointeraction}
For a dilute solution, with no interaction of the solvent, we have the same results as Lemma \ref{van't Hoff,Gibbs-Helmholtz2} hold, replacing $\epsilon(P(T))$ by $\epsilon(T,P(T))$ along the quasi-chemical equilibrium lines.

\end{lemma}
\begin{proof}
The proof is again clear from the proof of Lemma \ref{van't Hoff,Gibbs-Helmholtz2} and Lemma \ref{fugacity6nointeraction}.\\

\end{proof}
\begin{lemma}
\label{fugacity8nointeraction}
For a dilute solution, with no interaction of the solvent, we have the same results as Lemma \ref{eqlines2} hold, replacing $\epsilon(P')$ by $\epsilon(T_{1},P')$. In particularly, if $\epsilon\neq 0$, we have that;\\

$W(T,P)=e^{\epsilon ln({P\over P^{\circ}})-\epsilon(T,P)\over RT}$\\

and, if $\epsilon=0$, we have that;\\

$W(T,P)=e^{-\epsilon(T,P)\over RT}$\\

\end{lemma}
\begin{proof}
The proof is again clear from the proof of Lemma \ref{eqlines2} and Lemma \ref{fugacity7nointeraction}.\\

\end{proof}
\begin{lemma}
\label{fugacity9nointeraction}
For a dilute solution, with no interaction of the solvent, we have that, if $\epsilon\neq 0$;\\

$grad(W)(T,P)=(({-\epsilon ln({P\over P^{\circ}})\over RT^{2}}+{\epsilon(T,P)\over RT^{2}}-{{\partial \epsilon\over \partial T}(T,P)\over RT})({P\over P^{\circ}})^{\epsilon\over RT}e^{-\epsilon(T,P)\over RT},$\\

$({\epsilon P^{\circ}\over RTP}-{{\partial \epsilon\over \partial P}(T,P)\over RT})({P\over P^{\circ}})^{\epsilon\over RT}e^{-\epsilon(T,P)\over RT})$\\

and, if $\epsilon=0$;\\

$grad(Z)(W,P)=(({\epsilon(T,P)\over RT^{2}}-{{\partial \epsilon\over \partial T}(T,P)\over RT})e^{-\epsilon(T,P)\over RT},{-{\partial \epsilon\over \partial P}(T,P)\over RT}e^{-\epsilon(T,P)\over RT})$\\

The paths of maximal reaction in the region $|grad W(T,P)|>1$, $W(T,P)>0$, are given by implicit solutions to the differential equations;\\

${dP\over dT}={\epsilon TP^{\circ}-PT{\partial\epsilon\over \partial P}(T,P)\over (-\epsilon Pln({P\over P^{\circ}})+P\epsilon(T,P)-PT{\partial \epsilon\over \partial T}(T,P))}$\\

${dP\over dT}={-T{\partial \epsilon\over \partial P}(T,P)\over \epsilon(T,P)-T{\partial \epsilon\over \partial T}(T,P)}$\\

respectively.\\

\end{lemma}
\begin{proof}
The proof is the same as Lemma \ref{henrys9}, using $W$ instead of $Z$, noting that the error term $\epsilon(T,P)$ depends on $P$, unlike Lemma \ref{henrys9nointeraction}.\\

\end{proof}

\begin{rmk}
\label{other1}
Again we can formulate Lemmas \ref{fugacity5nointeraction} to \ref{fugacity9nointeraction} using $Q$ instead of $W$, see Remark \ref{other}. The results of the section might be useful in the production of ethanol, using $H_{2}O$ as the solvent, by varying the temperature and pressure of the reaction, $C_{6}H_{12}O_{6}\rightarrow 2C_{2}H_{5}OH+2CO_{2}$, see \cite{MJ} as well.\\

\end{rmk}

\end{section}

\begin{section}{Electrochemistry with Error Terms, Fugacity and No Interaction of the Solvent}
\label{final}
Using the new error term $\epsilon(T,P)$ and the activity coefficient $W$ from Section \ref{fug1}, we have that;\\

\begin{lemma}{The Nernst Equation for the Standard Cell}\\
\label{efi11}\\

At electrical chemical equilibrium $(T,P)$ and $(T,P^{\circ})$ for the standard cell;\\

$(E-E^{\circ})(T,P)=-{RTln(W(T,P))\over 2F}-{\epsilon(T,P)\over 2F}$\\

\end{lemma}
\begin{proof}
Just use the proof of Lemma \ref{nernst2}, replacing $\epsilon(P)$ with the error term $\epsilon(T,P)$ from Section \ref{fug1}, noting that for the Gibbs funcion involving $c+1$ species, including the solvent, $({\partial G\over \partial \xi})_{T,P}=\sum_{i=1}^{c}\nu_{i}\mu_{i}$ again.\\

\end{proof}
\begin{lemma}
\label{efi21}

At electrical chemical equilibrium $(T,P)$ and $(T,P^{\circ})$, and chemical equilibrium $(T,P)$;\\

$\Delta G^{\circ}=2F(E-E^{0})$\\

\end{lemma}
\begin{proof}
Follow the proof of Lemma \ref{delta2}, replacing $\epsilon(P)$ with $\epsilon(T,P)$.\\

\end{proof}
\begin{lemma}
\label{efi31}
If $\epsilon=0$, we have, for all $T_{1}>0$, that;\\

$({\partial G\over \partial \xi})_{T,P}|_{(T_{1},P_{1})}=({\partial G\over \partial \xi})_{T,P}|_{(T_{1},P_{1}^{\circ})}$\\

iff;\\

$E(T_{1},P_{1})=E(T_{1},P_{1}^{\circ})=E^{\circ}(T_{1})$\\

where $G$ is the Gibbs energy function for the $c+1$ charged and uncharged species.
\end{lemma}
\begin{proof}
Follow the proof of Lemma \ref{allt2}, replacing the result of Lemma 1.5, that $({\partial G_{chem'}\over \partial \xi})_{T,P}$ is independent of $P$,  with the corresponding result in Lemma \ref{fugacity8nointeraction}.
\end{proof}

\begin{lemma}
\label{efi41}
We have, for all $T_{1}>0,P_{1}>0$, that;\\

$2F(E(T_{1},P_{1})-E^{\circ}(T_{1}))=({\partial G\over \partial \xi})_{T,P}|_{(T_{1},P_{1})}-({\partial G\over \partial \xi})_{T,P}|_{(T_{1},P_{1}^{\circ})}-RT_{1}ln(Z(T_{1},P_{1}))-\epsilon(T_{1},P_{1})$ $(*)$\\

\end{lemma}
\begin{proof}
Follow the proof of Lemma \ref{alltelectrical2}, replacing $\epsilon(P)$ with $\epsilon(T,P)$.\\

\end{proof}
\begin{rmk}
\label{strategy}
The result of Lemma \ref{efi41} combined with the determination of the activity coefficient $W$ in Lemma \ref{fugacity8nointeraction} and the error term $\epsilon(T,P)$ in Lemma \ref{fugacity5nointeraction} can be used to determine the unknown quantity $({\partial G\over \partial \xi})_{T,P}$. We can measure the potential difference between the cathode and anode along the dynamical equilibrium paths provided by Lemma \ref{fugacity8nointeraction} and then use the formula $(*)$ in Lemma \ref{efi41}. Once this is determined, we then alter the power supply, in accordance with $(*)$, to push the reaction along the paths of maximal reaction given in Lemma \ref{fugacity9nointeraction}. This should improve the efficiency of the production of hydrogen, in the case of the reaction, given by $H_{2}+2AgCl+2e^{-}(R)\rightarrow 2HCl+2Ag+2e^{-}(L)$, where we use water $H_{2}O$ as a solvent with no interaction.\\
\end{rmk}
\end{section}
\begin{section}{The Constant Change in Enthalpy Assumption}
\label{enthalpy}
\begin{lemma}
\label{constancy}
In Lemma \ref{vanhoffhelmholtz}, without the assumption that $\Delta H^{\circ}$ is constant, we obtain that;\\

For $c\in\mathcal{R}$, if $D_{c}$ intersects the line $P=P^{\circ}$ at $(T_{1},P^{\circ})$, then, for $(T_{2},P)\in D_{c}$;\\

$\Delta G^{\circ}(T_{1})={T_{1}\over T_{2}}\Delta G^{\circ}(T_{2})-\Delta H^{\circ,ref}({T_{1}\over T_{2}}-1)+w(T_{1},T_{2})$\\

where $w(T_{1},T_{2})$ is an error term, with the property that, for fixed $T_{1}<T_{2}$, $w(T_{1},T_{2})\rightarrow 0$, as $m_{mix}\rightarrow \infty$, where $m_{mix}$ is the mass of the mixture and $\Delta H^{\circ,ref}$ is a reference value for $\Delta H^{\circ}$.\\

The results of Lemma \ref{eqlines} hold in the limit as $m_{mix}\rightarrow \infty$ and we can compute the error term effectively as $m_{mix}$ increases, given bounds on the heat capacities $C_{i}$ in the range $(T_{1},T_{2})$ of the proof of the Lemma.\\

In Lemma \ref{van't Hoff,Gibbs-Helmholtz2}, without the assumption that $\Delta H^{\circ}$ is constant, we again obtain that;\\

For $c\in\mathcal{R}$, if $D_{c}$ intersects the line $P=P^{\circ}$ at $(T_{1},P^{\circ})$, then, for $(T_{2},P)\in D_{c}$;\\

$\Delta G^{\circ}(T_{1})={T_{1}\over T_{2}}\Delta G^{\circ}(T_{2})-\Delta H^{\circ,ref}({T_{1}\over T_{2}}-1)+w(T_{1},T_{2})$\\

Again, the results of Lemma \ref{eqlines2} hold in the limit as $m_{mix}\rightarrow \infty$. The results of Lemma \ref{rates} still hold in the limit, and we can we can compute the error term effectively as $m_{mix}$ increases. In Lemma \ref{proofs}, we obtain, the same, that;\\

For $c\in\mathcal{R}$, if $D_{c}$ intersects the line $P=P^{\circ}$ at $(T_{1},P^{\circ})$, then, for $(T_{2},P)\in D_{c}$;\\

$\Delta G^{\circ}(T_{1})={T_{1}\over T_{2}}\Delta G^{\circ}(T_{2})-\Delta H^{\circ,ref}({T_{1}\over T_{2}}-1)+w(T_{1},T_{2})$\\

The same claims hold in the limit, with the computation of the error term at the end of Lemma \ref{proofs}, in Lemmas \ref{henrys7} to \ref{henrys9} and Lemmas \ref{fugacity7} to \ref{fugacity9}, Lemmas \ref{henrys6nointeraction} to \ref{henrys9nointeraction} and Lemmas \ref{fugacity6nointeraction} to \ref{fugacity9nointeraction}. The results of Section \ref{electrochemistry} are not effected, bearing in mind the use of the activity coefficient $Z$ in the final remark, similarly for Sections \ref{idealelectrochemistry} and \ref{final}.

\end{lemma}
\begin{proof}
We have, by the definition of enthalpy, the laws of differentials, the first law of thermodynamics, with $dP=0$, that;\\

$dH=dU+PdV+VdP$\\

$=dU+PdV$\\

$=(dQ-PdV)+PdV$\\

$=dQ$\\

where $Q$ is the internal energy including work, or heat. As in Kirchoff's Law of Thermodynamics, we have that;\\

$\Delta H^{\circ}(T)=\Delta H^{\circ}(T_{0})+\int_{T_{0}}^{T}{d(\Delta H^{\circ}(T))\over dT}dT$\\

$=\Delta H^{\circ}(T_{0})+\int_{T_{0}}^{T}{d(\Delta Q(T))\over dT}dT$\\

$=\Delta H^{\circ}(T_{0})+\int_{T_{0}}^{T}C(T)dT$ $(M)$\\

where $C(T)=\Delta({dQ\over dT})(T)$ is the change in the heat capacity of the mixture, after one mole of reaction. If we confine ourselves to a small temperature range, we can, therefore assume that $\Delta H^{\circ}$ is approximately temperature independent. However, we can also add a solvent to the reaction, to lower the magnitude of $C(T)$, as the heat capacity before and after the reaction would approach that of the solvent, and extend the temperature range of the reaction. More precisely, we have, by the law of mixtures for heat capacities, the fact that $m_{mix}$ is conserved during $1$ mole of reaction, that;\\

$C(T)=C_{fin}(T)-C_{in}(T)$\\

$={1\over m_{mix}}(\sum_{i=1}^{c}m_{i,fin}C_{i}(T)-\sum_{i=1}^{c}m_{i,in}C_{i}(T))$\\

$={N_{A}\over m_{mix}}(\sum_{i=1}^{c}n_{i,fin}m_{i,molec}C_{i}(T)-\sum_{i=1}^{c}n_{i,in}m_{i,molec}C_{i}(T))$\\

$={N_{A}\over m_{mix}}(\sum_{i=1}^{c}(n_{i,in}-\nu_{i})m_{i,molec}C_{i}(T)-\sum_{i=1}^{c}n_{i,in}m_{i,molec}C_{i}(T))$\\

$=-{N_{A}\over m_{mix}}(\sum_{i=1}^{c}\nu_{i}m_{i,molec}C_{i}(T))$\\

$=-{1\over m_{mix}}(\sum_{i=1}^{c}\nu_{i}m_{i,mol}C_{i}(T))$, $(N)$\\

where $\{m_{i,in},m_{i,fin},m_{i,molec},m_{i,mol}\}$ denote the initial, final, molecular and molar masses of substance $i$ respectively, $\{C_{i},C_{in},C_{fin}\}$ denote the heat capacity of substance $i$, the initial heat capacity of the mixture and the final heat capacity of the mixture respectively. In particular, we see that, as $m_{mix}\rightarrow \infty$, which we can achieve, for example, by increasing the solvent, $C(T)\rightarrow 0$ as well. This idea is pursued in Sections \ref{errorterms2} to \ref{final}, and doesn't alter the calculation of maximal reaction and equilibrium paths. We then have, by $(M),(N)$, that;\\

$\Delta H^{\circ}(T)=\Delta H^{\circ}(T_{0})-\int_{T_{0}}^{T}{1\over m_{mix}}(\sum_{i=1}^{c}\nu_{i}m_{i,mol}C_{i}(S))dS$\\

$=\Delta H^{\circ}(T_{0})+v(T_{0},T)$\\

where, for fixed $T$, $v(T_{0},T)\rightarrow 0$, as $m_{mix}\rightarrow \infty$. Fixing $T_{0}$, we can write this as;\\

$\Delta H^{\circ}(T)=\Delta H^{\circ,ref}+v(T)$\\

where, for fixed $T$, $v(T)\rightarrow 0$, as $m_{mix}\rightarrow \infty$.\\

For $0<T_{1}<T_{2}$, we let;\\

$v(T_{1},T_{2})=\int_{T_{1}}^{T_{2}}{v(S)\over S^{2}}dS$\\

so that, for fixed $\{T_{1},T_{2}\}$, as the convergence above is unform on $(T_{1},T_{2})$, $v(T_{1},T_{2})\rightarrow 0$, as $m_{mix}\rightarrow \infty$.\\

Following through the proof of Lemma \ref{vanhoffhelmholtz}, replacing $\Delta H^{\circ}$ by $\Delta H^{\circ,ref}+v(T)$, we obtain that;\\

${\Delta G^{\circ}(T_{2})\over T_{2}}-{\Delta G^{\circ}(T_{1})\over T_{1}}=(\Delta H^{\circ,ref}-\Delta G^{\circ}(T_{1}))({1\over T_{2}}-{1\over T_{1}})-v(T_{1},T_{2})$\\

to obtain;\\

$\Delta G^{\circ}(T_{1})={T_{1}\over T_{2}}\Delta G^{\circ}(T_{2})-\Delta H^{\circ,ref}({T_{1}\over T_{2}}-1)+w(T_{1},T_{2})$\\

where $w(T_{1},T_{2})=T_{1}v(T_{1},T_{2})\rightarrow 0$ as $m_{mix}\rightarrow \infty$, $(*)$\\

The same proof works in Lemma \ref{van't Hoff,Gibbs-Helmholtz2}, which included fixed error terms from Raoult's law, and in Lemma \ref{proofs}, where we use the activity $a_{0}(T,P)$.\\

The computation of $({\partial G\over \partial \xi})_{T,P}$ in terms of $\{\sigma,\epsilon,\beta\}$, in the proof of Lemma \ref{eqlines} is the same. When we compute the activity coefficient $Q$ at $(T,P)$, we have to take into account the intersection of $D_{c}$ at $(T_{1},P^{\circ})$ when determining $\Delta G^{\circ}(T_{2})$, using $(*)$ and equating coefficients. Clearly, we can determine the error $w(T_{1},T)\rightarrow 0$ from a bound in the heat capacities in the range $(T_{1},T)$, using $(M),(N)$, as we increase $m_{mix}$. We can then determine the corresponding error from the proof in the coefficients $\{\sigma,\epsilon,\beta\}$ of $({\partial G\over \partial \xi})_{T,P}$, and, then similarly for the activity coefficient $Q$ at $(T,P)$. Clearly, we can then effectively measure the vanishing deviation of a given dynamical equilibrium line from that in the formula of Lemma \ref{eqlines}. The same correction applies in Lemma \ref{eqlines2}. As we can compute the vanishing error term effectively in Lemma \ref{eqlines2}, as for Lemma \ref{eqlines}, the results of Lemma \ref{rates} still hold, with a corresponding vanishing error. The same argument applies for the end of Lemma \ref{proofs}, bearing in mind that we are using $c+1$ substances, when computing the error term. Similarly, in Lemmas \ref{henrys7} to \ref{henrys9}, Lemmas \ref{fugacity7} to \ref{fugacity9}, Lemmas \ref{henrys6nointeraction} to \ref{henrys9nointeraction} and Lemmas  \ref{fugacity6nointeraction} to \ref{fugacity9nointeraction}, replacing the fixed error terms. The claims for Sections \ref{electrochemistry}, \ref{idealelectrochemistry} and \ref{final} are clear.

\end{proof}

\end{section}
\begin{section}{Independence of Path and Existence}
\label{path}

\begin{lemma}
If $\lambda\neq 0$, see Lemma \ref{eqlines}, then no substance is formed in a loop. With the assumption that $\Delta H^{circ}$ is constant, we have that $\lambda=\Delta H^{\circ}-\epsilon ln(P^{\circ})$. If $D_{c}$ is a quasi-chemical equilibrium line in the theoretical limit, which we have computed, intersecting $P=P^{\circ}$ at $(T_{1},P^{\circ})$, and projecting onto an interval $(T_{1},T_{2})$, with $T_{1}<T_{2}$, then, making $w(T_{1},T_{2})\rightarrow 0$, see Section \ref{enthalpy}, by increasing the mass of the mixture, we have that $\lambda(T_{1},T_{2})\rightarrow\Delta H^{\circ,ref}(T_{1},T_{2})-\epsilon(T_{1},T_{2}) ln(P^{\circ})$, where $\epsilon((T_{1},T_{2}))$ can be effectively determined.

\end{lemma}

\begin{proof}
Suppose that an amount of substance $\xi$ is formed in a loop. We have, by Lemma \ref{differential}, that $dG=-SdT+VdP+\sum_{i=1}^{c}\mu_{i}dn_{i}$, $(A)$, and, by the definition of enthalpy in Definition \ref{constants}, that;\\

$dH=d(U+PV)$\\

$=dU+PdV+VdP$ $(\dag)$\\

$=d(G+TS)$\\

$=dG+TdS+SdT$ $(\dag\dag)$\\

We then have that, for a closed path $\gamma$, using $(A)$, $(\dag)$, $(\dag\dag)$, the definition of entropy as $dS={dQ\over T}$, the first law of thermodynamics, $dQ=dU+pdV$, the calculation of $({\partial G\over \partial \xi})_{(T,P)}$;\\

$\int_{\gamma}(dG+SdT-VdP)$\\

$=\int_{\gamma}((dH-TdS-SdT)+SdT-VdP)$\\

$=\int_{\gamma}(dH-TdS-VdP)$\\

$=\int_{\gamma}(dH-dQ-VdP)$\\

$=\int_{\gamma}((dU+PdV+VdP)-dQ-VdP)$\\

$=\int_{\gamma}(dU+PdV-dQ)$\\

$=\int_{\gamma}((dQ-PdV)+PdV-dQ)$\\

$=0$\\

$=\int_{\gamma}(\sum_{i=1}^{c}\mu_{i}dn_{i})$\\

$=\int_{\gamma}(\sum_{i=1}^{c}\mu_{i}\nu_{i}d\xi$\\

$=\int_{\gamma}({\partial G\over \partial \xi})_{(T,P)}d\xi$\\

$=\int_{\gamma}(\lambda+\epsilon ln(P)+\beta T+\sigma ln(T))d\xi$\\

$=\int_{\gamma}(\lambda+\epsilon ln(P)+\beta T+\sigma ln(T))d\xi$\\

$=\lambda \xi+\int_{\gamma}(\epsilon ln(P)\theta(P)dP+(\beta T+\sigma ln(T))\phi(T)dT)$ $(C)$\\

where $\theta(P)=\theta_{1}(P)$ along $\gamma_{1}$, $\theta_{1}(P)dP=d\xi|_{\gamma_{1}}$, $\theta(P)=\theta_{2}(P)$ along $\gamma_{2}$, $\theta_{2}(P)dP=d\xi|_{\gamma_{2}}$,  $\phi(T)=\phi_{1}(T)$ along $\gamma_{1}$, $\phi_{1}(T)dT=d\xi|_{\gamma_{1}}$, $\phi(T)=\phi_{2}(T)$ along $\gamma_{2}$, $\phi_{2}(T)dT=d\xi|_{\gamma_{2}}$.\\

so that, by Stokes Theorem;\\

$\int_{\gamma}(\epsilon ln(P)\theta(P)dP+\beta T\phi(T)dT)$\\

$=\int\int_{R}({\partial((\beta T+\sigma ln(T))\phi(T))\over \partial P}-{\partial (\epsilon ln(P)\theta(P))\over \partial T})dTdP$\\

$=0$ $(D)$\\

and, by $(C)$, $\xi=0$, if $\lambda\neq 0$. The second claim, with the assumption that $\Delta H^{\circ}$ is constant, follows from the proof of Lemma \ref{eqlines}, and in later sections, when we introduce error terms. The next claim follows easily as $W(T_{1},T_{2})\rightarrow 0$, see Section \ref{enthalpy}, and the computation of $\lambda$ in the limit. The computation of $\epsilon(T_{1},T_{2})$ is given in Lemma \ref{eqlines}.

\end{proof}

We take it as reasonable then, that if there are $2$ distinct feasible paths between $(T_{0},P_{0})$ and $(T,P)$, with a given initial condition $(n_{1,0},\ldots n_{c,0})$ at $(T_{0},P_{0})$, then the extent of the reaction $\xi$ determined by the paths $\{\gamma_{1},\gamma_{2}\}$ should be the same, $(*)$.  If this were not the case, then reversing one of the paths, we could obtain a reaction extent at $(T_{0},P_{0})$ along a loop $\gamma$. Even if $\lambda=0$, using the slight variation in the volume of liquids along a reaction path, and the fact that we return to the original pressure $P_{0}$ in a loop, we would have that;\\

$\Delta H=\int_{\gamma}dH$\\

$=\int_{\gamma}d(U+VP)$\\

$=\int_{\gamma}dU+PdV+VdP$\\

$=\int_{\gamma}dU+dL+VdP$\\

$=\Delta U+\Delta L+\int_{\gamma}VdP$\\

$\simeq \Delta U+\Delta L+V_{0}\int_{\gamma}dP$\\

$=\Delta U+\Delta L$\\

$=\Delta Q$\\

If $\Delta H$ is large, this would mean that $\Delta Q\neq 0$, which means that there is a change in heat, contradicting the fact that the temperature $T_{0}$ is unchanged. If $\Delta H$ is small, with $\Delta H=\int_{\gamma}VdP$, then, by generic considerations of bond energies, the amount of substance $\xi$ formed by the reaction would also be quite small. If $(*)$ holds for a single pair $\{A,B\}$, then it holds for all pairs $\{C,D\}$, as we can compose with reactions from $A$ to $C$ and $D$ to $B$. If follows that if we define;\\

$f_{i}(T,P)=e^{\mu_{i}(T,P)-\mu_{i}^{\circ}(T)\over RT}>0$, for $1\leq i\leq c$\\

so that $f_{i}(T,P)=x_{i}(T,P)={n_{i}(T,P)\over n(T,P)}$, then, by the definition of extent;\\

$\xi={n_{i}-n_{i,0}\over \nu_{i}}$\\

$n_{i}=\nu_{i}\xi+n_{i,0}$ $(1\leq i\leq c)$\\

we must have that;\\

$x_{i}={n_{i}\over n}={\nu_{i}\xi+n_{i,0}\over \sum_{i=1}^{c}(\nu_{i}\xi+n_{i,0})}={\nu_{i}\xi+n_{i,0}\over \lambda\xi+n_{0}}=f_{i}$ $(1\leq i\leq c)$\\

where $\lambda=\sum_{i=1}^{c}\nu_{i}$, so that;\\

$\nu_{i}\xi+n_{i,0}=(\lambda\xi+n_{0})f_{i}$\\

$\xi={n_{i,0}-n_{0}f_{i}\over \kappa f_{i}-\nu_{i}}$, $(1\leq i\leq c)$\\

so that, for $1\leq i\leq j\leq c$;\\

${n_{i,0}-n_{0}f_{i}\over \kappa f_{i}-\nu_{i}}={n_{j,0}-n_{0}f_{i}\over \kappa f_{j}-\nu_{i}}=\xi$ $(**)$\\

As all the steps are reversible, the requirement $(**)$ at $(T,P)$, for all $(n_{1,0},\ldots, n_{c,0})$ satisfying
${n_{i,0}\over n_{0}}=f_{i}(T_{0},P_{0})$, (so that $\xi(T_{0},P_{0})=0$, $(***)$), is equivalent to $(*)$. We impose the condition that $Ker(M)\cap \mathcal{R}_{>0}^{c}\neq \emptyset$, where;\\

$M_{ii}=f_{i}(T_{0},P_{0})-1$, $(1\leq i\leq c)$\\

$M_{ij}=f_{i}(T_{0},P_{0})$ $(1\leq i<j\leq c)$\\

so that there exists at least one choice $(n_{1,0},\ldots,n_{c,0})\in \mathcal{R}_{>0}^{c}$ satisfying $(***)$. With this requirement there do exist feasible paths between any $2$ pairs $(A,B)$, see Lemma \ref{exists}. If not, there is no feasible path involving a reaction from $A=(T_{0},P_{0})$, which seems physically unreasonable.\\

The same arguments apply when we incorporate error terms into the functions $\{f_{i}:1\leq i\leq p\}$ or extend the functions to a set $\{f_{i}:0\leq i\leq p\}$, when we consider a solvent, which we do in later sections.\\
\end{section}

\end{document}